\documentclass[11pt]{article}%
\usepackage{amsmath}%
\usepackage{amsfonts}%
\usepackage{amssymb}%
\usepackage{graphicx}
\usepackage{color}
\usepackage{xypic}
\usepackage[pdftex]{hyperref}
\hypersetup{citecolor=red,urlcolor=red,linkcolor=red}
\usepackage{ifthen}
\newcounter{rem}
\newenvironment{remark}{%
  \refstepcounter{rem}%
  \par\medskip\noindent%
  \textbf{Remark \therem\ ---}\hspace{1ex}}%
{\hfill$\lrcorner$\par\medskip}
\newcounter{lem}
\newenvironment{lemma}{%
  \refstepcounter{lem}%
  \par\medskip\noindent%
  \textbf{Lemma \thelem\ ---}\hspace{1ex}}%
{\hfill$\lrcorner$\par\medskip}
\newcounter{que}
{\hfill$\lrcorner$\par\medskip}
\newcounter{hyp}
{\hfill$\lrcorner$\par\medskip}
\newcounter{def}
\newenvironment{definition}{%
  \refstepcounter{def}%
  \par\medskip\noindent%
  \textbf{Definition \thedef\ ---}\hspace{1ex}}%
{\hfill$\lrcorner$\par\medskip}
\newcounter{exm}
\newcounter{xm}
\newenvironment{example}{%
  \refstepcounter{exm}%
  \setcounter{xm}{\theexm}%
  \par\medskip\noindent%
  \textbf{Example \theexm\ ---}\hspace{1ex}}%
{\hfill $-$Ex~\theexm$\dashv$\par\medskip}
{\hfill$\lrcorner$\par\medskip}
\renewcommand{\square}{{\rule{0.5em}{0.5em}}}
\newenvironment{proof}[1][Proof]%
{\noindent\textbf{#1.\ }}%
{\hfill \square\par\medskip}
\newtheorem{theorem}{Theorem}
\newtheorem{proposition}[theorem]{Proposition}
\newtheorem{corollary}[theorem]{Corollary}
\newtheorem{notation}[theorem]{Notation}
\DeclareMathOperator{\Ass}{\mathtt{Ass}}
\DeclareMathOperator{\ord}{ord}
\DeclareMathOperator{\card}{card}
\DeclareMathOperator{\Dom}{Dom}
\DeclareMathOperator{\discr}{discr}
\newcommand{\DifferentialAlgebraicSystem}{\textsc{dae}}
\newcommand{\indexa}{i}
\newcommand{\indexb}{j}
\newcommand{\indexc}{h}
\newcommand{\TransposedMatrix}[1]{{}^{t}#1}

\newcommand{\DifferentialSystem}{\Sigma}
\newcommand{\TheEquation}[1]{f_{#1}}
\newcommand{\EquationSet}{F}
\newcommand{\DifferentialSystemOrder}[1]%
{%
\ifx#1\empty%
{e}%
\else%
{e_{#1}}%
\fi%
}
\newcommand{\EquationSetCardinal}{n}
\newcommand{\DifferentialSystemSize}{\EquationSetCardinal}
\newcommand{\DifferentiationIndex}{\sigma}
\newcommand{\AnalyticSolutionSet}{\varphi}
\newcommand{\AnalyticSolution}[1]{{\varphi_{#1}}}
\newcommand{\SolutionSpace}{\Variety{S}}
\newcommand{\Output}[1]{\underline{#1}}
\newcommand{\AuxiliaryDifferentialSystemOrder}[1]%
{%
\ifx#1\empty%
{e}%
\else%
{e_{#1}}%
\fi%
}
\newcommand{\FirstOrderDifferentialSystem}{\Gamma}
\newcommand{\StateVariable}[1]{{x_{#1}}}
\newcommand{\StateVariableSet}{X}
\newcommand{\StateVariableSubset}{Y}
\newcommand{\StateVariableSubsetElement}[1]{{y_{#1}}}
\newcommand{\StateVariableSetCardinal}{n}
\newcommand{\AlgebraicIndeterminate}[1]{{y_{#1}}}
\newcommand{\AlgebraicIndeterminatesSet}{Y}
\newcommand{\AlgebraicIndeterminateNumber}{m}
\newcommand{\DifferentialIndeterminate}[1]{{\StateVariable{#1}}}
\newcommand{\DifferentialIndeterminatesSet}{\StateVariableSet}
\newcommand{\DifferentialIndeterminateNumber}{\StateVariableSetCardinal}
\newcommand{\NewVariableSet}{Z}
\newcommand{\NewVariable}[1]{{z_{#1}}}
\newcommand{\DifferentialPolynomial}[1]{{q_{#1}}}
\newcommand{\DifferentialPolynomialSet}{Q}
\newcommand{\DifferentialPolynomialNumber}{\nu}
\newcommand{\DifferentialPolynomialRing}[2]{#1\lbrace #2 \rbrace}
\newcommand{\DifferentialIdeal}[1]
{%
\ifx#1\empty%
{[\textrm{Id}]}%
\else%
{[#1]}%
\fi%
}
\newcommand{\DifferentialIdealOrder}[1]{{\ord\DifferentialIdeal{#1}}}
\newcommand{\Polynomial}[1]{{p_{#1}}}

\newcommand{\PolynomialNumber}{\rho}
\newcommand{\LeadingCoefficient}[1]{\mathrm{lc}#1}
\newcommand{\PolynomialRing}[2]{#1[#2]}

\newcommand{\Ideal}[1]%
{%
\ifx#1\empty%
I
\else%
{\mathrm{Id}\!\left(#1\right)}%
\fi%
}
\newcommand{\PrimeIdeal}[1]%
{%
\ifx#1\empty%
{\mathfrak{p}}%
\else%
{\mathrm{Ip}\!\left(#1\right)}%
\fi%
}
\newcommand{\AssociatedPrime}[1]{{\Ass{#1}}}

\newcommand{\IdealSet}{\mathtt{A}}
\newcommand{\ProlongatedRing}[1]{\mathsf{R}^{(#1)}}
\newcommand{\ProlongatedIdeal}[1]{\textsf{pr}^{(#1)}\EquationSet}
\newcommand{\PartiallyProlongatedIdeal}[1]{\mathsf{I}_{#1}}
\newcommand{\AffineSpace}[1]{\Field{A}^{#1}}
\newcommand{\Variety}[1]{\mathcal{#1}}
\newcommand{\AffinePoint}[1]{\Variety{#1}}
\newcommand{\ZariskiOpenSet}{\Variety{O}}
\newcommand{\AuxiliaryZariskiOpenSet}{\Variety{Q}}
\newcommand{\Time}{t}
\newcommand{\OpenInterval}[2]{(#1,#2)}
\newcommand{\fluxion}[2]%
{%
\ifthenelse{\equal{#2}{1}}%
{\dot{#1}}%
{{{#1}^{(#2)}}}%
}
\newcommand{\leibnizoperator}%
{\frac{\textup{d}\hfill}{\textup{d}\Time}}%
\newcommand{\tleibnizoperator}%
{\tfrac{\textup{d}\hfill}{\textup{d}\Time}}%

\newcommand{\leibniznotation}[2]%
{%
\ifthenelse{\equal{#2}{1}}%
{\frac{{\textup{d}{#1}\hfill}}{{\textup{d}\Time}\hfill}}%
{\frac{\textup{d}^{#2}{#1}}{\textup{d}\Time^{#2}}}%
}
\newcommand{\lineleibniznotation}[2]%
{%
\ifthenelse{\equal{#2}{1}}%
{{\textup{d}{#1}}/{\textup{d}\Time}}%
{{\textup{d}^{#2}{#1}}/{\textup{d}\Time^{#2}}}%
}

\newif\ifUseLeibnizNotation
\newcommand{\timederivative}[2]%
{%
\ifUseLeibnizNotation%
\leibniznotation{#1}{#2}%
\else%
\fluxion{#1}{#2}%
\fi%
}%
\newcommand{\ltimederivative}[2]%
{%
\ifUseLeibnizNotation%
\lineleibniznotation{#1}{#2}%
\else%
\fluxion{#1}{#2}%
\fi%
}%
\newcommand{\alltimederivative}[2]{{#1}^{[{#2}]}}
\newcommand{\BaseFieldDerivation}{\delta}
\newcommand{\NaturalNumbers}{\mathbb{N}_{0}}
\newcommand{\NaturalNumbersWithoutZero}{\mathbb{N}}
\newcommand{\Field}[1]{\mathbb{#1}}
\newcommand{\BaseField}{\Field{K}}
\newcommand{\AlgebraicClosure}[1]{\overline{#1}}

\newcommand{\TranscendenceBasis}{U}
\newcommand{\TranscendenceBasisElement}[1]{{u_{#1}}}

\newcommand{\PrimitiveElement}{\upsilon}
\newcommand{\PrimitiveElementMinimalPolynomial}{q}
\newcommand{\GeometricResolutionFiber}[1]{{p_{#1}}}
\newcommand{\LinearAlgebraConstant}{\omega}
\newcommand{\EvaluationComplexity}{\mathbf{L}}
\newcommand{\SpecializationVariableSet}{W}

\newcommand{\SpecializationVariableSetCardinal}{s}
\newcommand{\SpecializationPoint}{\AffinePoint{\SpecializationVariableSet}}
\newcommand{\SpecializationOf}[2]{{{#1}\!\left|_{#2}\right.}}
\newcommand{\SpecializedProlongatedIdeal}{\SpecializationOf{\ProlongatedIdeal{\DifferentiationIndex+1}}{\SpecializationPoint}}
\newcommand{\StraightLineProgramLenght}{L}
\newcommand\PartialDerivative[3]
{\ifthenelse{\equal{#3}{1}}%
{\frac{\partial\, #1\hfill}{{\partial #2}\hfill}}%
{\frac{\partial^{#3}\, #1\hfill}{{\partial #2}^{#3}\hfill}}%
}
\newcommand\tPartialDerivative[3]
{\ifthenelse{\equal{#3}{1}}%
{\tfrac{\partial\, #1\hfill}{{\partial #2}\hfill}}%
{\tfrac{\partial^{#3}\, #1\hfill}{{\partial #2}^{#3}\hfill}}%
}
 
\definecolor{FOdarkblue}{rgb}{0.,0.,0.8}
\definecolor{FOdarkred}{rgb}{0.7,0.,0.}

\def\ord{{\rm ord}}

\def\DifferentialAlgebraicSystem{DAE}
\def\rank{\mathrm{rank}}
\title{A Geometric Index Reduction Method for Implicit Systems of
  Differential Algebraic Equations}

\author{L.\ D'Alfonso$^{(\text{a})}$\thanks{Partially supported by UBACyT X211
    (2008-2010) and ANPCyT PICT2007-816.},
  G.\ Jeronimo$^{(\text{b})*}$\thanks{Partially supported by CONICET
    PIP 5852/05, UBACyT X847 (2006-2009).},
  F.\ Ollivier$^{(\text{c})}$, A.\ Sedoglavic$^{(\text{d})}$,
  P.\ Solern\'o$^{(\text{b})*}$
  \\[0.3cm]
  {\text{\normalsize (a) Departamento de Ciencias Exactas, Ciclo B\'asico Com\'un,}}\\
  {\text{\normalsize Universidad de Buenos Aires, Ciudad
      Universitaria, 1428
      Buenos Aires, Argentina}}\\[0.3cm]
  {\text{\normalsize (b) Departamento de Matem\'atica, Facultad de
      Ciencias
      Exactas y Naturales}}\\
  {\text{\normalsize Universidad de Buenos Aires, Ciudad
      Universitaria, 1428 Buenos Aires, Argentina}}\\{\normalsize and
    CONICET, \text{Argentina}}
  \\[0.2cm]
  {\text{\normalsize (c) LIX, UMR CNRS -- \'Ecole
      polytechnique {n$^{o}$~7161},
      F-91128 Palaiseau, France}}\\[0.3cm]
  {\text{\normalsize (d) LIFL,}}
  {\text{\normalsize UMR CNRS -- Universit\'e de Lille I {n$^{o}$~8022}, F-59655 Villeneuve d'Ascq, France}}\\[0.3cm]
  {\normalsize E-mails: lisi@cbc.uba.ar, jeronimo@dm.uba.ar, francois.ollivier@lix.polytechnique.fr,}\\
  {\normalsize Alexandre.Sedoglavic@univ-lille1.fr,
    psolerno@dm.uba.ar}}
\begin{document}
\maketitle
\begin{abstract} This paper deals with the index reduction problem for the class of quasi-regular \DifferentialAlgebraicSystem\
systems. It is shown that any of these systems
can be transformed to a generically equivalent first order \DifferentialAlgebraicSystem\
system  consisting of a single purely algebraic
(polynomial) equation plus an under-determined ODE (that is, a semi-explicit \DifferentialAlgebraicSystem\ system of differentiation
index~$1$) in as many variables as the order of the input system. This can be done by means of a Kronecker-type algorithm with bounded complexity.
\end{abstract}

\section{Introduction}
In this paper, we consider implicit, ordinary differential algebraic
equation (\DifferentialAlgebraicSystem ) systems
\begin{equation}
  \label{sistema}
  \DifferentialSystem:= \left\{
    \begin{array}
      [c]{ccl}%
      \TheEquation{1}\Bigl(
      \StateVariableSet, \timederivative{\StateVariableSet}{1},\ldots ,
      \timederivative{\StateVariableSet}{\DifferentialSystemOrder{1}}
      \Bigr) &=&0,\\
      & \vdots & \\
      \TheEquation{\EquationSetCardinal}\Bigl(
      \StateVariableSet, \timederivative{\StateVariableSet}{1},\ldots,
      \timederivative{\StateVariableSet}{\DifferentialSystemOrder{\EquationSetCardinal}}
      \Bigr) &=&0,  \\
    \end{array}
  \right.
\end{equation}
where, for any integer~$\indexa$ s.t.~${1\le \indexa \le
  \StateVariableSetCardinal}$, $\TheEquation{\indexa}$ is a polynomial
in the variables~${\StateVariableSet:=\{\StateVariable{1},\ldots,
  \StateVariable{\StateVariableSetCardinal}\}}$ and in
their~$\indexb$th (${1\le \indexb\le
  \DifferentialSystemOrder{\indexa}}$) time
derivatives~${\timederivative{\StateVariableSet}{\indexb} := \left\{
    \timederivative{\StateVariable{1}}{\indexb},
    \ldots,\timederivative{\StateVariable{\StateVariableSetCardinal}}{\indexb}\right\}}$,
with coefficients in a differential field~$\BaseField$ of
characteristic~$0$.
\par
% The \emph{differentiation index} is one of the most studied
% invariants of DAE systems and it represents a measure of the
% implicitness of the system (see \cite{brenan}, \cite{campbell:1995},
% \cite{fliess:1995}, \cite{kunkel:2006}, \cite{levey:1994},
% \cite{pantelides:1988}, \cite{poulsen}, \cite{pritchard:2007},
% \cite{rabier:1994}, \cite{reid:2001}, \cite{seiler:1999} for
% different definitions of indices).
One of the main invariants of \DifferentialAlgebraicSystem\
systems is their \emph{differentiation index}. There are several
definitions of differentiation indices not all completely equivalent
(see for instance~\cite{brenan:1996, campbell:1995, fliess:1995,
  kunkel:2006, levey:1994, pantelides:1988, poulsen:2001,
  pritchard:2007, pryce:2001, rabier:1994, reid:2001, seiler:1999}), but in every
case it represents a measure of the implicitness of the given system
in a fixed coordinates set. For instance, for first order equations,
differentiation indices provide bounds for the number of total
derivatives of the system needed in order to obtain in the same
set of coordinates an explicit \textsc{ode} system which is verified by
all the solutions of the original system (see~\cite[Definition
2.2.2]{brenan:1996}).
\par
Since explicitness is strongly related to the existence of classical
solutions, a differentiation index should also bound the number of
derivatives needed in order to obtain existence and uniqueness
theorems (see~\cite{pritchard:2003, pritchard:2007, rabier:1994}).
{}From the point of view of numerical resolution, it is desirable for
the \DifferentialAlgebraicSystem\ to have an index as small as
possible. As shown in~\cite[\S 2.5.3]{brenan:1996}, for first order
systems a reduction of the index can be achieved by differentiating
the algebraic constraints, but the numerical solution of the resulting system do not satisfy necessarily the original equations.
\paragraph{Main contributions.}
In this article we address the index reduction problem for an 
ubiquitous class of quasi-regular \DifferentialAlgebraicSystem\
systems (see Section~\ref{sec:quasi-regular}). We show that any of these systems
is generically equivalent to a related (in a non intrinsic way) first
order \DifferentialAlgebraicSystem\
system~$\Output{\DifferentialSystem}$ with a particular
structure. This new system consists of a single purely algebraic
(polynomial) equation plus an under-determined \textsc{ode} (see
Definition~\ref{sigma1} page~\pageref{sigma1}).
Indeed,~$\Output{\DifferentialSystem}$ is a \emph{semi-explicit
  \DifferentialAlgebraicSystem\ system} in the usual sense (see for
instance~\cite[Section 1.2]{brenan:1996}) with differentiation
index~$1$ (see Proposition~\ref{index1} page~\pageref{index1}). It is
a well-known fact that this class of systems can be handled
successfully by means of numerical methods (see~\cite{petzold:1986,
  lotstedt:1986, brenan:1983}).
\par
The index reduction problem has already been considered in several
previous articles (see for instance~\cite{gear:1988, gear:1989,
  kunkel:2006, mattsson:1993}). The techniques applied in these
works are based on the computation of sufficiently many successive
derivatives of the original equations combined with rewriting
procedures relying on the Implicit Function Theorem, elimination of
critical equations, introduction of dummy derivatives, etc.
\par
Our approach also makes use of the computation of successive
derivatives, as many as the index, but, unlike the methods mentioned above,
we deal with the system of all these new equations in a purely
algebraic way. This new system defines an algebraic variety in a
suitable jet space and we parametrize this variety by means of the
points of a hypersurface.  This construction, originally introduced by
Kronecker, is known as a \emph{geometric resolution}
(see~\cite{giusti:2001,schost:2003,DuLe:2006:cpkpsss} and references therein). In order
to keep track of the differential structure, we use the
parametrizations to construct a vector field over the hypersurface
defining the semi-explicit \DifferentialAlgebraicSystem\
system~$\Output{\DifferentialSystem}$. A result of the same flavor (i.e. an univariate differential equation plus parameterizations of the variables) may be given by means of the notion of primitive element of an extension of differential fields. This construction, due to J. Ritt (\cite{ritt:1932}, see also \cite{seidenberg:1952}), is known as a \emph{resolvent representation} of the system $\Sigma$ (see \cite{CluzeauH:2003,CluzeauH:2008,dalfonso:2006} for effective versions of it).
\par
With respect to the known index reduction approaches, our method is symbolic and, in some sense, automatic: it
does not make use of the Implicit Function Theorem as it is the case
in~\cite{gear:1988} or~\cite{gear:1989} and it does not rely on any
smart choice of ad-hoc equations as in~\cite{kunkel:2006}. Moreover,
the construction can be done algorithmically within an admissible
complexity by applying well-known techniques from computer algebra
(see~\cite{schost:2003,lecerf:2003}).
\par
The number of variables of our semi-explicit system is the order of
the \emph{differential ideal} associated to the input system plus one
and it is always lower than those involved in the index reduction
methods of previous papers. In the first order case, where it is easy
to compare, this number is at most~${\StateVariableSetCardinal+1}$ and
in the general case, it is bounded by the Jacobi number of the system
(see~\cite{jacobi:1865, kondratieva:1982, ollivier:2007}
and~\cite{dalfonso:2009}).
\par
A further advantage of our method is that it preserves the constraints
of the initial conditions of the original system and then we do
not need to compute constants of integration.
\paragraph{Outline.}
The paper is organized as follows: in Section~\ref{preliminaries} the
basic notions needed throughout the article are introduced. The core
of the paper is in Section~\ref{hypersurface} where the semi-explicit
system~$\Output{\DifferentialSystem}$ is constructed. In
Section~\ref{passing} we study the relation between the solutions of
both systems. Finally, two appendices are included: the first one
contains some Bertini type results from commutative algebra we need and the second
one is devoted to existence and uniqueness theorems for
\DifferentialAlgebraicSystem\ systems.
\section{Preliminaries}
\label{preliminaries}
In this section, we introduce some notations used throughout this paper
and we recall some basic definitions from elementary (differential)
algebraic geometry for the reader convenience. Furthermore, we discuss the assumptions on the systems considered and some results concerning the differentiation index and the order of these systems.

\subsection{Basic notions and notations}
\label{basic}
Let~$\BaseField$ be a characteristic zero field equipped with a
derivation~$\BaseFieldDerivation$.  For
instance~${\BaseField=\Field{Q},\Field{R}}$ or~$\Field{C}$
with~${\BaseFieldDerivation=0}$, or~${\BaseField=\Field{Q}(t)}$ with
the usual derivation~${\BaseFieldDerivation(t)=1}$, etc.
\par
As in the Introduction, for any set
of $\DifferentialIndeterminatesSet:={\left\{ \DifferentialIndeterminate{1},\ldots
    ,\DifferentialIndeterminate{\DifferentialIndeterminateNumber}\right\}}$ of $\DifferentialIndeterminateNumber$ (differential) indeterminates
over~$\BaseField$, we
denote
by~$\timederivative{\DifferentialIndeterminate{\indexa}}{\indexb}$
the~$\indexb$-th successive formal derivative of the
variable~$\DifferentialIndeterminate{\indexa}$ (following Newton's
notation, its first derivative is also denoted
by~$\timederivative{\DifferentialIndeterminate{\indexa}}{1}$) and we
use the following notations:
\[ {\timederivative{\DifferentialIndeterminatesSet}{\indexb}
  :=\left\{\timederivative{\DifferentialIndeterminate{1}}{\indexb},\ldots,
    \timederivative{\DifferentialIndeterminate{\DifferentialIndeterminateNumber}}{\indexb}
  \right\}}\qquad
\textrm{and}\qquad{\alltimederivative{\DifferentialIndeterminatesSet}{\indexb}
  :=\left\{ \DifferentialIndeterminatesSet,
    \timederivative{\DifferentialIndeterminatesSet}{1},
    \timederivative{\DifferentialIndeterminatesSet}{2},\ldots,
    \timederivative{\DifferentialIndeterminatesSet}{\indexb}
  \right\}}.
\]
The derivation~$\BaseFieldDerivation$ can be extended to a derivation
in the polynomial
ring~${\BaseField\!\left[\timederivative{\DifferentialIndeterminatesSet}{\indexb},\,
    \indexb \in \NaturalNumbers\right]}$ as follows: for any
differential polynomial~$\DifferentialPolynomial{\empty}$
in~${\BaseField\!\left[\timederivative{\DifferentialIndeterminatesSet}{\indexb},
    \, \indexb\in \NaturalNumbers\right]}$ the following classical
recursive relations hold for the successive total derivatives
of~$\DifferentialPolynomial{\empty}$:
\[
% \begin{array}{rcl}
\timederivative{\DifferentialPolynomial{\empty}}{0}:=
\DifferentialPolynomial{\empty}, %\\
\qquad \timederivative{\DifferentialPolynomial{\empty}}{\indexb}:=
\BaseFieldDerivation\!\left \lvert
  \timederivative{\DifferentialPolynomial{\empty}}{\indexb-1}\right
\rvert +\displaystyle{%
  \sum_{\indexc\in\NaturalNumbers} \sum_{1\le \indexa\le
    \DifferentialIndeterminateNumber} \dfrac{\partial
    \timederivative{\DifferentialPolynomial{\empty}}{\indexb-1}}
  {\partial\timederivative{\DifferentialIndeterminate{\indexa}}{\indexc}}
  \timederivative{\DifferentialIndeterminate{\indexa}}{\indexc+1} },
\qquad \textrm{ for } \indexb\ge 1,
%\end{array}
\]
where~$\BaseFieldDerivation\bigl
\lvert\timederivative{\DifferentialPolynomial{\empty}}{\indexb-1}\bigr
\rvert$ denotes the polynomial obtained
from~$\timederivative{\DifferentialPolynomial{\empty}}{\indexb-1}$ by
applying the derivative~$\BaseFieldDerivation$ to all its
coefficients.  The (non-Noetherian) polynomial
ring~${\BaseField\!\left[\timederivative{\DifferentialIndeterminatesSet}{\indexb},\
    \indexb \in \NaturalNumbers\right]}$ with this derivation is denoted
by~$\DifferentialPolynomialRing{\BaseField}{\DifferentialIndeterminate{1},\ldots
  ,\DifferentialIndeterminate{\DifferentialIndeterminateNumber}}$ (or
simply~$\DifferentialPolynomialRing{\BaseField}{\DifferentialIndeterminatesSet}$)
and is called the ring of \emph{differential polynomials}.
\par
Given a finite set of (differential)
polynomials~${\DifferentialPolynomialSet:=
  \{\DifferentialPolynomial{1}, \ldots,
  \DifferentialPolynomial{\DifferentialPolynomialNumber}}\}$
in~$\DifferentialPolynomialRing{\BaseField}{\DifferentialIndeterminatesSet}$,
we write~$\DifferentialIdeal{\DifferentialPolynomialSet}$ to denote
the smallest ideal of~$\BaseField\{\DifferentialIndeterminatesSet\}$
stable under differentiation, i.e.\ the smallest ideal
containing~${\DifferentialPolynomial{1},\ldots,
  \DifferentialPolynomial{\DifferentialPolynomialNumber}}$ and all
their derivatives of arbitrary order. The
ideal~$\DifferentialIdeal{\DifferentialPolynomialSet}$ is called the
\emph{differential ideal} generated
by~$\DifferentialPolynomialSet$. Furthermore, for every
integer~$\indexb$, we extend our previous notations as follows~:
\[ {\timederivative{\DifferentialPolynomialSet}{\indexb}:=\left\{
    \timederivative{\DifferentialPolynomial{1}}{\indexb},\ldots,
    \timederivative{\DifferentialPolynomial{\DifferentialPolynomialNumber}}{\indexb}
  \right\}}
\qquad\textrm{and}\qquad
{\alltimederivative{\DifferentialPolynomialSet}{\indexb}:=\left\{\DifferentialPolynomialSet,
    \timederivative{\DifferentialPolynomialSet}{1},
    \timederivative{\DifferentialPolynomialSet}{2}, \ldots ,
    \timederivative{\DifferentialPolynomialSet}{\indexb}\right\}\!.}
\]
Let us introduce also some notions concerning elementary algebraic
geometry.
\par
Let~${\AlgebraicIndeterminatesSet:=\left\lbrace
    \AlgebraicIndeterminate{1}, \ldots,
    \AlgebraicIndeterminate{\AlgebraicIndeterminateNumber}\right\rbrace}$
be (algebraic) indeterminates over the field~$\BaseField$; we
write~$\PolynomialRing{\BaseField}{\AlgebraicIndeterminatesSet}$ to
denote the polynomial ring in~$\AlgebraicIndeterminateNumber$
variables over~$\BaseField$. Let~$\AlgebraicClosure{\BaseField}$ be a
fixed algebraic closure of~$\BaseField$.  Given some
polynomials~${\Polynomial{1}, \ldots, \Polynomial{\PolynomialNumber}}$
in~$\PolynomialRing{\BaseField}{\AlgebraicIndeterminatesSet}$, the
set~${\bigl\{ \AffinePoint{\AlgebraicIndeterminatesSet} \in
  {\AlgebraicClosure{\BaseField}}^{\AlgebraicIndeterminateNumber},\
  \Polynomial{1}(\AffinePoint{\AlgebraicIndeterminatesSet} )=\cdots=
  \Polynomial{\PolynomialNumber}(\AffinePoint{\AlgebraicIndeterminatesSet}
  )=0 \bigr\}}$ is called an \textit{algebraic variety definable
  over~$\BaseField$} (or simply a \textit{variety} if~$\BaseField$ is
clear from the context). The affine
space~$\AlgebraicClosure{\BaseField}^{\AlgebraicIndeterminateNumber}$
is endowed with a topology (the so-called \textit{Zariski topology})
where the closed sets are exactly the algebraic varieties definable
over~$\BaseField$. We denote this topological space
by~$\AffineSpace{\AlgebraicIndeterminateNumber}$. The
space~$\AffineSpace{\AlgebraicIndeterminateNumber}$ is a Noetherian
space and then every closed set is an irredundant union of a finite
number of irreducible closed sets.
\par
Given a subvariety~$\Variety{V}$
of~$\AffineSpace{\AlgebraicIndeterminateNumber}$ we denote
by~$I({\Variety{V}})$ the ideal
in~$\PolynomialRing{\BaseField}{\AlgebraicIndeterminatesSet}$ of all
the polynomials that vanish on~$\Variety{V}$
and by~$\PolynomialRing{\BaseField}{\Variety{V}}:={\PolynomialRing{\BaseField}{\AlgebraicIndeterminatesSet}/I({\Variety{V}})}$
the \textit{coordinate ring} of~$\Variety{V}$.
\subsection{The considered system --- Primality assumption}
\label{sistema_original}
In this section we recall some notations concerning the
\DifferentialAlgebraicSystem\ systems considered in this paper. Then,
we explicit a natural primality assumption necessary in the sequel.
\par
Let~$\EquationSetCardinal$ denote a fixed non-negative integer.
Throughout the paper we consider \DifferentialAlgebraicSystem\ systems
of the following type:
\[
\DifferentialSystem:= \left\{
    \begin{array}
      [c]{ccl}%
      \TheEquation{1}\Bigl(
      \StateVariableSet, \timederivative{\StateVariableSet}{1},\ldots ,
      \timederivative{\StateVariableSet}{\DifferentialSystemOrder{1}}
      \Bigr) &=&0,\\
      & \vdots & \\
      \TheEquation{\EquationSetCardinal}\Bigl(
      \StateVariableSet, \timederivative{\StateVariableSet}{1},\ldots,
      \timederivative{\StateVariableSet}{\DifferentialSystemOrder{\EquationSetCardinal}}
      \Bigr) &=&0,  \\
    \end{array}
  \right.
\]
where, for every~${1\le \indexa \le \EquationSetCardinal}$,
$\TheEquation{\indexa}$ is a polynomial in the
variables~$\StateVariableSet$ and the
derivatives~${\timederivative{\StateVariableSet}{\indexb}}$,
with~${1\le \indexb \le \DifferentialSystemOrder{\indexa}}$; the
coefficients of these polynomials are in the field~$\BaseField$. Each
non-negative integer~$\DifferentialSystemOrder{\indexa}$ denotes the
maximal derivation order appearing in the
polynomial~$\TheEquation{\indexa}$.  We
write~${\DifferentialSystemOrder{\empty}:=\max\{\DifferentialSystemOrder{\indexa}\}}$
for the maximal derivation order that occurs in~$\DifferentialSystem$
and we assume that~${\DifferentialSystemOrder{\empty}}$ is greater or
equal to~$1$.  As done previously, we use the following notations:
\[
\EquationSet:=\{\TheEquation{1},\ldots,\TheEquation{\EquationSetCardinal}\},\quad
{\timederivative{\EquationSet}{\indexb}:=\left\lbrace\timederivative{\TheEquation{1}}{\indexb},\ldots,\timederivative{\TheEquation{n}}{\indexb}\right\rbrace}
\quad\textrm{and}\quad
{\alltimederivative{\EquationSet}{\indexb}:=\left\lbrace \EquationSet,
    \timederivative{\EquationSet}{1},
    \timederivative{\EquationSet}{2},\ldots
    ,\timederivative{\EquationSet}{\indexb}\right\rbrace}.
\]
Let~$\DifferentialIdeal{\EquationSet}\subset \DifferentialPolynomialRing{\BaseField}{\StateVariableSet}$ be the
differential ideal generated by the polynomials~$\EquationSet$. We
introduce also the following auxiliary (Noetherian) polynomial rings
and ideals: for every~$\indexb$
in~$\NaturalNumbers$,~$\ProlongatedRing{\indexb}$ denotes the
polynomial
ring~${\BaseField\!\left[\alltimederivative{\StateVariableSet}{\indexb}\right]}$
and~$\ProlongatedIdeal{\indexb}$ the ideal
in~$\ProlongatedRing{\indexb-1+\DifferentialSystemOrder{\empty}}$
generated by the total derivatives of the defining equations up to
order~${\indexb-1}$,
namely~${\ProlongatedIdeal{\indexb}:=(\alltimederivative{\EquationSet}{\indexb-1})}$
(this ideal is usually known as the \emph{${(\indexb-1)}$th
  prolongation ideal}). We set~${\ProlongatedIdeal{0}:=(0)}$ by
definition.
\par
For~${\indexa=0,\ldots,\EquationSetCardinal}$
in~$\NaturalNumbersWithoutZero$ and for every integer~$\indexb$, we
will assume that the ideals generated by the
polynomials~${\alltimederivative{\EquationSet}{\indexb-1}
  ,\timederivative{\TheEquation{1}}{\indexb},\ldots
  ,\timederivative{\TheEquation{\indexa}}{\indexb}}$ \emph{are all
  prime ideals} in their respective rings. In particular, the
differential ideal~$\DifferentialIdeal{\EquationSet}$ is a prime
differential ideal in the
ring~$\DifferentialPolynomialRing{\BaseField}{\StateVariableSet}$.
\subsection{Quasi-regular DAE systems and prime complete intersection}
\label{sec:quasi-regular}
In this section, we establish a relationship between the notion of
\emph{quasi-regularity} of a differential system and an algebraic
property---complete intersection---that is required by the geometric
elimination algorithm used in this paper.  The notion of
\emph{quasi-regularity} appears implicitly in~\cite{johnson:1978} in
order to generalize a conjecture of Janet to non-linear systems.
\begin{definition} \label{quasir} \label{quasirgen}
  Let~$\Gamma$ be a
  \DifferentialAlgebraicSystem\ system given in the
  ring~${\DifferentialPolynomialRing{\BaseField}{\DifferentialIndeterminatesSet}}$
  by differential
  polynomials~${\DifferentialPolynomialSet:=\{\DifferentialPolynomial{1},\ldots,\DifferentialPolynomial{\DifferentialPolynomialNumber}\}}$
   of order bounded by a nonegative
  integer~$\AuxiliaryDifferentialSystemOrder{\empty}$. Let~$\mathfrak{p}$
  be a prime differential ideal
  containing~$\DifferentialIdeal{\DifferentialPolynomialSet}$. We say
  that~$\Gamma$ is \emph{quasi-regular
    at~$\mathfrak{p}$} if for every integer~$\indexb$
  in~$\NaturalNumbers$, the Jacobian matrix of the
  polynomials~${\alltimederivative{\DifferentialPolynomialSet}{\indexb}}$
  with respect to the set of
  variables~$\alltimederivative{\DifferentialIndeterminatesSet}{\indexb+\AuxiliaryDifferentialSystemOrder{\empty}}$
  has full row rank over the
  domain~${\ProlongatedRing{\indexb+\DifferentialSystemOrder{\empty}}/(\ProlongatedRing{{\indexb+\DifferentialSystemOrder{\empty}}}\cap\,
    \mathfrak{p})}$. We say
  that~$\Gamma$ is \emph{quasi-regular} if it is
  quasi-regular at any \emph{minimal} prime differential
  ideal
  containing~$\DifferentialIdeal{\DifferentialPolynomialSet}$.
\end{definition}
For the systems~$\DifferentialSystem$ considered in this paper, since the
ideal~$\DifferentialIdeal{\EquationSet}$ is assumed to be prime, the
quasi-regularity of~$\DifferentialSystem$ is equivalent to say that
for each integer~$\indexb$, the Jacobian matrix of the
polynomials~$\alltimederivative{\EquationSet}{\indexb}$ with respect
to the set of
variables~$\alltimederivative{\StateVariableSet}{\indexb+\DifferentialSystemOrder{\empty}}$
has full row rank over the
domain~${\ProlongatedRing{\indexb+\DifferentialSystemOrder{\empty}}/(\ProlongatedRing{\indexb+\DifferentialSystemOrder{\empty}}\cap\,
  \DifferentialIdeal{\EquationSet})}$.  This condition can be easily
rephrased in terms of K\"ahler differentials (as in Johnson's original
work~\cite{johnson:1978}) saying that the set of
differentials~${\bigl\{ \textrm{d}
  \timederivative{\TheEquation{\indexa}}{\indexb},\ 1\le \indexa \le
  \EquationSetCardinal,\ \indexb \in \NaturalNumbers\bigr\}}$ are
a~${\DifferentialPolynomialRing{\BaseField}{\StateVariableSet}/\DifferentialIdeal{\EquationSet}}$-linearly
independent set in the module of
differentials~${\Omega_{\DifferentialPolynomialRing{\BaseField}{\StateVariableSet}/\BaseField}\otimes_{\DifferentialPolynomialRing{\BaseField}{\StateVariableSet}}
  \DifferentialPolynomialRing{\BaseField}{\StateVariableSet}/\DifferentialIdeal{\EquationSet}}$.
Geometrically, it means that
% for any non negative integer $i$
the algebraic variety defined by the ideal generated by
the~${(\indexb+1)\EquationSetCardinal}$
polynomials~$\alltimederivative{\EquationSet}{\indexb}$ in
the~${(\indexb+\DifferentialSystemOrder{\empty}+1)\StateVariableSetCardinal}$-variate
polynomial
ring~$\ProlongatedRing{\indexb+\DifferentialSystemOrder{\empty}}$
\emph{is smooth at almost every point} of the closed subvariety
defined by the prime
ideal~${\ProlongatedRing{\indexb+\DifferentialSystemOrder{\empty}}\cap
  \DifferentialIdeal{\EquationSet}}$.
\par
Under our assumptions we have the following straightforward
consequence:
\begin{proposition}
  \label{prime}
  If the system~$\DifferentialSystem$ is quasi-regular, then
  for~${\indexa=0,\dots, \EquationSetCardinal}$ and every~$\indexb$
  in~$\NaturalNumbers$, the
  polynomials~${\alltimederivative{\EquationSet}{\indexb-1}
    ,\timederivative{\TheEquation{1}}{\indexb},\ldots ,
    \timederivative{\TheEquation{\indexa}}{\indexb}}$ form a regular
  sequence in the
  ring~${\ProlongatedRing{\indexb+\DifferentialSystemOrder{\empty}}}$.
  In particular, the prolongation ideals~$\ProlongatedIdeal{\indexb}$
  are prime complete intersection ideals.
\end{proposition}
\begin{proof}
  Let~$\PartiallyProlongatedIdeal{\indexb\indexa}$  be
  the ideal
  of~$\ProlongatedRing{\indexb+\DifferentialSystemOrder{\empty}}$ generated by~${\alltimederivative{\EquationSet}{\indexb-1}
    , \timederivative{\TheEquation{1}}{\indexb},\ldots ,
    \timederivative{\TheEquation{\indexa}}{\indexb}}$. From the
  Jacobian Criterion (see~\cite[\S30]{matsumura:1970}) we observe that
  the quasi-regularity condition implies that the
  ideal~$\PartiallyProlongatedIdeal{\indexb\indexa} $ is complete
  intersection and its generators form a regular sequence in the local
  ring~$\bigl(\ProlongatedRing{\indexb+\DifferentialSystemOrder{\empty}}\bigr)_\mathsf{q}$,
  where~$\mathsf{q}$
  denotes~${\ProlongatedRing{\indexb+\DifferentialSystemOrder{\empty}}\cap
    \DifferentialIdeal{\EquationSet}}$.  Since the
  ideal~$\PartiallyProlongatedIdeal{\indexb\indexa} $ is assumed to be
  prime and it is contained in~$\mathsf{q}$, we conclude that the
  polynomials~${\alltimederivative{\EquationSet}{\indexb-1}
    ,\timederivative{\TheEquation{1}}{\indexb},\ldots ,
    \timederivative{\TheEquation{\indexa}}{\indexb}}$ form a
  \emph{global} regular sequence
  in~$\ProlongatedRing{\indexb+\DifferentialSystemOrder{\empty}}$.
\end{proof}

\subsection{Differentiation index --- Linearized
  standpoint} \label{defindex}
We introduce here the notion of differentiation index of quasi-regular DAE systems used in this
paper and establish a relationship between this index and the dimension
of Jacobian matrix kernels. We keep the hypotheses on the system $\DifferentialSystem$ made on Section \ref{sistema_original} and, from now on, we also assume that \emph{$\DifferentialSystem$ is quasi-regular}.

Consider the following chain
$\mathcal{C}$ of (prime) ideals in the polynomial
ring~${\ProlongatedRing{\DifferentialSystemOrder{\empty}-1}}$:
\[ {\mathcal{C}}:\quad 0=\ProlongatedIdeal{0} \cap
\ProlongatedRing{\DifferentialSystemOrder{\empty}-1}\subseteq
\ProlongatedIdeal{1} \cap
\ProlongatedRing{\DifferentialSystemOrder{\empty}-1}\subseteq \cdots
\subseteq \ProlongatedIdeal{\indexb} \cap
\ProlongatedRing{\DifferentialSystemOrder{\empty}-1}\subseteq \cdots
\subseteq {[F]\cap}\ProlongatedRing{\DifferentialSystemOrder{\empty}-1}.
\]
Since~$\ProlongatedRing{\DifferentialSystemOrder{\empty}-1}$ is a
Noetherian ring, the ideal chain~$\mathcal{C}$ eventually becomes
stationary. Clearly, the biggest proper ideal of the chain must
be~${\DifferentialIdeal{\EquationSet}\cap
  \ProlongatedRing{\DifferentialSystemOrder{\empty}-1}}$.

\begin{definition}
  The \emph{differentiation index}~$\DifferentiationIndex$ of the
  system~$\DifferentialSystem$ is the minimum integer~$\indexb$ at
  which the chain~$\mathcal{C}$ becomes stationary; more precisely,
  \[
  {\DifferentiationIndex:=\min \left\{\indexb\in \NaturalNumbers\mid
      \ProlongatedIdeal{\indexb} \cap
      \ProlongatedRing{\DifferentialSystemOrder{\empty}-1} =
      \ProlongatedIdeal{\indexb+\indexc}\cap
      \ProlongatedRing{\DifferentialSystemOrder{\empty}-1},\ \forall
      \indexc\in \NaturalNumbersWithoutZero\right\}}.
  \]
  Clearly we have~${\DifferentiationIndex=\min \{\indexb\in
    \NaturalNumbers\mid \ProlongatedIdeal{\indexb}\cap
    \ProlongatedRing{\DifferentialSystemOrder{\empty}-1}=\DifferentialIdeal{\EquationSet}
    \cap \ProlongatedRing{\DifferentialSystemOrder{\empty}-1}\}}$.
\end{definition}
The differentiation index can also be defined by means of Jacobian
matrices related to the input system.
\par
For any positive integers~$\indexb$ and~$\indexc$, with~${\indexc \ge
  \DifferentialSystemOrder{\empty}-1}$,
let~$\mathfrak{J}_{\indexb,\indexc}$ be the Jacobian matrix
of~${\timederivative{\EquationSet}{\indexc-\DifferentialSystemOrder{\empty}
    + 1}, \ldots ,
  \timederivative{\EquationSet}{\indexc-\DifferentialSystemOrder{\empty}
    + \indexb}}$ with respect to the
variables~${\timederivative{\StateVariableSet}{\indexc+1},\ldots
  ,\timederivative{\StateVariableSet}{\indexc+\indexb}}$. Let the
integer~$d_{\indexb,\indexc}$ denote~${\dim (\ker
  (\TransposedMatrix{\mathfrak{J}_{\indexb,\indexc}}))}$
(where~$\TransposedMatrix{\mathfrak{J}_{\indexb,\indexc}}$ denotes the
transposed matrix of~${\mathfrak{J}_{\indexb,\indexc}}$) and let~${d_{0,\indexc}:=0}$.  Under the additional hypothesis that the
rank of the
matrices~$\TransposedMatrix{\mathfrak{J}_{\indexb,\indexc}}$ is
independent of the ring where it is computed (the rank over the
rings~${\ProlongatedRing{\indexb+\indexc+\indexa}/\ProlongatedIdeal{\indexb+\indexc+\indexa-\DifferentialSystemOrder{\empty}+1}}$
or
over~${\ProlongatedRing{\indexc+\indexb+\indexa}/(\DifferentialIdeal{\EquationSet}\cap
  \ProlongatedRing{\indexc+\indexb+\indexa}})$ is the same for any
integer~$\indexa$), it can be shown that the double
sequence~$d_{\indexb,\indexc}$ is in fact independent of~$\indexc$
(see \cite[Proposition 11]{dalfonso:2008}). If we
write~${d_{\indexb}:=d_{\indexb,\indexc}}$, we have the following
alternative characterization of the differentiation index~(see
\cite[Theorem 8 and Definition 9]{dalfonso:2009}):
$${\DifferentiationIndex=\min\{\indexb\in \NaturalNumbers \mid
      d_{\indexb}=d_{\indexb+1}\}}.$$
{}From this characterization, we deduce the following result (see
\cite[Theorem 10]{dalfonso:2009}):
\begin{theorem}
\label{index}
The differentiation index~$\DifferentiationIndex$ satisfies:
\begin{eqnarray*}
\DifferentiationIndex&=& \min\{ \indexb\in \NaturalNumbers \mid
    \ProlongatedIdeal{\indexb+\indexc-\DifferentialSystemOrder{\empty}+1}
    \cap \ProlongatedRing{\indexc} = \ProlongatedIdeal{\indexb+\indexc-\DifferentialSystemOrder{\empty}+2}    % \DifferentialIdeal{\EquationSet}
    \cap \ProlongatedRing{\indexc}\}\\
    &=& \min\{ \indexb\in \NaturalNumbers \mid
    \ProlongatedIdeal{\indexb+\indexc-\DifferentialSystemOrder{\empty}+1}
    \cap \ProlongatedRing{\indexc} =  \DifferentialIdeal{\EquationSet}
    \cap \ProlongatedRing{\indexc}\}
    \end{eqnarray*}
for every integer~${\indexc \ge {\DifferentialSystemOrder{\empty}-1}}$.
\end{theorem}

The techniques used in \cite{dalfonso:2009} rely on the structure of the Jacobian matrices involved. Here we give an alternative proof of the above result for the case $\indexc=\DifferentialSystemOrder{\empty}$ based on the characteristic set theory (see \cite{kolchin:1973,mishra:1993}).

\begin{lemma}\label{stal-un-jour}
If, for some integer $\indexb$,
$\ProlongatedIdeal{\indexb}\cap\ProlongatedRing{\DifferentialSystemOrder{\empty}}=
\ProlongatedIdeal{\indexb+1}\cap\ProlongatedRing{\DifferentialSystemOrder{\empty}}$,
then $\ProlongatedIdeal{\indexb}\cap\ProlongatedRing{\DifferentialSystemOrder{\empty}}=
[\EquationSet]\cap\ProlongatedRing{\DifferentialSystemOrder{\empty}}$.
\end{lemma}

\begin{proof}
Let $\mathcal{A}$ be an algebraic characteristic set of the prime ideal $\ProlongatedIdeal{\indexb}\cap\ProlongatedRing{\DifferentialSystemOrder{\empty}}$
for some orderly ranking on derivatives. From $\mathcal{A}$ we extract a
minimal chain $\mathcal{B}$ as follows: from all the polynomials in $\mathcal{A}$ with the same leading variable we take the one with the minimal order of derivation in this variable. We claim that $\mathcal{B}$ is
autoreduced in the differential meaning.

This is equivalent to the fact that, if $x_{\indexa}^{(h)}$ is the leading
derivative of some element $B$ of $\mathcal{B}$, then this derivative
does not appear in some other element. As we use an orderly ranking,
the $e-h$ first derivatives of $B$ belong to
$\ProlongatedRing{\DifferentialSystemOrder{\empty}}$ and, since by assumption $\ProlongatedIdeal{\indexb+1}\cap\ProlongatedRing{\DifferentialSystemOrder{\empty}}=
\ProlongatedIdeal{\indexb}\cap\ProlongatedRing{\DifferentialSystemOrder{\empty}}$, they belong to  $\ProlongatedIdeal{\indexb}\cap\ProlongatedRing{\DifferentialSystemOrder{\empty}}$. Then, the
derivatives $x_{\indexa}^{(\ell)}$, $h<\ell\le e$, are the leading
derivatives of these elements of
$\ProlongatedIdeal{\indexb}\cap\ProlongatedRing{\DifferentialSystemOrder{\empty}}$.
These derivatives appear with degree $1$ and with initial equal to $S_{B}$, the
separant of $B$, that does not belong to $\ProlongatedIdeal{\indexb}\cap\ProlongatedRing{\DifferentialSystemOrder{\empty}}$. So they are the leading derivatives, with degree $1$, of some
elements of $\mathcal{A}$, and they do not appear in other elements of
this characteristic set. Hence our claim.

So, $\mathcal{B}$ is the characteristic set of some prime differential ideal $\
\mathcal{P}\subset[F]$ (see \cite{blop:2009}).
Now, it is easily seen that all polynomials in $\ProlongatedIdeal{\indexb}
\cap\ProlongatedRing{\DifferentialSystemOrder{\empty}}$ are reduced to
$0$ by $\mathcal{B}$, {which implies} that $F\subset\mathcal{P}$,
{so that} $[F]=\mathcal{P}$, {and also that $\ProlongatedIdeal{\indexb}
\cap\ProlongatedRing{\DifferentialSystemOrder{\empty}}=[F]
\cap\ProlongatedRing{\DifferentialSystemOrder{\empty}}$}.
\end{proof}

In the last part of this section, we recall in a geometric framework
the notion of initial conditions associated to a given differential
system.
\subsection{Hilbert-Kolchin regularity --- Independent
  variables} \label{HK_regularity}

Under our assumptions, the differential dimension of the prime
differential ideal~$\DifferentialIdeal{\EquationSet}$ is~$0$ (see e.g.~\cite{kondratieva:2009}). So,
following~\cite[Chapter II, Section 12, Theorem 6]{kolchin:1973}, the
transcendence degree of the fraction field of the
domain~${\ProlongatedRing{\indexb}/(\ProlongatedRing{\indexb}\cap
  \DifferentialIdeal{\EquationSet})}$ over the ground
field~$\BaseField$ becomes constant for all~$\indexa$ sufficiently
big. This constant is a non-negative integer called \emph{the order
  of~$\DifferentialIdeal{\EquationSet}$} and it is denoted
by~$\DifferentialIdealOrder{\EquationSet}$.
\par
The minimum of the indices~$\indexb_{0}$ such that the order
of~$\DifferentialIdeal{\EquationSet}$ equals the transcendence degree
of the fraction field
of~${\ProlongatedRing{\indexb}/(\ProlongatedRing{\indexb}\cap
  \DifferentialIdeal{\EquationSet}))}$ over~$\BaseField$ for
all~${\indexb \ge \indexb_{0}}$ is known as \emph{the Hilbert-Kolchin
  regularity of the ideal~$\DifferentialIdeal{ \EquationSet}$}.
In our situation, the Hilbert-Kolchin regularity
of~$\DifferentialIdeal{\EquationSet}$ is bounded
by~${\DifferentialSystemOrder{\empty}-1}$ (see \cite[Theorem
12]{dalfonso:2009}).

{}From the results of the previous subsection, it follows that the differentiation index of the system $\DifferentialSystem$ is at most $\DifferentialSystemOrder{\empty} \DifferentialIndeterminateNumber - \DifferentialIdealOrder{\EquationSet}$ (for more precise bounds, see for instance, \cite{dalfonso:2009}).

Since the fraction fields of
the domains~${\ProlongatedRing{\DifferentialSystemOrder{\empty}-1}/(
  \ProlongatedRing{\DifferentialSystemOrder{\empty}-1} \cap
  \DifferentialIdeal{\EquationSet})}$
and~${\ProlongatedRing{\DifferentialSystemOrder{\empty}}/(\ProlongatedRing{\DifferentialSystemOrder{\empty}}\cap\DifferentialIdeal{\EquationSet})}$
have the same transcendence degree over~$\BaseField$,
from the canonical
inclusion \[{\ProlongatedRing{\DifferentialSystemOrder{\empty}-1}/(
  \ProlongatedRing{\DifferentialSystemOrder{\empty}-1} \cap
  \DifferentialIdeal{\EquationSet})
  \hookrightarrow
  \ProlongatedRing{\DifferentialSystemOrder{\empty}}/(\ProlongatedRing{\DifferentialSystemOrder{\empty}}\cap\DifferentialIdeal{\EquationSet})},\]
we conclude that there exists
in~$\alltimederivative{\StateVariableSet}{\DifferentialSystemOrder{\empty}-1}$
a subset~$\TranscendenceBasis$ of~$\DifferentialIdealOrder{\EquationSet}$
many variables that is a transcendence basis of both these fields.
Moreover, we may also choose $U$ in such a way that $x_{j}^{(h)}\in U$ implies
$x_{j}^{(\ell)}\in U$ for every $0\le \ell\le h$, {e.g.\ $U$ may be
chosen as the set of derivatives that are not leading derivatives of
the algebraic characteristic set $\mathcal{A}$ in the proof of Lemma~\ref{stal-un-jour}}.
We
are going to see in the sequel that this set of variables could be
considered as initial conditions.

\section{A related vector field over an algebraic hypersurface}
\label{hypersurface}
In this section we exhibit a new \DifferentialAlgebraicSystem\
system~$\Output{\DifferentialSystem}$ related (in a non intrinsic way)
to the original one~$\DifferentialSystem$. This new
\DifferentialAlgebraicSystem\ system has a very particular structure:
a single purely algebraic (polynomial)
equation~${\PrimitiveElementMinimalPolynomial=0}$ plus an
under-determined \textsc{ode} system (see Definition~\ref{sigma1}
page~\pageref{sigma1}). In particular,~$\Output{\DifferentialSystem}$ is a
\emph{semi-explicit \DifferentialAlgebraicSystem\ system} in the usual
sense (see for instance~\cite[Section 1.2]{brenan:1996}). Moreover, we
will prove that the differentiation index
of~$\Output{\DifferentialSystem}$ is~$1$ (see Proposition~\ref{index1}
page~\pageref{index1}).
\par
The polynomial equation~${\PrimitiveElementMinimalPolynomial=0}$ is
obtained by means of a classical, \emph{purely algebraic} procedure
known today as the \emph{geometric resolution} (see
Section~\ref{res-geo}) applied to a suitable algebraic variety
associated to the input \DifferentialAlgebraicSystem\
system~$\DifferentialSystem$ and some of its derivatives (see
Section~\ref{geosol}). The differential equations
of~$\Output{\DifferentialSystem}$ are introduced in
Section~\ref{vector_field}.
\par
We leave for Section~\ref{passing} the analysis of the relations between
the solutions of both~\DifferentialAlgebraicSystem\
systems~$\DifferentialSystem$ and~$\Output{\DifferentialSystem}$.

\subsection {The prolonged algebraic system and its partial
  specialization}
\label{geosol}
We keep the notations and assumptions introduced in
Section~\ref{preliminaries} related to the
\DifferentialAlgebraicSystem\ input system~$\DifferentialSystem$.
\par
We recall that~$\TranscendenceBasis$ denotes a subset
of~$\alltimederivative{\StateVariableSet}{\DifferentialSystemOrder{\empty}-1}$
that is a transcendence basis of the fraction fields of the
domains~${\ProlongatedRing{\DifferentialSystemOrder{\empty}-1}/(\DifferentialIdeal{\EquationSet}
  \cap \ProlongatedRing{\DifferentialSystemOrder{\empty}-1})}$
and~${\ProlongatedRing{\DifferentialSystemOrder{\empty}}/(
  \DifferentialIdeal{\EquationSet} \cap
  \ProlongatedRing{\DifferentialSystemOrder{\empty}})}$. Following
Section~\ref{HK_regularity} such a basis exists and its cardinality
is~$\DifferentialIdealOrder{\EquationSet}$. Recall
that~$\DifferentiationIndex$ denotes the differentiation index
of~$\DifferentialSystem$ introduced in Section~\ref{defindex}.

\begin{proposition} \label{UW}
  \begin{enumerate}
  \item The variables~$\TranscendenceBasis$ as elements of the
    ring~${\ProlongatedRing{\DifferentialSystemOrder{\empty}+\DifferentiationIndex}/\ProlongatedIdeal{\DifferentiationIndex+1}}$
    remain algebraically independent over~$\BaseField$.
  \item Let~$\SpecializationVariableSet$ be a subset
    of~$\alltimederivative{\StateVariableSet}{\DifferentiationIndex+\DifferentialSystemOrder{\empty}}$
    such that~${\{\TranscendenceBasis, \SpecializationVariableSet\}}$
    is a transcendence basis of the fraction field
    of~${\ProlongatedRing{\DifferentialSystemOrder{\empty}+\DifferentiationIndex}/\ProlongatedIdeal{\DifferentiationIndex+1}}$. Then
    every variable in~$\SpecializationVariableSet$ has order at
    least~${\DifferentialSystemOrder{\empty}+1}$; in other
    words,~$\SpecializationVariableSet$ is a subset of~${\{
      \timederivative{\StateVariableSet}{\indexb};\
      \DifferentialSystemOrder{\empty}+1\le \indexb \le
      \DifferentialSystemOrder{\empty} + \DifferentiationIndex\}}$.
  \end{enumerate}
\end{proposition}
\begin{proof}
  Note that Theorem~\ref{index} (for~${\indexa= \DifferentialSystemOrder{\empty}}$) or Lemma \ref{stal-un-jour} imply that the canonical
  inclusion of~$\ProlongatedRing{\DifferentialSystemOrder{\empty}}$
  in~$\ProlongatedRing{\DifferentialSystemOrder{\empty}+\DifferentiationIndex}$
induces an injective~$\BaseField$-algebra
morphism~${\ProlongatedRing{\DifferentialSystemOrder{\empty}}/(
  \DifferentialIdeal{\EquationSet} \cap
  \ProlongatedRing{\DifferentialSystemOrder{\empty}}) \hookrightarrow
  \ProlongatedRing{\DifferentialSystemOrder{\empty}+\DifferentiationIndex}/\ProlongatedIdeal{\DifferentiationIndex
    +1}}$. In particular, this inclusion
preserves~$\BaseField$-algebraically free elements and then the
statement~(1) follows. In order to prove the second assertion simply
observe that~$\TranscendenceBasis$ is a transcendence basis of the
fraction field~${\ProlongatedRing{\DifferentialSystemOrder{\empty}} /
  (\DifferentialIdeal{\EquationSet} \cap
  \ProlongatedRing{\DifferentialSystemOrder{\empty}})}$ and then, for every $1\le \indexa\le \DifferentialIndeterminateNumber$, $\{ \TranscendenceBasis, x_{\indexa}^{(\DifferentialSystemOrder{\empty})} \}$
  is an algebraically dependent set modulo $\DifferentialIdeal{\EquationSet} \cap
  \ProlongatedRing{\DifferentialSystemOrder{\empty}}$, and the same holds in $\ProlongatedRing{\DifferentialSystemOrder{\empty}
  +\DifferentiationIndex}/ \ProlongatedIdeal{\DifferentiationIndex
    +1}$.
\end{proof}

Let~$\SpecializationVariableSet$ be a subset of~${\bigl\{
  \timederivative{\StateVariableSet}{\indexb} ;\
  \DifferentialSystemOrder{\empty}+1 \le \indexb \le
  \DifferentialSystemOrder{\empty} + \DifferentiationIndex\bigr\}}$
verifying the second assertion in Proposition~\ref{UW} (observe
that if~${\DifferentiationIndex=0}$ there are no
variables~$\SpecializationVariableSet$);
since~$\ProlongatedIdeal{\DifferentiationIndex +1}$ is a complete
intersection prime ideal of the polynomial
ring~$\ProlongatedRing{\DifferentialSystemOrder{\empty}+\DifferentiationIndex}$
(Proposition~\ref{prime}), we have that the cardinality
of~${\{\TranscendenceBasis,\SpecializationVariableSet\}}$ equals the
number of variables of the polynomial
ring~$\ProlongatedRing{\DifferentialSystemOrder{\empty}+\DifferentiationIndex}$
minus the number of elements of the regular sequence
defining~$\ProlongatedIdeal{\DifferentiationIndex +1}$. In other
words:
\[\card \{\TranscendenceBasis
,\SpecializationVariableSet\}=\dim
\ProlongatedRing{\DifferentialSystemOrder{\empty}
  +\DifferentiationIndex}-(\DifferentiationIndex
+1)\EquationSetCardinal =( \DifferentialSystemOrder{\empty} +
\DifferentiationIndex +1)\EquationSetCardinal-(\sigma
+1)\EquationSetCardinal=\EquationSetCardinal\DifferentialSystemOrder{\empty}.
\]
Let $\SpecializationVariableSetCardinal$ be the
cardinality of~$\SpecializationVariableSet$, that
is~${\EquationSetCardinal \DifferentialSystemOrder{\empty} -
  \DifferentialIdealOrder{\EquationSet}}$.  For any differential
polynomial~$\TheEquation{\empty}$
in~$\DifferentialPolynomialRing{\BaseField}{\StateVariableSet}$ and
any point~$\SpecializationPoint$
in~$\AffineSpace{\SpecializationVariableSetCardinal}$ denote
by~$\SpecializationOf{\TheEquation{\empty}}{\SpecializationPoint}$ the
polynomial obtained by replacing in~$\TheEquation{\empty}$ the
variables~$\SpecializationVariableSet$ by the corresponding
value~$\SpecializationPoint$.
\par

\begin{proposition} \label{omega0} There exists a nonempty Zariski
  open subset of~$\AffineSpace{\SpecializationVariableSetCardinal}$
  such that for any $\SpecializationPoint$ in this set and for all
  integer~$\indexa$ such that~${1\le \indexa \le \DifferentiationIndex
    +1}$, the following conditions are satisfied:
  \begin{enumerate}
 \item The sequence $\SpecializationOf{\timederivative{\EquationSet}{\indexb}}{\SpecializationPoint}$
    for~${0\le \indexb \le \indexa-1}$ is a reduced regular sequence in $\ProlongatedRing{\DifferentialSystemOrder{\empty}+\DifferentiationIndex}$.  In particular, the ideals~${\ProlongatedIdeal{\indexa} +
      (\SpecializationVariableSet -\SpecializationPoint)}$
    in~$\ProlongatedRing{\DifferentialSystemOrder{\empty}+\DifferentiationIndex}$
    are radical and complete intersection.
  \item No prime component of these
    ideals~${\ProlongatedIdeal{\indexa} + (\SpecializationVariableSet
      -\SpecializationPoint)}$ contains a nonzero polynomial pure
    in~$\TranscendenceBasis$.
  \end{enumerate}
\end{proposition}

\begin{proof}
 The proposition is a consequence of the results given in Appendix
 \ref{appendix:specialization}. The first statement follows directly
 from Theorem~\ref{esp_reg_seq} and for the second one we apply
 Corollary~\ref{intersection} (remark that the
 ideals~$\ProlongatedIdeal{\indexa}$ are supposed to be prime).
\end{proof}

Now we introduce an algebraic variety defined by the prolonged
equations of the input system~$\DifferentialSystem$ up to
order~$\DifferentiationIndex$ followed by a specialization of the
variables~$\SpecializationVariableSet$. Fix a specialization
point~$\SpecializationPoint$
in~$\AffineSpace{\SpecializationVariableSetCardinal}$ and suppose it
belongs to the Zariski open set given by Proposition~\ref{omega0}.
\begin{notation}
  \label{variety}
  Let $\SpecializedProlongatedIdeal$ be the ideal spanned by
  the
  subset~$\SpecializationOf{\alltimederivative{\EquationSet}{\DifferentiationIndex}}{\SpecializationPoint}$
  of~${\BaseField\!\left[\alltimederivative{\StateVariableSet}{\DifferentialSystemOrder{\empty}+\DifferentiationIndex}\setminus
      \SpecializationVariableSet\right]}$. We denote by~$\Variety{V}$
  the algebraic (equidimensional) variety
  in~$\AffineSpace{\DifferentialIdealOrder{\EquationSet} +
    (1+\DifferentiationIndex) \EquationSetCardinal}$ defined
  by the ideal~$\SpecializedProlongatedIdeal$ and by~${\Variety{V}_{1}\cup \cdots \cup \Variety{V}_{N}}$ the irreducible
  decomposition of~$\Variety{V}$.
\end{notation}

{
\begin{example}
The ideal $\SpecializedProlongatedIdeal$ may actually fail to be prime
for all values of $\mathcal{W}$ in a dense set, as shown by the
following example: $x_{1}^{(2)}-{x_{2}^{(2)}}^2=0$, $x_{2}=0$. It is easy to see that $\sigma=2$ and we may choose $\{x_{1}^{(4)},x_{2}^{(4)}\}$ as the set
$W$.  Then, for an orderly ordering, an algebraic characteristic set (in fact, a system of generators) of
$\SpecializedProlongatedIdeal$  is
$2{x_{2}^{(3)}}^{2}-\mathcal{W}_{1}$, $x_{1}^{(3)}$, $x_{1}^{(2)}$,
$x_{2}^{(2)}$, $x_{2}'$, $x_{2}$. We see that the ideal is prime if and only if
${\mathcal{W}_{1}}/{2}$ is not a square in $\BaseField$. Moreover, even in
the prime case, the field extension associated to
$\SpecializedProlongatedIdeal$ is a non-trivial algebraic extension of
degree $2$ of the field associated to the ideal
$[F]\cap\ProlongatedRing{\DifferentialSystemOrder{\empty}}$, which is
$1$. On the other hand, we could choose also $W$ as the set $\{x_{2}^{(4)},
x_{2}^{(3)}\}$; in this case $\SpecializedProlongatedIdeal$ is prime and its associated variety is birational equivalent to 
$V([F]\cap\ProlongatedRing{\DifferentialSystemOrder{\empty}})$.

Finding whenever possible, such a choice of $W$, remains a subject for
further investigations.
\end{example}
}

Observe that the algebraic variety~$\Variety{V}$ is not intrinsically
associated to the input~\DifferentialAlgebraicSystem\ system because
its definition depends on the choice of the transcendence
basis~$\TranscendenceBasis$, the
variables~$\SpecializationVariableSet$, and the
point~$\SpecializationPoint$ where the
variables~$\SpecializationVariableSet$ are evaluated.
\par
Let us also remark that the second assertion in
Proposition~\ref{omega0} states that the projection on
the~$\TranscendenceBasis$-space of any irreducible
component~$\Variety{V}_{\indexa}$ is \emph{dominant}; i.e.\ the
closure of the image of~$\Variety{V}_i$ by the projection on the
variables~$\TranscendenceBasis$ is the whole space~$\AffineSpace{
  \DifferentialIdealOrder{\EquationSet}}$ or equivalently, the natural
ring map~$\PolynomialRing{\BaseField}{\TranscendenceBasis}\to
\PolynomialRing{\BaseField}{\Variety{V}_{\indexa}}$ is injective.
\par
The following proposition shows that the identity~${
  \DifferentialIdeal{\EquationSet} \cap
  \ProlongatedRing{\DifferentialSystemOrder{\empty} } =
  \ProlongatedIdeal{\DifferentiationIndex+1} \cap
  \ProlongatedRing{\DifferentialSystemOrder{\empty}}}$ (see
Theorem~\ref{index} and Lemma \ref{stal-un-jour}) remains correct after specialization in a
suitable~$\SpecializationPoint$:
\begin{proposition}
  \label{intersec}
  Let $\SpecializationPoint$
  in $\AffineSpace{\SpecializationVariableSetCardinal}$ chosen as in
  Proposition~\ref{omega0}.
  Then the identity $\DifferentialIdeal{\EquationSet} \cap
    \ProlongatedRing{\DifferentialSystemOrder{\empty}}
    =\SpecializedProlongatedIdeal \cap
    \ProlongatedRing{\DifferentialSystemOrder{\empty}}$ holds.
\end{proposition}
\begin{proof}
  Since~${\DifferentialIdeal{\EquationSet} \cap
    \ProlongatedRing{\DifferentialSystemOrder{\empty}} =
    \ProlongatedIdeal{\DifferentiationIndex+1} \cap
    \ProlongatedRing{\DifferentialSystemOrder{\empty}}}$ and the
  variables~$\SpecializationVariableSet$ do not appear in the
  ring~$\ProlongatedRing{\DifferentialSystemOrder{\empty}}$, the
  ideal~${\DifferentialIdeal{\EquationSet} \cap
    \ProlongatedRing{\DifferentialSystemOrder{\empty}}}$ is included
  in~${\SpecializedProlongatedIdeal \cap
    \ProlongatedRing{\DifferentialSystemOrder{\empty}}}$. On the other
  hand if~$\mathfrak{p}$ is a primary component
  of~$\SpecializedProlongatedIdeal$, we have
  that~${\PolynomialRing{\BaseField}{\TranscendenceBasis}
    \hookrightarrow \ProlongatedRing{\DifferentialSystemOrder{\empty}
      + \DifferentiationIndex}/\mathfrak{p}}$ (because of the choice
  of~$\SpecializationPoint$ verifying
  Proposition~\ref{omega0}). Then~${\mathfrak{p} \cap
    \ProlongatedRing{\DifferentialSystemOrder{\empty}}}$ is a prime
  ideal of dimension at least~$\DifferentialIdealOrder{\EquationSet}$
  containing~${\DifferentialIdeal{\EquationSet} \cap
    \ProlongatedRing{\DifferentialSystemOrder{\empty}}}$, which is a
  prime ideal of
  dimension~$\DifferentialIdealOrder{\EquationSet}$. Hence both prime
  ideals are the same. Since the argument holds for any primary
  component of~$\SpecializedProlongatedIdeal$, the proposition
  follows.
\end{proof}
In other words, this proposition says that \emph{all} differential
conditions of order at most~$\DifferentialSystemOrder{\empty}$ induced
by the input system can be generated by differentiation of the
original equations up to order~$\DifferentiationIndex$ followed by the
specialization~${\SpecializationVariableSet \mapsto
  \SpecializationPoint}$.
\par
In particular, if ${\BaseField}=\mathbb{R}, \mathbb{C}$, suppose
that~${\AnalyticSolutionSet:=(\AnalyticSolution{1}, \ldots,
\AnalyticSolution{\EquationSetCardinal})}$ is a classical analytic
solution of the \DifferentialAlgebraicSystem\
system~$\DifferentialSystem$ defined locally in a neighborhood
of~$0$. Then, Proposition~\ref{intersec} implies that for any~$t$
in~$\Field{R}$ small enough, the complex vector formed by the
derivatives up to order~$\DifferentialSystemOrder{\empty}$ of the
function~$\AnalyticSolutionSet$ evaluated at the instant~$t$ is a
point of the algebraic
variety~${V}{([F]\cap\ProlongatedRing{\DifferentialSystemOrder{\empty}})}$,
\emph{independently} of the choice of the
variables~$\TranscendenceBasis,
\SpecializationVariableSet$ and the point~$\SpecializationPoint$
in~$\AffineSpace{\SpecializationVariableSetCardinal}$.

\subsection{A parametric geometric resolution of the
  variety~$\Variety{V}$}\label{res-geo}

{In this Section we introduce the algebraic part of our related {semi-explicit} \DifferentialAlgebraicSystem\ system
  $\Output{\DifferentialSystem}$ and certain rational functions that
  allow us to express solutions of $\DifferentialSystem$ from
  solutions of $\Output{\DifferentialSystem}$. To do this, we
 use a classical tool in effective Algebraic Geometry: the
  \emph{geometric resolution} of an equidimensional variety. In our
  case this construction will be applied to the algebraic
  variety~$\Variety{V}$ introduced in Notation \ref{variety}.}
\par
Let us explain informally this well-known notion (for simplicity, we
assume~$\BaseField$ is algebraically closed and of
characteristic~$0$): suppose that a~$d$-dimensional \emph{irreducible}
affine variety~$\Variety{V}$ in the~$m$-dimensional ambient
space~$\AffineSpace{m}$ is given. Then the
field~$\BaseField(\Variety{V})$ of rational functions
over~$\Variety{V}$ has transcendence degree~$d$ over the ground
field~$\BaseField$. Therefore, there exist~$d$ variables,
say~${\AlgebraicIndeterminate{1},\ldots ,\AlgebraicIndeterminate{d}}$,
such that the extension
field~${\BaseField(\AlgebraicIndeterminate{1},\ldots
  ,\AlgebraicIndeterminate{d}) \hookrightarrow
  \BaseField(\Variety{V})}$ is finite. These variables are called
\emph{parametric\/} or \emph{free\/} variables. The Primitive Element
Theorem (see for instance~\cite[\S{\empty}V.4, Theorem
4.6]{lang:2002}) asserts that there exists an
element~$\PrimitiveElement$ in~$\BaseField(\Variety{V})$ such
that~${\BaseField(\Variety{V})=\BaseField(\AlgebraicIndeterminate{1},\ldots,
  \AlgebraicIndeterminate{d})[\PrimitiveElement]}$; moreover, the
element~$\PrimitiveElement$ can be taken as a
generic~$\BaseField$-linear combination of the remaining
variables~${\AlgebraicIndeterminate{d+1},\ldots,\AlgebraicIndeterminate{m}}$. The
minimal polynomial~$\PrimitiveElementMinimalPolynomial$ of
$\PrimitiveElement$
over~${\BaseField(\AlgebraicIndeterminate{1},\ldots,
  \AlgebraicIndeterminate{d})}$ defines an irreducible
hypersurface~$\Variety{H}:={\{\PrimitiveElementMinimalPolynomial=0\}}$ in the affine
space~${\AffineSpace{d}\times \AffineSpace{1}}$
($\PrimitiveElementMinimalPolynomial$ can be taken with coefficients
in the polynomial
ring~$\PolynomialRing{\BaseField}{\AlgebraicIndeterminate{1},\ldots,
  \AlgebraicIndeterminate{d}}$). Since each of the (non-free)
variables~$\AlgebraicIndeterminate{d+1},\ldots,\AlgebraicIndeterminate{m}$,
as elements of the field~$\BaseField(\Variety{V})$, can be written as
rational functions in the variable~$\PrimitiveElement$ over the
field~$\BaseField(\AlgebraicIndeterminate{1},\ldots,\AlgebraicIndeterminate{d})$,
it follows that a dense open subset of the irreducible
variety~$\Variety{V}$ can be rationally parametrized from a dense
open subset of the hypersurface~$\Variety{H}$.
\par
The 4-tuple consisting of the parametric
set~${\{\AlgebraicIndeterminate{1},\ldots,\AlgebraicIndeterminate{d}\}}$,
the element~$\PrimitiveElement$, its minimal
polynomial~$\PrimitiveElementMinimalPolynomial$, and the rational
parametrizations is called a \emph{parametric geometric resolution of
  the variety}~$\Variety{V}$.
\par
{If} the variety~$\Variety{V}$ is not irreducible but
equidimensional, a similar construction can be reproduced with
suitable changes (see for instance~\cite[Section 2]{schost:2003}). For
instance, there is in general no chance that the same choice of the
free variables as a subset of the input variables works for any
component of~$\Variety{V}$. In this case a (generic) linear change of
coordinates may be necessary in order to obtain the same parametric
set of variables for any irreducible component. We point out also that
in this case, {for any choice of} the element~$\PrimitiveElement$,
%can be taken in such a way that
each irreducible component of the induced
hypersurface~$\Variety{H}$ parametrizes generically one (and only
one) component of~$\Variety{V}$. In particular, the number of
irreducible components of~$\Variety{V}$ and~$\Variety{H}$ is the same.
\par\bigskip
In our situation, we consider a {parametric geometric resolution} for
the equidimensional variety~$\Variety{V}$ introduced in
Notation~\ref{variety}. From Proposition~\ref{omega0} we observe first
that the variables~$\TranscendenceBasis$ are a \emph{parametric set}
with respect to the equidimensional algebraic
variety~${\Variety{V}:=\Variety{V}_{1}\cup \cdots \cup
  \Variety{V}_{N}}$, since the canonical
morphism~${\PolynomialRing{\BaseField}{\TranscendenceBasis} \to
  \PolynomialRing{\BaseField}{\Variety{V}_{\indexa}}}$ is injective
and since the relations~${\DifferentialIdealOrder{\EquationSet} =\card
  \TranscendenceBasis=\dim \Variety{V}_{\indexa}}$ hold for
all~${\indexa=1,\ldots,N}$. In particular no linear change of
coordinates is necessary in order to obtain free variables with
respect to the irreducible components of~$\Variety{V}$. Secondly, the
ideal~$\SpecializedProlongatedIdeal$ is radical and so, it is the
defining ideal of~$\Variety{V}$. Moreover, it is generated by the
regular
sequence~$\SpecializationOf{\alltimederivative{\EquationSet}{\DifferentiationIndex}}{\SpecializationPoint}$.
\par
These facts imply that for each {prime}
ideal~$I(\Variety{V}_{\indexa})$ associated
to~$\SpecializedProlongatedIdeal$ defined in Notation~\ref{variety},
${I(\Variety{V}_{\indexa}) \otimes
  \BaseField(\TranscendenceBasis)}$ is a $0$-dimensional prime ideal in the
polynomial ring with coefficients in~$\BaseField(\TranscendenceBasis)$
and
variables~${\alltimederivative{\StateVariableSet}{\DifferentialSystemOrder{\empty}
    + \DifferentiationIndex} \setminus \{ \TranscendenceBasis,
  \SpecializationVariableSet\}}$. Hence the Jacobian determinant of
the
polynomials~
$\SpecializationOf{\alltimederivative{\EquationSet}{\DifferentiationIndex}}{\SpecializationPoint}$
with respect to these
variables~${\alltimederivative{\StateVariableSet}{\DifferentialSystemOrder{\empty}
    + \DifferentiationIndex} \setminus \{ \TranscendenceBasis,
  \SpecializationVariableSet\}}$ does not vanish identically over any
component~$\Variety{V}_{\indexa}$.
\par
Thus the requirements of~\cite[Section 2.1]{schost:2003} are fulfilled
and a parametric geometric
resolution~${\left(\TranscendenceBasis,\PrimitiveElement,\PrimitiveElementMinimalPolynomial,\left\lbrace\left(\tfrac{\GeometricResolutionFiber{\indexa}}{\PrimitiveElementMinimalPolynomial'}\right)\!,{1\le
        \indexa\le (1+\DifferentiationIndex)
        \StateVariableSetCardinal} \right\rbrace\right)}$ exists.
Here~$\PrimitiveElement$ is a~$\Field{Q}$-linear combination of the
variables~${\alltimederivative{\StateVariableSet}{\DifferentialSystemOrder{\empty}
    + \DifferentiationIndex} \setminus \{ \TranscendenceBasis,
  \SpecializationVariableSet\}}$, $\PrimitiveElementMinimalPolynomial$ the
square-free polynomial
in~$\PolynomialRing{\BaseField}{\TranscendenceBasis,\PrimitiveElement}$
of positive degree in~$\PrimitiveElement$ defining a
hypersurface~$\Variety{H}$
in~${\AffineSpace{\DifferentialIdealOrder{\EquationSet}} \times
  \AffineSpace{1}}$, and~$\PrimitiveElementMinimalPolynomial'$ the
partial
derivative~$\tPartialDerivative{\PrimitiveElementMinimalPolynomial}{\PrimitiveElement}{1}$. The
fractions~$\tfrac{\GeometricResolutionFiber{\indexa}}{\PrimitiveElementMinimalPolynomial'}$
in~$\BaseField(\TranscendenceBasis,\PrimitiveElement)$ are the
parametrizations of the remaining variables. More precisely,
each~$\GeometricResolutionFiber{\indexa}$ can be written as
\begin{equation}
  \label{ab}
  \GeometricResolutionFiber{\indexa}=\frac{a_\indexa(\TranscendenceBasis,\PrimitiveElement)}{b_\indexa(\TranscendenceBasis )}
\end{equation}
where~$a_{\indexa}(\TranscendenceBasis,\PrimitiveElement)$
in~$\PolynomialRing{\BaseField}{\TranscendenceBasis,
  \PrimitiveElement}$ and~$b_{\indexa}(\TranscendenceBasis)$
in~${\PolynomialRing{\BaseField}{\TranscendenceBasis} \setminus
  \{0\}}$ are coprime polynomials verifying that ${\deg_{\PrimitiveElement} a_{\indexa} <
  \deg_{\PrimitiveElement} \PrimitiveElementMinimalPolynomial}$;
furthermore, for each variable~$\AlgebraicIndeterminate{\empty}$
in~${\alltimederivative{\StateVariableSet}{\DifferentialSystemOrder{\empty}
    + \DifferentiationIndex} \setminus \{ \TranscendenceBasis,
  \SpecializationVariableSet\}}$, there exists an integer~$\indexb$
such that~${b_{\indexb}(\TranscendenceBasis)
  \PrimitiveElementMinimalPolynomial' \AlgebraicIndeterminate{\empty}
  -a_{\indexb}(\TranscendenceBasis,\PrimitiveElement)}$ vanishes on
the variety~$\Variety{V}$.
\par\medskip
We define the \emph{total ring of fractions} of the
variety~$\Variety{V}$ in the usual way as the Artinian
ring~${\BaseField(\Variety{V}) := \BaseField(\Variety{V}_{1})\times
  \cdots \times \BaseField( \Variety{V}_{N})}$ and analogously
for~$\Variety{H}$. Therefore, from the canonical ring
inclusions~${\PolynomialRing{\BaseField}{\TranscendenceBasis}
  \hookrightarrow \PolynomialRing{\BaseField}{\Variety{H}}
  \hookrightarrow \PolynomialRing{\BaseField}{\Variety{V}}}$, by means
of the geometric resolution and passing to the total ring of
fractions, we infer that the
relations~${\BaseField(\TranscendenceBasis) \hookrightarrow
  \BaseField( \Variety{H}) \cong \BaseField (\Variety{V})}$ hold. The
inverse application~${\BaseField(\Variety{V}) \to
  \BaseField(\Variety{H})}$ is induced by the
parametrization~${\AlgebraicIndeterminate{\empty} \mapsto
  \tfrac{\GeometricResolutionFiber{\indexa}}{\PrimitiveElementMinimalPolynomial'}}$,
for~$\AlgebraicIndeterminate{\empty}$
in~${\alltimederivative{\StateVariableSet}{\DifferentialSystemOrder{\empty}
    + \DifferentiationIndex} \setminus \{ \TranscendenceBasis,
  \SpecializationVariableSet\}}$.
\par
{}From a more geometrical point of view, these facts can be stated in
the following way. Consider the linear map~${\Psi:
  \AffineSpace{\DifferentialIdealOrder{\EquationSet} +
    (\DifferentiationIndex +1)\StateVariableSetCardinal} \to
  \AffineSpace{\DifferentialIdealOrder{\EquationSet}+1}}$ defined
as~${\alltimederivative{\StateVariableSet}
  {\DifferentialSystemOrder{\empty}+\DifferentiationIndex}\setminus
  \{\SpecializationVariableSet\}\mapsto
  (\TranscendenceBasis,\PrimitiveElement)}$. For each irreducible
component~$\Variety{V}_{\indexa}$ of~$\Variety{V}$ the restriction
of~$\Psi$ to~$\Variety{V}_{\indexa}$ induces an isomorphism between
suitable nonempty Zariski open sets of~$\Variety{V}_{\indexa}$ and of the
irreducible component~${\Variety{H}_{\indexa} =
  \overline{\Psi(\Variety{V}_{\indexa})}}$ of~$\Variety{H}$. In other
words, the components~$\Variety{V}_{\indexa}$
and~$\Variety{H}_{\indexa}$ are birationally equivalent, i.e.\ the
fields of rational functions~$\BaseField(\Variety{V}_{\indexa})$
and~$\BaseField(\Variety{H}_{\indexa})$ are~$\BaseField$-isomorphic.
\par
As shown in~\cite{schost:2003}, if the polynomials
defining~$\Variety{V}$ are encoded by straight line programs, a
parametric geometric resolution can be computed by means of a
probabilistic algorithm of bounded complexity in terms of certain
parameters.
\par
In order to estimate the running time of this algorithm in our case,
we point out that from a straight-line program of
length~$\StraightLineProgramLenght$ encoding the input
polynomials~$\EquationSet$, we can obtain a straight-line program of
length~{${O\!\left( ( ( \DifferentialSystemOrder{\empty}+
    \DifferentiationIndex) \DifferentialSystemOrder{\empty}
    \StateVariableSetCardinal + \StraightLineProgramLenght)
    \DifferentiationIndex^{2} \right)}$} encoding all the
polynomials~$\alltimederivative{\EquationSet}{\DifferentiationIndex}$
(see~\cite[Section 5.2]{matera:2003} or~\cite[Lemma
21]{dalfonso:2006}). Then applying~\cite[Theorem 2]{schost:2003} to
our algebraic system~${\left\{
    \SpecializationOf{\alltimederivative{\EquationSet}{\DifferentiationIndex}}{\SpecializationPoint}
    = 0 \right\}}$ and the parametric variables~$\TranscendenceBasis$,
we deduce the following complexity result concerning the computation
of a geometric resolution of~$\Variety{V}$:

\begin{proposition}\label{a-la-Schost}
  Assume the input polynomials~$\EquationSet$ are encoded by a
  straight-line program of length~$\StraightLineProgramLenght$. Then,
  a parametric geometric resolution of~$\Variety{V}$ can be computed
  through the following steps:
  \begin{enumerate}
  \item take a point~$\Variety{U}$
    in~$\BaseField^{\DifferentialIdealOrder{\EquationSet}}$ and
    compute a geometric resolution of the zeros of the system obtained by specializing $\TranscendenceBasis = \Variety{U}$ in~${\SpecializationOf
      {\alltimederivative{\EquationSet}{\DifferentiationIndex}}
      {\SpecializationPoint}}$;
  \item apply a formal Newton lifting process, that requires
    \[
    O_{\log}\left( \left(
        \begin{array}{c}
          (\EvaluationComplexity+\mathtt{N}^{3}) \mathtt{N} \mathbf{M}(\rho)
          \\
          + ( \DifferentialIdealOrder{\EquationSet})^{2} \, \rho\, \mathbf{M}(\deg(\Variety{V}))
        \end{array}
      \right) \mathbf{M}_s(4\deg(\Variety{V}),
      \DifferentialIdealOrder{\EquationSet}) \mathtt{N} \right)
    \]
    operations in~$\BaseField$,
  \end{enumerate}
  where~${\mathtt{N} = \DifferentialIdealOrder{\EquationSet}+
    (1+\DifferentiationIndex) \StateVariableSetCardinal}$,
  ${\EvaluationComplexity =((\DifferentialSystemOrder{\empty}
    +\DifferentiationIndex) \DifferentialSystemOrder{\empty}
    \StateVariableSetCardinal +\StraightLineProgramLenght)
    \DifferentiationIndex^{2}}$ and~$\rho$ stands for the degree of
  the projection~${\Variety{V} \to \AffineSpace{\rm{ord
        \DifferentialIdeal{\EquationSet}}}}$
  mapping~${\alltimederivative{\StateVariableSet}{\DifferentialSystemOrder{\empty}+\DifferentiationIndex}
    \setminus \SpecializationVariableSet}$ to~$\TranscendenceBasis$;
  $\mathbf{M}(\indexa)$ denotes the cost of the arithmetic operations
  with univariate polynomials of degrees bounded by~$\indexa$ with
  coefficients in a ring---we can take~${\mathbf{M}(\indexa) =
    O(\indexa \log^2(\indexa) \log\log
    (\indexa))}$---and~$\mathbf{M}_s(\indexa,\indexb)$ the cost of
  $\indexb$-variate series multiplication at
  precision~$\indexa$---that can be taken less than~$O_{\log}
  \bigl(\mathbf{M}\tbinom{\indexa+\indexb}{\indexb}\bigr)$.
\end{proposition}

\begin{remark} It is easy to see that the variables $U$
can be chosen in such a way that for every variable $x_{j}^{(h)}\in U$, all its previous derivatives
$x_{j}^{(\ell)}$, for $0\le \ell\le h$, also belong to $U$. In this case
the representation given in Proposition \ref{a-la-Schost} allows us to obtain a characteristic set of the ideal for some ranking on
derivatives, using for instance the method described in \cite{dahan:2008}.
\end{remark}

% $U$ satisties the
%hypothesis of Section , \emph{viz.} $x_{j}^{(h)}\in U$ implies
%$x_{j}^{(h_{2})}\in U$ $\forall 0\le h_{2}\le h$, we may deduce from
%such a representation a characteristic set for some ranking on
%derivatives, using e.g.\ the method described in \cite{dahan:2008}
%\end{remark}
%}

We may also consider a particular kind of parametric geometric
resolution of the variety~$\Variety{V}$ that we call a \textit{Noether
  parametric geometric resolution} which requires an additional
condition: the natural
morphism~${\PolynomialRing{\BaseField}{\TranscendenceBasis} \to
  \PolynomialRing{\BaseField}{\Variety{V}}}$ must be not only
injective but also integral (i.e.\ it verifies Noether's Normalization
Lemma). This new requirement implies a more stable geometric behavior
which allows improvements in the algorithms performance.
\par
We are not able to ensure the existence of a set of
variables~$\TranscendenceBasis$ for which the variety~$\Variety{V}$ is
in Noether position, but this can be achieved by a
(generic)~$\Field{Q}$-linear change of coordinates (see for
instance~\cite{giusti:2001} and~\cite{heintz:2000}).
\par
Taking into account that the
polynomials~$\SpecializationOf{\alltimederivative{\EquationSet}{\DifferentiationIndex}}{\SpecializationPoint}$
form a reduced regular sequence, we can apply the algorithm presented
in~\cite{giusti:2001} to compute a Noether parametric geometric
resolution of the variety~$\Variety{V}$. This
leads to the following complexity result.

\begin{proposition} \label{a-la-Lecerf} Using the previous notations, assume also that the
  polynomials~$\EquationSet$ have degrees bounded by a positive
  integer~$d$.  Then, a Noether parametric geometric resolution
  of~$\Variety{V}$ can be computed over the field~$\BaseField$ by
  means of a probabilistic algorithm which runs in time
  \[
  O\left( \mathtt{N} \left(\mathtt{N}
      \EvaluationComplexity+\mathtt{N}^{\LinearAlgebraConstant}\right)
    \left(\mathbf{M}(\mathtt{D}d)^2 + \mathbf{M}(\mathtt{D})
      \sum_{\indexa=0}^{\lceil \log_2(\mathtt{D})\rceil}
      \mathbf{a}(2^{\indexa}) \right) \right)\!\!,
  \]
  where~$\mathtt{D}$ is the maximum of the degrees of the varieties
  successively defined by the
  polynomials~$\SpecializationOf{\alltimederivative{\EquationSet}{\DifferentiationIndex}}{\SpecializationPoint}$,
  and, for every positive integer~$\indexb$, $\mathbf{a}(\indexb)$ is
  the cost of the arithmetic operations in the
  quotient~$R/\mathfrak{m}^{\indexb}$, where $R$ is a polynomial
  ring with coefficients in~$\BaseField$
  in~$\DifferentialIdealOrder{\EquationSet}$ variables
  and~$\mathfrak{m}$ is the maximal ideal generated by all the
  variables ($\LinearAlgebraConstant$ denotes the linear algebra
  constant).\end{proposition}

\subsection{Computing $\sigma$, $U$ and $W$}

Up to now, we have assumed the sets $\sigma$, $U$ and $W$ to be known \emph{a priori}.
This may often be the case for $U$ for obvious physical
reasons (for example, one does not need to compute the equations of a mechanical system such as the pendulum to
know which quantities could be arbitrarily chosen). Such an assumption is much harder to justify for $W$, but we will see that suitable sets $U$ and $W$ may be computed with little extra cost.

According to Theorem \ref{index} or Lemma \ref{stal-un-jour} in order to compute $\sigma$ it is enough to find the minimum $j_0$ such that $\ProlongatedIdeal{\indexb}\cap\ProlongatedRing{\DifferentialSystemOrder{\empty}}=
\ProlongatedIdeal{\indexb+1}\cap\ProlongatedRing{\DifferentialSystemOrder{\empty}}$ (then this minimum $j_0$ is $\sigma +1$). By the primality assumption of these ideals it suffices to compare their dimensions.
Following \cite[Proposition 2 and Remark 3]{dalfonso:2008} or \cite[Proposition 6]{dalfonso:2009} the dimension of the ideal $\ProlongatedIdeal{\indexb}\cap\ProlongatedRing{\DifferentialSystemOrder{\empty}}$ is equal to $(e-j+1)n +\rank(\partial F^{(r)}/\partial
X^{(h)})_{1\le r < j,\ e<h<e+j}$, the rank being computed modulo the
prime ideal $\ProlongatedIdeal{\indexb}$. Our algorithm computes successively the ranks for the ideals $\ProlongatedIdeal{\indexb}$ and it stops when two consecutive dimensions coincide.

To do this we apply the Kronecker algorithm described in
\cite{giusti:2001,DuLe:2006:cpkpsss}. For every $j$ it computes a Noether geometric resolution for the algebraic variety defined by
$\ProlongatedIdeal{\indexb}$. This geometric resolution allows us to
reduce the rank computation modulo $\ProlongatedIdeal{\indexb}$ to a
probabilistic rank computation modulo a principal ideal. Due to the
recursive structure of Kronecker algorithm, if the equality of the
dimensions does not hold, the geometric resolution already computed
can be taken as input for the next step.

Once the differential index is obtained, then we are able to compute the set of variables $U$ and $W$. This can be done by considering the Jacobian matrix $J_\sigma$ of the polynomials $F^{[\sigma]}$ with respect to the variables $X^{(e+\sigma)},\ldots ,\dot{X},X$ . After a Gauss triangulation of $J_\sigma$, the variables indexing columns with no pivot give a transcendence basis modulo $\ProlongatedIdeal{\DifferentiationIndex
    +1}$. The set $U$ corresponds to those variables of order at most $e-1$ and the set $W$ to the remaining ones

The complexity of this procedure is similar to the one of Proposition
\ref{a-la-Lecerf} for the adequate parameters, namely: the number of
variables and equations; the degrees of the intermediate varieties.

The previous computations might simplify the obtention of a parametric geometric resolution of the variety $\Variety{V}$.
This question, and its computational
interest, are left to further investigations.

\subsection{An associated vector field over the
  hypersurface}\label{vector_field}

In this section we define a vector field on the algebraic
hypersurface~$\Variety{H}$ defined in~$\AffineSpace{1+
  \DifferentialIdealOrder{\EquationSet}}$
by~${\{\PrimitiveElementMinimalPolynomial=0\}}$ and introduced in the
previous section. Moreover, we introduce the new first-order,
quasi-regular system~$\Output{\DifferentialSystem}$ having
differentiation index~$1$, whose solutions will enable us to obtain
solutions of the given system~$\DifferentialSystem$.
\par
Consider a parametric geometric resolution of the
variety~$\Variety{V}$ and
let \[{\left(\TranscendenceBasis,\PrimitiveElement,\PrimitiveElementMinimalPolynomial,\left\lbrace
      \tfrac{\GeometricResolutionFiber{\indexa}}{\PrimitiveElementMinimalPolynomial'}
      ,{1\le \indexa\le (1+\DifferentiationIndex)
        \StateVariableSetCardinal} \right\rbrace\right)}\] be the
parametric variables, the primitive element, its minimal square-free
polynomial and the parametrizations respectively, as in the previous
section.
\par
The linear map~${\Psi:
  \AffineSpace{\DifferentialIdealOrder{\EquationSet} + (1+
    \DifferentiationIndex ) \EquationSetCardinal} \to \AffineSpace{1+
    \DifferentialIdealOrder{\EquationSet}}}$ defined
as~${\alltimederivative{\StateVariableSet}
  {\DifferentialSystemOrder{\empty} + \DifferentiationIndex} \setminus
  \SpecializationVariableSet \rightarrow (\TranscendenceBasis,
  \PrimitiveElement)}$ (recall that~$\PrimitiveElement$ is a
$\Field{Q}$-linear combination of the
variables~${\alltimederivative{\StateVariableSet}
  {\DifferentialSystemOrder{\empty} + \DifferentiationIndex} \setminus
  \{\TranscendenceBasis,\SpecializationVariableSet\}}$) gives, by
restriction, a morphism of algebraic varieties between~$\Variety{V}$
and~$\Variety{H}$ and so, it induces a
dual~$\BaseField$-morphism~$\Psi^{\star}$ between the Artinian
rings~$\BaseField(\Variety{H})$ and~$\BaseField(\Variety{V})$. From
the properties satisfied by the geometric resolution, we have
that~$\Psi^{\star}$ is an isomorphism of~$\BaseField$-algebras and its
inverse morphism~$\Phi^{\star}$ is defined, by means of the
parametrizations, as the dual of the (rational, not necessarily
polynomial) morphism of algebraic varieties:~${\Phi:\Variety{H}
  \rightarrow \Variety{V}}$ defined by~${(\TranscendenceBasis,
  \PrimitiveElement) \rightarrow (\TranscendenceBasis, \tfrac
  {\GeometricResolutionFiber{\indexa}}{\PrimitiveElementMinimalPolynomial'}
  \vert {1\le \indexa \le (1+\DifferentiationIndex)
    \EquationSetCardinal})}$.  Let us observe that both ring
morphisms~$\Psi^{\star}$ and~$\Phi^{\star}$ fix the
variables~$\TranscendenceBasis$.
\par
Since the parametric set~$\TranscendenceBasis$ has been chosen as a
subset
of~$\alltimederivative{\StateVariableSet}{\DifferentialSystemOrder{\empty}-1}$,
the set~$\timederivative{\TranscendenceBasis}{1}$ of derivatives
of~$\TranscendenceBasis$ is included
in~$\alltimederivative{\StateVariableSet}{\DifferentialSystemOrder{\empty}}$
and so, by Proposition~\ref{UW}, the
relation~${\timederivative{\TranscendenceBasis}{1} \cap
  \SpecializationVariableSet =\emptyset}$ holds.  In
particular,~$\TranscendenceBasis$
and~$\timederivative{\TranscendenceBasis}{1}$ remain invariant after
specialization of the variables~$\SpecializationVariableSet$ at any
point~$\SpecializationPoint$
in~$\BaseField^{\SpecializationVariableSetCardinal}$.
\par
Fix a
variable~$\timederivative{\TranscendenceBasisElement{\indexa}}{1}$ of
the set~$\timederivative{\TranscendenceBasis}{1}$ (${1\le
  \indexa \le \DifferentialIdealOrder{\EquationSet}}$).  We have:
\begin{enumerate}
\item[(a)]
  If~$\timederivative{\TranscendenceBasisElement{\indexa}}{1}$ is
  in~$\TranscendenceBasis$, there exists a unique integer~$\indexc$
  such that~${1\le \indexc \le \DifferentialIdealOrder{\EquationSet}}$
  and~${\timederivative{\TranscendenceBasisElement{\indexa}}{1}=\TranscendenceBasisElement{\indexc}}$.
\item[(b)]
  If~$\timederivative{\TranscendenceBasisElement{\indexa}}{1}$ is not
  in~$\TranscendenceBasis$, there exists a unique index~$\indexb$ such
  that~${1\le \indexb \le (1+\DifferentiationIndex)
    \StateVariableSetCardinal}$ and:
  \[
  \Phi^{\star}(\timederivative{\TranscendenceBasisElement{\indexa}}{1})=
  \dfrac{ \GeometricResolutionFiber{\indexb}}
  {\PrimitiveElementMinimalPolynomial'}(\TranscendenceBasis,\PrimitiveElement
  ) = \dfrac{1}{b_{\indexb}(\TranscendenceBasis)}
  \dfrac{a_{\indexb}}{\PrimitiveElementMinimalPolynomial'}(\TranscendenceBasis,\PrimitiveElement).  \]
\end{enumerate}

\begin{definition} \label{sigma1} Let~$\Output{\DifferentialSystem}$
  be the square \DifferentialAlgebraicSystem\ system in
  the~${\DifferentialIdealOrder{\EquationSet}+1}$ differential
  unknowns~$\TranscendenceBasis,\PrimitiveElement$:
  \[
  \Output{\DifferentialSystem} := \left\{
    \begin{array} [c]{cccl}%
      \timederivative{\TranscendenceBasisElement{\indexa}}{1}-
      \TranscendenceBasisElement{\indexc} &=&0,
      &\quad   \textrm{for all}\  \timederivative{\TranscendenceBasisElement{\indexa}}{1}
      \ \textrm{verifying condition~(a)}
      \\[\medskipamount]
      b_{\indexb}(\TranscendenceBasis)\PrimitiveElementMinimalPolynomial'(\TranscendenceBasis,\PrimitiveElement)
      \timederivative{\TranscendenceBasisElement{\indexa}}{1} - {a_{\indexb} (\TranscendenceBasis,\PrimitiveElement)}&=& 0,
      &\quad \textrm{for all}\   \timederivative{\TranscendenceBasisElement{\indexa}}{1} \ \textrm{verifying condition~(b)}
      \\[\medskipamount]
      \PrimitiveElementMinimalPolynomial(\TranscendenceBasis,\PrimitiveElement)&=&0.
    \end{array}
  \right.
  \]
  We denote by~$\Output{\EquationSet}:={\Output{\TheEquation{1}}, \ldots,
    \Output{\TheEquation{1+\DifferentialIdealOrder{\EquationSet}}}}$
  the polynomials
  in~$\PolynomialRing{\BaseField}{\TranscendenceBasis,\timederivative{\TranscendenceBasis}{1},\PrimitiveElement}$
  defining the system~$\Output{\DifferentialSystem}$ and
  by~$\DifferentialIdeal{\Output{\EquationSet}}$ the differential
  ideal generated by them
  in~$\DifferentialPolynomialRing{\BaseField}{\TranscendenceBasis,\PrimitiveElement}$.
\end{definition}
% Summarizing, we have introduced a new square
% \DifferentialAlgebraicSystem\ system~$\Output{\DifferentialSystem}$
% coming from the original one $\DifferentialSystem$, but also from
% less fixed (or invariant) ingredients as the predetermination of
% certain transcendence bases ($U$ and $W$), specialization of the
% variables $W$ in a generic point $\omega_0$, and a parametric
% geometric resolution of an associated algebraic variety.

Note that~$\Output{\DifferentialSystem}$ is a \emph{semi-explicit
  \DifferentialAlgebraicSystem\ system}, i.e.\ an explicit
under-determined~\textsc{ode}, consisting of its
first~$\DifferentialIdealOrder{\EquationSet}$ many equations, plus a
purely algebraic equation, given by the square-free
polynomial~$\PrimitiveElementMinimalPolynomial$.
\par
Let~${\PrimitiveElementMinimalPolynomial=\PrimitiveElementMinimalPolynomial_{1}\cdots
  \PrimitiveElementMinimalPolynomial_{r}}$ be the decomposition
of~$\PrimitiveElementMinimalPolynomial$ as a product of irreducible
factors in the polynomial
ring~$\PolynomialRing{\BaseField}{\TranscendenceBasis,\PrimitiveElement}$. Since
the variables~$\TranscendenceBasis$ are algebraically independent
modulo the ideal ~$(\PrimitiveElementMinimalPolynomial)\subset \PolynomialRing{\BaseField}{\TranscendenceBasis,\PrimitiveElement}$, we
have~${\deg_{\PrimitiveElement}
  \PrimitiveElementMinimalPolynomial_{\indexa}>0}$ for all
integer~$\indexa$. For each
factor~$\PrimitiveElementMinimalPolynomial_{\indexa}$
let~$\mathfrak{p}_i$ be the
ideal~$\DifferentialIdeal{\Output{\TheEquation{1}},\ldots,
  \Output{\TheEquation{\DifferentialIdealOrder{\EquationSet}}},
  \PrimitiveElementMinimalPolynomial_{\indexa}}$
in~$\DifferentialPolynomialRing{\BaseField}{\TranscendenceBasis,\PrimitiveElement}$.
\begin{proposition}
  \label{gamma}
  The ideal~$\DifferentialIdeal{\Output{\EquationSet}}$ is a radical
  quasi-regular differential ideal
  in~$\DifferentialPolynomialRing{\BaseField}{\TranscendenceBasis,\PrimitiveElement}$
  and its minimal primes
  are~$\mathfrak{p}_1,\ldots,\mathfrak{p}_r$.
\end{proposition}
\begin{proof}
  Let us define~${B(\TranscendenceBasis):=\prod
    b_j(\TranscendenceBasis)}$. From the particular form of the
  polynomials~$\Output{\EquationSet}$ we observe that the
  ring~$\DifferentialPolynomialRing{\BaseField}{\TranscendenceBasis,\PrimitiveElement}
  /\DifferentialIdeal{\Output{\EquationSet}}$ is isomorphic to a
  subring of the localization
  of~$\PolynomialRing{\BaseField}{\TranscendenceBasis,
    \PrimitiveElement}/(\PrimitiveElementMinimalPolynomial)$ at the
  polynomial~$B\PrimitiveElementMinimalPolynomial'$. Since~$\PrimitiveElementMinimalPolynomial$
  is assumed square-free, the
  ring~$\PolynomialRing{\BaseField}{\TranscendenceBasis,
    \PrimitiveElement}/(\PrimitiveElementMinimalPolynomial)$ has no
  non-zero nilpotent elements and the same property remains true for
  any localization of it. Therefore the
  ideal~$\DifferentialIdeal{\Output{\EquationSet}}$ is
  radical. Similarly, each
  ring~$\DifferentialPolynomialRing{\BaseField}{\TranscendenceBasis,
    \PrimitiveElement}/\mathfrak{p}_i$ is isomorphic to a
  subring of the localization
  of~$\PolynomialRing{\BaseField}{\TranscendenceBasis,
    \PrimitiveElement}/(\PrimitiveElementMinimalPolynomial_{\indexa})$
  at~$B\PrimitiveElementMinimalPolynomial'$, which is a domain
  since~$\PrimitiveElementMinimalPolynomial'_{\indexa}$ is irreducible
  in~$\PolynomialRing{\BaseField}{\TranscendenceBasis,\PrimitiveElement}
  $; thus the ideals~$\mathfrak{p}_i$ are prime.
  \par
  From the previous argument, we observe that the canonical
  map~${\PolynomialRing{\BaseField}{\TranscendenceBasis} \rightarrow
    \DifferentialPolynomialRing{\BaseField}{\TranscendenceBasis,
      \PrimitiveElement}/\DifferentialIdeal{\Output{\EquationSet}}}$
  is injective and no polynomial
  in~$\PolynomialRing{\BaseField}{\TranscendenceBasis}$ is a zero
  divisor of the
  ring~$\DifferentialPolynomialRing{\BaseField}{\TranscendenceBasis,
    \PrimitiveElement}/\DifferentialIdeal{\Output{\EquationSet}}$. Since~$\DifferentialIdeal{\Output{\EquationSet}}$
  is a radical ideal, it has only finitely many minimal (hence
  associated) prime ideals (see~\cite[Chapter 1, \S16]{ritt:1950}) and
  so, none of these minimal primes contains a non-zero polynomial
  in~$\PolynomialRing{\BaseField}{\TranscendenceBasis}$.
  \par
  Fix a minimal prime~$\mathfrak{p}$
  of~$\DifferentialIdeal{\Output{\EquationSet}}$. Since~$\PrimitiveElementMinimalPolynomial$
  is an element of~$\DifferentialIdeal{\Output{\EquationSet}}\subseteq\mathfrak{p}$, there is an irreducible
  factor~$\PrimitiveElementMinimalPolynomial_{i_0}$
  of~$\PrimitiveElementMinimalPolynomial$ lying in the
  ideal~$\mathfrak{p}$. Moreover, exactly one of these irreducible
  factors belongs to~$\mathfrak{p}$ since they are pairwise coprime
  in~$\PolynomialRing{\BaseField(\TranscendenceBasis)}{\PrimitiveElement}$
  (otherwise, B\'ezout's Identity would imply the existence of a
  non-zero polynomial in~${\mathfrak{p} \cap
    \PolynomialRing{\BaseField}{\TranscendenceBasis}}$, leading to a
  contradiction). We will show now that~$\mathfrak{p}_{i_0}$ is
  included in~$\mathfrak{p}$.
  \par
  Since the total successive derivatives of the
  polynomials~${\Output{\TheEquation{1}}, \ldots,
    \Output{\TheEquation{\DifferentialIdealOrder{\EquationSet}}}}$
  belong to~$\mathfrak{p}$
  and~$\DifferentialIdeal{\Output{\EquationSet}}$ is a subset
  of~$\mathfrak{p}$, it suffices to prove that, for all
  integer~$\indexb$, the total
  derivatives~$\timederivative{\PrimitiveElementMinimalPolynomial_{i_0}}{\indexb}$
  belong to~$\mathfrak{p}$ (note that~$\mathfrak{p}$ is not
  \textit{a priori} necessarily a differential ideal). This can be
  done by recursion in~$\indexb$. For~${\indexb=0}$ there is nothing
  to prove. Otherwise, for~${\indexb\ge 0}$,
  % Using the recursive formula of the successive total derivatives of
  % a product,
  we have
  that~${\timederivative{\PrimitiveElementMinimalPolynomial}{\indexb+1}=\sum
    _{|\indexc|=\indexb+1}
    \tfrac{(\indexb+1)!}{\indexc_{1}!\ldots\indexc_{r}!}\,
    \timederivative{\PrimitiveElementMinimalPolynomial_{1}}{\indexc_{1}}\ldots
    \timederivative{\PrimitiveElementMinimalPolynomial_{r}}{\indexc_{r}}}$
  is in~${\DifferentialIdeal{\Output{\EquationSet}} \subseteq
    \mathfrak{p}}$, which implies by induction hypothesis that the
  expression~${\timederivative{\PrimitiveElementMinimalPolynomial_{i_0}}{\indexb+1}\prod_{i\ne
      i_0} \PrimitiveElementMinimalPolynomial_{i}}$ is
  in~$\mathfrak{p}$. Since~$\prod_{i\ne i_0}
  \PrimitiveElementMinimalPolynomial_i$ is not in~$\mathfrak{p}$, we
  conclude
  that~$\timederivative{\PrimitiveElementMinimalPolynomial_{i_0}}{\indexb+1}$
  is in~$\mathfrak{p}$.
  \par
  Again, from the special form of the
  polynomials~$\Output{\EquationSet}$, it is easy to see that the
  system is quasi-regular at each minimal prime differential
  ideal~$\mathfrak{p}_i$ and
  then,~$\DifferentialIdeal{\Output{\EquationSet}}$ is quasi-regular
  (see Definition~\ref{quasirgen}).
\end{proof}
The previous proposition ensures that the hypotheses
of~\cite[Section~2]{dalfonso:2009} are fulfilled. Hence, all the
considerations concerning the differentiation index, the order and the
Hilbert-Kolchin regularity explained there can be applied to our new
\DifferentialAlgebraicSystem\
system~$\Output{\DifferentialSystem}$. In particular we can compute
the differentiation index of~$\Output{\DifferentialSystem}$ at each
minimal prime~$\mathfrak{p}$ as
in~\cite[Section~3.1]{dalfonso:2009}:
\begin{proposition}
  \label{index1}
  Let~$\mathfrak{p}$ be a minimal prime differential
  ideal containing~$\DifferentialIdeal{\Output{\EquationSet}}$. Then
  the \DifferentialAlgebraicSystem\
  system~$\Output{\DifferentialSystem}$
  has~$\mathfrak{p}$-differentiation index~$1$.
\end{proposition}
\begin{proof}
  Following Proposition~\ref{gamma},
  let~$\PrimitiveElementMinimalPolynomial_{\indexa}$
  in~$\PolynomialRing{\BaseField}{\TranscendenceBasis,\PrimitiveElement}$
  be an irreducible factor of~$\PrimitiveElementMinimalPolynomial$
  such
  that~${\mathfrak{p}_{\indexa}=\DifferentialIdeal{\Output{\TheEquation{1}},\ldots,
      \Output{\TheEquation{\DifferentialIdealOrder{\EquationSet}}
      },\PrimitiveElementMinimalPolynomial_{\indexa}}}$. Then, the
  fraction field~$\BaseField(\mathfrak{p}_{\indexa})$
  of~$\DifferentialPolynomialRing{\BaseField}{\TranscendenceBasis,\PrimitiveElement}/\mathfrak{p}_{\indexa}$
  can be identified in a natural way with the fraction field of the
  domain~$\PolynomialRing{\BaseField}{\TranscendenceBasis,
    \PrimitiveElement}/(\PrimitiveElementMinimalPolynomial_{\indexa})$;
  in particular~$\PolynomialRing{\BaseField}{\TranscendenceBasis}$ may
  be regarded as a subring of~$\BaseField(\mathfrak{p}_{\indexa})$.
  \par
  Let~$\mathfrak{J}_{1}$ be the square Jacobian matrix of the
  polynomials~$ \Output{\TheEquation{1}},
    \ldots,
    \Output{\TheEquation{\DifferentialIdealOrder{\EquationSet}+1}}$
  with respect to the
  variables~$\timederivative{\TranscendenceBasis}{1}$
  and~$\timederivative{\PrimitiveElement}{1}$. Then~$\mathfrak{J}_1$
  is the diagonal matrix:
  \[
  \mathfrak{J}_1=\left(
    \begin{tabular}{lll}
      \cline{1-1}
      \multicolumn{1}{|l}{Id} & \multicolumn{1}{|l}{$\cdots $} & $0$ \\ \cline{1-2}
      $\vdots $ & \multicolumn{1}{|l}{$\ C$} & \multicolumn{1}{|l}{$%
        \vdots $} \\ \cline{2-2}
      $0$ & $\cdots $ & $0$%
    \end{tabular}%
  \right)\!,
  \]
  where~$C$ is a diagonal matrix with the
  elements~$b_{\indexb}(\TranscendenceBasis)\PrimitiveElementMinimalPolynomial'$
  in the
  diagonal. Since~$b_{\indexb}(\TranscendenceBasis)\PrimitiveElementMinimalPolynomial'$
  is non zero in the
  domain~$\PolynomialRing{\BaseField}{\TranscendenceBasis,
    \PrimitiveElement}/(\PrimitiveElementMinimalPolynomial_{\indexa})$
  for all index~$\indexb$, we deduce that~$\mathfrak{J}_1$ has
  rank~$\DifferentialIdealOrder{\EquationSet}$ over the
  field~$\BaseField(\mathfrak{p}_{\indexa})$.
  \par
  Consider now~$\mathfrak{J}_2$ the Jacobian matrix of
  the~${2(\DifferentialIdealOrder{\EquationSet}+1)}$ many
  polynomials~${\Output{\EquationSet},\timederivative{\Output{\EquationSet}}{1}}$
  with respect to the~${2(\DifferentialIdealOrder{\EquationSet}+1)}$
  many variables~${\timederivative{\TranscendenceBasis}{1},
    \timederivative{\PrimitiveElement}{1},
    \timederivative{\TranscendenceBasis}{2},
    \timederivative{\PrimitiveElement}{2}}$. We have:
  \[
  \mathfrak{J}_2=\left(
    \begin{tabular} [c]{llllll}\cline{1-3}%
      \multicolumn{1}{|l}{\empty} &  & \multicolumn{1}{l|}{\empty} &  &  & \\
      \multicolumn{1}{|l}{\empty} & $\mathfrak{J}_1$ &  \multicolumn{1}{l|}{\empty} &   & $\ \ 0$ & \\
      \multicolumn{1}{|l}{\empty} &  & \multicolumn{1}{l|}{\empty} &  &  & \\\hline
      $\ast$ & $\cdots$ & $\ast$ & \multicolumn{1}{|l}{\empty} &  & \multicolumn{1}{l|}{\empty}%
      \\
      $\vdots$ &  & $\vdots$ & \multicolumn{1}{|l}{\empty} & $\ \mathfrak{J}_1$& \multicolumn{1}{l|}{\empty}%
      \\\cline{1-3}%
      \multicolumn{1}{|l}{$\ast$} & $\cdots$ &
      \multicolumn{1}{l|}{$\PrimitiveElementMinimalPolynomial^{\prime}$} &
      &  & \multicolumn{1}{l|}{$\ \ $}\\\hline
    \end{tabular}
  \right).\] Therefore the rank of~$\mathfrak{J}_2$
  over~$\BaseField(\mathfrak{p}_{\indexa})$ is equal
  to~${2\,\DifferentialIdealOrder{\EquationSet}+1}$ (recall that the
  last row and column of~$\mathfrak{J}_1$ are both~$0$).
  \par
  Then, the
  relation~${\dim_{\BaseField(\mathfrak{p}_{\indexa})}\ker\TransposedMatrix{\mathfrak{J}_{1}}
    =
    {1+\DifferentialIdealOrder{\EquationSet}}-\textrm{rank}_{\BaseField(\mathfrak{p}_{\indexa})}
    (\mathfrak{J}_1)=1}$ hold and thus, the
  equality~${\dim_{\BaseField(\mathfrak{p}_{\indexa})}\ker\TransposedMatrix{\mathfrak{J}_{2}}
    =
    2(1+\DifferentialIdealOrder{\EquationSet})-\textrm{rank}_{\BaseField(\mathfrak{p}_{\indexa})}
    (\mathfrak{J}_2)=1}$ hold. Thus, from~\cite[Definitions~5 \&~9 and
  Theorem~8]{dalfonso:2009} we conclude that the
  $\mathfrak{p}_{\indexa}$-differentiation index of the
  system~$\Output{\DifferentialSystem}$ equals~$1$.
\end{proof}
In order to illustrate the notions above, let us now consider the
classical pendulum example.
\begin{example} {\rm Let $\DifferentialSystem$ be the \DifferentialAlgebraicSystem\ system
     arising from a variational problem
    describing the motion of a pendulum of length~$1$, where~$g$ is
    the gravitational constant and the unknown~$\lambda$ is a Lagrange
    multiplier: \begin{equation}\label{pendulo}
      {\DifferentialSystem}=\quad \left\{ \begin{array} [c]{ccl}%
          \timederivative{x}{2}-\lambda x&=& 0, \\
          \timederivative{y}{2}-\lambda y+g&=& 0,\\
          x^{2}+y^{2}-1 &=& 0. \\
        \end{array} \right.  \end{equation} Using the
    notation~${\EquationSet :=\{ \timederivative{x}{2}-\lambda x,
      \timederivative{y}{2}-\lambda y+g, x^{2}+y^{2}-1 \} }$ to denote
    the system defining~$\DifferentialSystem$, we consider in the
    sequel the differential
    ideal~$\DifferentialIdeal{\EquationSet}$ in the differential
    ring~$\DifferentialPolynomialRing{\Field{R}}{x,y,\lambda}$.  \par
    The differentiation index of~$\DifferentialSystem$
    is~${\DifferentiationIndex = 4}$, as shown in~\cite[Example
    2]{dalfonso:2009}. We have, \begin{eqnarray*}
      \DifferentialIdeal{\EquationSet} \cap \ProlongatedRing{1} &=& \ProlongatedIdeal{4} \cap \ProlongatedRing{1} \\
      &=& (x^{2} +y^{2} -1,\, \timederivative{y}{1} x^{2} - y
      x\timederivative{x}{1} - \timederivative{y}{1},\, x
      \timederivative{x}{1} + y \timederivative{y}{1} ,\,
      \timederivative{x}{1}^{2} + \timederivative{y}{1}^{2} - y g +
      \lambda,\, \timederivative{\lambda}{1} - 3 \timederivative{y}{1}
      g) \end{eqnarray*} and~${\TranscendenceBasis := \{ x ,
      \timederivative{x}{1}\}}$ is a common transcendence basis of the
    fields of fractions
    of~${\ProlongatedRing{1}/(\DifferentialIdeal{\EquationSet}
      \cap \ProlongatedRing{1})}$
    and~${\ProlongatedRing{2}/(\DifferentialIdeal{\EquationSet}\cap\ProlongatedRing{2})}$. In
    addition, it is not difficult to see that~${ \left\lbrace x,
        \timederivative{x}{1}, \timederivative{\lambda}{3},
        \timederivative{\lambda}{4}, \timederivative{\lambda}{5},
        \timederivative{\lambda}{6}\right\rbrace}$ is a transcendence
    basis
    of~${\textrm{Frac}(\ProlongatedRing{\DifferentialSystemOrder{\empty}+\DifferentiationIndex}/\ProlongatedIdeal{\DifferentiationIndex
        +1}) =
      \textrm{Frac}(\ProlongatedRing{6}/\ProlongatedIdeal{5})}$, that
    is, we can take~${\SpecializationVariableSet := \left\lbrace
        \timederivative{\lambda}{3},\timederivative{\lambda}{4},
        \timederivative{\lambda}{5},
        \timederivative{\lambda}{6}\right\rbrace}$ as in
    Proposition~\ref{UW}.

%$\Delta_{5} =( X^{(2)}-T X,  Y^{(2)}-T Y+g,X^2+Y^2-1,
%-T \dot X + X^{(3)} - X \dot T, %
%-T \dot Y + Y^{(3)} - Y \dot T, %
%2 X \dot X + 2 Y \dot Y, %
%-2 \dot T \dot X - T X^{(2)} + X^{(4)} - X
%T^{(2)}, %
%-2 \dot T \dot Y - T Y^{(2)} + y[4] - Y T^{(2)}, %
% 2 \dot X^2  + 2 X X^{(2)} + 2 \dot Y^2  + 2 Y Y^{(2)}, %
%-3 T^{(2)} \dot X - 3 \dot T X^{(2)} - T X^{(3)} +
%X^{(5)} - X T^{(3)}, %
%-3 T^{(2)} \dot Y - 3 \dot T Y^{(2)} - T Y^{(3)} +
%Y^{(5)} - Y T^{(3)}, %
%6 X^{(2)} \dot X + 2 X X^{(3)} + 6 Y^{(2)} \dot Y + 2 Y Y^{(3)}, %
%-4 T^{(3)} \dot X - 6 T^{(2)} X^{(2)} - 4 \dot T
%X^{(3)} - T X^{(4)} + X^{(6)} - X T^{(4)}, %
%-4 T^{(3)} \dot Y - 6 T^{(2)} Y^{(2)} - 4 \dot T
%Y^{(3)} - T Y^{(4)} + Y^{(6)} - Y T^{(4)}, %
%8 X^{(3)} \dot X + 6 (X^{(2)})^2  + 2 X X^{(4)} + 8 Y^{(3)} \dot Y +
%6 (Y^{(2)})^2 + 2 Y Y^{(4)}) $

We specialize the variables $\SpecializationVariableSet$ in a point $\SpecializationPoint\in\Field{Q}^{4}$ in the generators of~$\ProlongatedIdeal{5}$ so that the conditions in Theorem~\ref{esp_reg_seq} below hold (note that, due to the structure of the
system $F$ and its successive derivatives, any specialization leads
to a reduced regular sequence).
\par
In order to compute a parametric geometric resolution of~$\SpecializationOf{\DifferentialIdeal{\EquationSet}}{\SpecializationPoint}$, we consider the linear form~${\PrimitiveElement:=y}$ that defines the primitive element~$\PrimitiveElement$. The minimal
polynomial of this linear form is
\[ \PrimitiveElementMinimalPolynomial = \PrimitiveElement^{2} + x^2 - 1. \]
Now, for each variable~$z\in{\left\lbrace \timederivative{x}{\indexa}, 2\le \indexa \le 6;\ \timederivative{y}{\indexb},
1\le \indexb \le 6; \timederivative{\lambda}{\indexc}, 0\le \indexc \le 2\right\rbrace}$, we have
polynomials~$b_z(\TranscendenceBasis)$ in~$\PolynomialRing{\Field{R}}{\TranscendenceBasis}$ and~$a_{z}(\TranscendenceBasis,\PrimitiveElement)$ in~$\PolynomialRing{\Field{R}}{\TranscendenceBasis,\PrimitiveElement}$ such that
\[
b_{z}(\TranscendenceBasis) \PartialDerivative{\PrimitiveElementMinimalPolynomial}{\PrimitiveElement}{1}
(\TranscendenceBasis,\PrimitiveElement)
\, z - a_{z}(\TranscendenceBasis,\PrimitiveElement) \] vanishes over~${\Variety{V} := \Variety{V}(\SpecializationOf{\DifferentialIdeal{\EquationSet}}{\SpecializationPoint})}$. For
instance,
\[\begin{array}{cclcccl}
b_{\timederivative{y}{1}} &=& 1, & & a_{\timederivative{y}{1}} &=& - 2 x \timederivative{x}{1},\\
b_{\lambda} &=& 1 - x^{2}, & & a_{\lambda} &=&  -2\timederivative{x}{1}^{2} \PrimitiveElement + 2g (1-x^{2})^{2}, \\
b_{\timederivative{\lambda}{1}}& =&1, & & a_{\timederivative{\lambda}{1}} &=&  - 6g x \timederivative{x}{1}, \\
b_{\timederivative{x}{2}} &=& 1 - x^{2}, & & a_{\timederivative{x}{2}} &=& -2 x \timederivative{x}{1}^{2} \PrimitiveElement + 2g x (1-x^{2})^{2},\\
b_{\timederivative{y}{2}} &=& 1, & & a_{\timederivative{y}{2}} &=& - 2g x^{2} \PrimitiveElement - 2\timederivative{x}{1}^{2},\\
b_{\timederivative{\lambda}{2}} &=& 1, & & a_{\timederivative{\lambda}{2}} &=& 6g^{2} x^{2} \PrimitiveElement - 6g\timederivative{x}{1}^{2}.
\end{array}
\]
If we denote~${\TranscendenceBasisElement{1} := x}$ and~${\TranscendenceBasisElement{2} := \timederivative{x}{1}}$, the system~$\Output{\DifferentialSystem}$ from Definition~\ref{sigma1} is

\[
\Output{\DifferentialSystem} = \left\{
\begin{array}
[c]{cccl}%
\timederivative{\TranscendenceBasisElement{1}}{1}- \TranscendenceBasisElement{2} &=&0,\\
2 \PrimitiveElement (1-\TranscendenceBasisElement{1}^{\!2}) \timederivative{\TranscendenceBasisElement{2}}{1} -2 \TranscendenceBasisElement{1} \TranscendenceBasisElement{2}^{\!2}
\PrimitiveElement- 2g \TranscendenceBasisElement{1} (1-\TranscendenceBasisElement{1}^{\!2})^{2}&=& 0, \\
\PrimitiveElement^{2} + \TranscendenceBasisElement{1}^{\!2} -1&=&0.
\end{array}
\right.
\]
By inverting~${\tPartialDerivative{\PrimitiveElementMinimalPolynomial}{\PrimitiveElement}{1} (\TranscendenceBasis,\PrimitiveElement)= 2\PrimitiveElement}$ modulo~$\PrimitiveElementMinimalPolynomial(\TranscendenceBasis,\PrimitiveElement)$, we get the simplified semi-explicit system
\[
\Output{\DifferentialSystem} = \left\{
\begin{array}
[c]{ccl}%
\timederivative{\TranscendenceBasisElement{1}}{1}&=&\TranscendenceBasisElement{2},\\
(1-\TranscendenceBasisElement{1}^{2}) \timederivative{\TranscendenceBasisElement{2}}{1} &=&
((1-\TranscendenceBasisElement{1}^{2}) g \PrimitiveElement - \TranscendenceBasisElement{2}^{2}) \TranscendenceBasisElement{1}, \\
\PrimitiveElement^{2} + \TranscendenceBasisElement{1}^{2} &=&1.
\end{array}
\right.
\]
}
\end{example}
In the following section we will show how such a
\DifferentialAlgebraicSystem\ system allows us to recover information
about the original system~$\DifferentialSystem$.

\section{Recovering solutions of $\DifferentialSystem$ from solutions
  of~$\Output{\DifferentialSystem}$}
\label{passing}
In this section, we will show that almost any solution of the
system~$\Output{\DifferentialSystem}$ introduced in
Definition~\ref{sigma1} can be lifted to a solution of the input
system~$\DifferentialSystem$. Moreover, we will prove that almost any
solution of~$\DifferentialSystem$ may be recovered from a solution
of~$\Output{\DifferentialSystem}$; more precisely, our main result
states that there is a dense Zariski open set~$\ZariskiOpenSet$ of the variety of initial conditions such
that for any point~$\Variety{\StateVariableSet}$ in~$\ZariskiOpenSet$
there exists a unique solution of~$\DifferentialSystem$ with initial
condition~$\Variety{\StateVariableSet}$ that can be obtained by
lifting a solution of~$\Output{\DifferentialSystem}$.
\par
For the sake of simplicity we assume that the ground differential
field~$\BaseField$ is a subfield of the field of complex
numbers~$\Field{C}$ and the solutions of the involved systems are
solutions in the classical sense. The arguments we will use can be
easily extended to any differential subfield~$\BaseField$ of the field
of rational complex functions~$\Field{C}(t)$ by considering~$t$ as a
new unknown variable and adding the
equation~${\timederivative{t}{1}=1}$.
\par\medskip
In order to lift solutions~$\Output{\AnalyticSolutionSet}$
of~$\Output{\DifferentialSystem}$ to solutions~$\AnalyticSolutionSet$
of~$\DifferentialSystem$, we start by introducing a dense Zariski open
subset of the hypersurface~$\Variety{H}$ that defines suitable initial
conditions determining those solutions
of~$\Output{\DifferentialSystem}$ that we will be able to lift.
\par
Let~$\timederivative{\StateVariable{\indexa}}{\indexb}$,~${0\le
  \indexb \le \DifferentialSystemOrder{\empty}-1}$,~${1\le \indexa \le
  \StateVariableSetCardinal}$, be a variable that does not belong to
the set~$\TranscendenceBasis$. Let~$\Ass{(\SpecializedProlongatedIdeal)}$ be the
set of the associated primes of the radical ideal~$\SpecializedProlongatedIdeal$ and let~${
  \cap \, \mathfrak{p}} $, where $\mathfrak{p}$ runs over   $\Ass{(\SpecializedProlongatedIdeal)}$, be the primary decomposition of~$\SpecializedProlongatedIdeal$ (see
Proposition~\ref{omega0} and Notation~\ref{variety}). Then, for each
component~$\mathfrak{p}$ (which is a prime ideal with~${\dim
  \mathfrak{p} = \card\TranscendenceBasis}$) there exists an
irreducible polynomial~$\Polynomial{\indexb\indexa\mathfrak{p}}$
in~$\PolynomialRing{\BaseField}{\TranscendenceBasis,
  \timederivative{\StateVariable{\indexa}}{\indexb}}$ that lies
in~$\mathfrak{p}$.
\par
Let~$\Polynomial{\indexb\indexa}$ be the least common multiple of the
polynomials~$(\Polynomial{\indexb\indexa\mathfrak{p}})_{\mathfrak{p}\in
  \Ass{(\SpecializedProlongatedIdeal)}}$. Note that~$\Polynomial{\indexb\indexa}$
is the product of the
polynomials~$\Polynomial{\indexb\indexa\mathfrak{p}}$, without
repeated factors:
\begin{equation}
  \label{pij}
  \Polynomial{\indexb\indexa} = \prod_{\mathfrak{p} \in \IdealSet}
  \Polynomial{\indexb\indexa\mathfrak{p}}
  \qquad \textrm{for some subset}\ \IdealSet
  \ \textrm{of}\ \Ass{(\SpecializedProlongatedIdeal)}.
\end{equation}
Therefore, we have that~$\Polynomial{\indexb\indexa}$ is in~$\SpecializedProlongatedIdeal$. Moreover, observe
that no ~$\mathfrak{p}$ in~$\Ass{(\SpecializedProlongatedIdeal)}$ contains ~$\tPartialDerivative{\Polynomial{\indexb\indexa}}{\timederivative{\StateVariable{\indexa}}{\indexb}}{1}$. In
fact, the following relation holds:
\[
\tPartialDerivative{\Polynomial{\indexb\indexa}}%
{\timederivative{\StateVariable{\indexa}}{\indexb}}{1} =
\sum\limits_{\mathfrak{p} \in \IdealSet}
\tPartialDerivative{\Polynomial{\indexb\indexa\mathfrak{p}}}
{\timederivative{\StateVariable{\indexa}}{\indexb}}{1}
\prod\limits_{\mathfrak{p}'\ne \mathfrak{p},\, \mathfrak{p}'\in
  \IdealSet} \Polynomial{\indexb\indexa\mathfrak{p}'}.
\]
Now if~$\mathfrak{p}$ is in~$\IdealSet$
then~${\tPartialDerivative{\Polynomial{\indexb\indexa\mathfrak{p}}}
  {\timederivative{\StateVariable{\indexa}}{\indexb}}{1}
  \prod\limits_{\mathfrak{p}'\ne \mathfrak{p}}
  \Polynomial{\indexb\indexa\mathfrak{p}'}}$ is not
in~$\mathfrak{p}$. Since all the remaining terms of the sum are
multiples of~$\Polynomial{\indexb\indexa\mathfrak{p}}$ they lie
in~$\mathfrak{p}$, and we conclude
that~$\tPartialDerivative{\Polynomial{\indexb\indexa}}
{\timederivative{\StateVariable{\indexa}}{\indexb}}{1}$ is not
in~$\mathfrak{p}$. The same argument runs identically
if~$\mathfrak{p}$ is not in~$\IdealSet$.
\par
Consider now the rational (not necessarily polynomial) map~${\Phi :
  \Variety{H} \rightarrow \Variety{V}}$ associated to the
parametrization of~$\Variety{V}$ from~$\Variety{H}$. Observe that for
any polynomial~$\Polynomial{\empty}$ in~$\PolynomialRing {\BaseField}
{\alltimederivative{\StateVariableSet}
  {\DifferentialSystemOrder{\empty} + \DifferentiationIndex} \setminus
  \SpecializationVariableSet}$ there exists a non negative
integer~$\indexc$ (depending on~$\Polynomial{\empty}$) such
that~$B(\TranscendenceBasis)^{\indexc}
\tPartialDerivative{\PrimitiveElementMinimalPolynomial}
{\PrimitiveElement}{1} (\TranscendenceBasis,
\PrimitiveElement)^{\indexc} \Polynomial{\empty}
\bigl(\Phi(\TranscendenceBasis,\PrimitiveElement)\bigr)$ is a
polynomial in~$\PolynomialRing{\BaseField}{\TranscendenceBasis,
  \PrimitiveElement}$, where $B(U)$ is the polynomial introduced in the proof of Proposition \ref{gamma}.
\par
\begin{notation}
  \label{defP0}
  Let~$\indexc$ be a  positive integer such that
  \[
  \Polynomial{0}(\TranscendenceBasis, \PrimitiveElement) :=
  B(\TranscendenceBasis)^{\indexc}
  \tPartialDerivative{\PrimitiveElementMinimalPolynomial}
  {\PrimitiveElement}{1} (\TranscendenceBasis,
  \PrimitiveElement)^{\indexc}
  \prod_{\timederivative{\StateVariable{\indexa}}{\indexb} \notin
    \TranscendenceBasis}
  \PartialDerivative{\Polynomial{\indexb\indexa}}{\timederivative{\StateVariable{\indexa}}{\indexb}}{1}
  \bigl(\Phi(\TranscendenceBasis, \PrimitiveElement)\bigr)
  \]
  is a polynomial in~$\PolynomialRing{\BaseField}
  {\TranscendenceBasis,
    \PrimitiveElement,\timederivative{\StateVariable{\indexa}}{\indexb}}$
  divisible
  by~$B\tPartialDerivative{\PrimitiveElementMinimalPolynomial}{\PrimitiveElement}{1}$. Note
  that, since~$\tPartialDerivative{\Polynomial{\indexb\indexa}}
  {\timederivative{\StateVariable{\indexa}}{\indexb}}{1}
  \bigl(\TranscendenceBasis,
  \timederivative{\StateVariable{\indexa}}{\indexb}\bigr)$ is not
  in~$\mathfrak{p}$ for each primary component~$\mathfrak{p}$
  of~$\SpecializedProlongatedIdeal $, the set
  \[
  {\Variety{G}_{0}:=\Variety{H}\cap
    \{\Polynomial{0}\ne 0\}}\]
  is a Zariski open set which is dense in
  the hypersurface~$\Variety{H}$. Observe that~$\Variety{G}_0$ is
  included in the definition domain of the rational map~$\Phi$.
\end{notation}
Consider the projection~${\pi_1: \AffineSpace {(1+
    \DifferentiationIndex)
    \StateVariableSetCardinal+\DifferentialIdealOrder{\EquationSet} }
  \to
  \AffineSpace{(1+\DifferentialSystemOrder{\empty})\StateVariableSetCardinal}}$
to the coordinates~${\bigl(\StateVariable{\empty},
  \timederivative{\StateVariable{\empty}}{1},\dots, \timederivative
  {\StateVariable{\empty}}
  {\DifferentialSystemOrder{\empty}}\bigr)}$. Since the
relation~${\DifferentialIdeal{\EquationSet} \cap
  \ProlongatedRing{\DifferentialSystemOrder{\empty}} =
  \SpecializedProlongatedIdeal \cap
  \ProlongatedRing{\DifferentialSystemOrder{\empty}}}$ holds (see
Proposition~\ref{intersec}), we conclude that~${
  \overline{\pi_1(\Variety{V})} =
  V\!\left(\DifferentialIdeal{\EquationSet} \cap
    \ProlongatedRing{\DifferentialSystemOrder{\empty}}\right)}$.
In the sequel, we denote by~$\SolutionSpace= V\!\left(\DifferentialIdeal{\EquationSet} \cap
    \ProlongatedRing{\DifferentialSystemOrder{\empty}}\right)$.  In
particular, if the point~$(\Variety{\TranscendenceBasis}_0,
\Variety{\PrimitiveElement}_0)$ is in~$\Variety{G}_{0}$, the
point~$\pi_1 (\Phi(\Variety{\TranscendenceBasis}_0,
\Variety{\PrimitiveElement}_0))$ is in~$\SolutionSpace$.

% Let $\pi_2:\mathbb{A}^{(1+e) n}\to \mathbb{A}^{n}$ be the projection
% $(z_0,\ldots,z_e)
% \mapsto z_0$, where $z_i:=(z_{i,1},\ldots,z_{i,n})\in \mathbb{A}^n$ for all $0\le i\le e$.\\

Now we will show that an analytic
solution~$\Output{\AnalyticSolutionSet}$ of the
\DifferentialAlgebraicSystem\ system~$\Output{\DifferentialSystem}$
such that the point~$\Output{\AnalyticSolutionSet}(0)$ is
in~$\mathcal{G}_0$ can be lifted to a solution~$\AnalyticSolutionSet$
of the \DifferentialAlgebraicSystem\ system~$\DifferentialSystem$:

\begin{theorem} \label{sig1asig} Suppose
  that~${\Output{\AnalyticSolutionSet}:
    \OpenInterval{-\varepsilon}{\varepsilon}\to
    \Field{C}^{1+\DifferentialIdealOrder{\EquationSet}}}$ is an
  analytic solution of the \DifferentialAlgebraicSystem\
  system~$\Output{\DifferentialSystem}$ such
  that~$\Output{\AnalyticSolutionSet}(t)$ is in~$\mathcal{G}_0$ for
  all $t$ and let
  \[\AnalyticSolutionSet(t) =
  \bigl( \AnalyticSolution{0}(t), \ldots,
  \AnalyticSolution{\DifferentialSystemOrder{\empty}}(t)
  \bigr):=\pi_{1}\circ\Phi\, (\Output{\AnalyticSolutionSet}(t))\]
  where $\AnalyticSolutionSet_{\indexb}=(\AnalyticSolution{\indexb,1},
  \ldots, \AnalyticSolution{\indexb, \EquationSetCardinal})$
  for~${\indexb=0, \ldots,
    \DifferentialSystemOrder{\empty}}$. Then~${\AnalyticSolutionSet_{0}:
    \OpenInterval{-\varepsilon}{\varepsilon} \to \Field{C}^{n}}$ is a
  well defined analytic function that is a solution of the input
  \DifferentialAlgebraicSystem\ system~$\DifferentialSystem$.
\end{theorem}
\begin{proof}
  Since~$\Output{\AnalyticSolutionSet}(t)$ is in the
  subset~$\mathcal{G}_0$ of~$\Dom(\Phi)$, the
  map~$\AnalyticSolutionSet_{0}$ is well defined and analytic.
  \par
  In order to prove that~$\AnalyticSolutionSet_{0}$ is a solution
  of~$\DifferentialSystem$, first note that, since the image
  of~${\pi_1\circ \Phi\circ \Output{\AnalyticSolutionSet}}$ is
  included in the variety~$\SolutionSpace$, every polynomial
  in~${\DifferentialIdeal{\EquationSet}\cap
    \ProlongatedRing{\DifferentialSystemOrder{\empty}}}$ vanishes
  at~${(\pi_1 \circ \Phi\circ \Output{\AnalyticSolutionSet})(t)}$ for
  all~$t$. In particular, this holds for each equations
  in~$\DifferentialSystem$.
  \par
  Then, it suffices to show that the coordinate functions
  of~$\AnalyticSolutionSet$ verify the following relations
  \begin{equation}
    \label{solutions}
    \leibnizoperator{\AnalyticSolution{\indexb, \indexa}(t)} = \AnalyticSolution{\indexb+1, \indexa}(t)
  \end{equation}
  for every integer~${0\le \indexb \le
    \DifferentialSystemOrder{\empty}-1,\, 1\le \indexa \le
    \DifferentialSystemSize}$ and for any~$t\in \OpenInterval{-\varepsilon}{\varepsilon}$. To do so, one can
  consider two cases.
  \par
  First suppose that~$\indexa,\indexb$ are such that the
  variable~$\timederivative{\StateVariable{\indexa}}{\indexb}$ is an
  element, say~$\TranscendenceBasisElement{\indexc}$, of the
  transcendence basis~$\TranscendenceBasis$
  of~${\textrm{Frac}(\ProlongatedRing{\DifferentialSystemOrder{\empty}-1}/\DifferentialIdeal{\EquationSet}
    \cap \ProlongatedRing{\DifferentialSystemOrder{\empty}-1})}$
  chosen in subsection~\ref{geosol}. Hence~$\AnalyticSolution{\indexb,
    \indexa}(t)$ is equal to~$\Output{\AnalyticSolution{\indexc}}(t)$
  and relation~(\ref{solutions}) agrees with the equation
  corresponding to~$\TranscendenceBasisElement{\indexc}$ in
  $\Output{\DifferentialSystem}$ after the
  specialization~${(\TranscendenceBasis,\PrimitiveElement)\mapsto
    \Output{\AnalyticSolutionSet}(t)}$. Therefore this relation is
  satisfied because~$\Output{\AnalyticSolutionSet}$ is a solution
  of~$\Output{\DifferentialSystem}$.
  \par
  Suppose now that~$\timederivative{\StateVariable{\indexa}}{\indexb}$
  is not an element of~$\TranscendenceBasis$ and
  let~$\Polynomial{\indexb\indexa}$ be its associated minimal
  polynomial in~${\DifferentialIdeal{\EquationSet} \cap
    \PolynomialRing{\BaseField}{\TranscendenceBasis,\timederivative{\StateVariable{\indexa}}{\indexb}}}\subset {\DifferentialIdeal{\EquationSet} \cap
    \ProlongatedRing{\DifferentialSystemOrder{\empty}-1}}$ as
  in~(\ref{pij}). By taking the total derivative
  of~$\Polynomial{\indexb\indexa}$, we obtain the following relation:
  \[
  \leibnizoperator \Polynomial{\indexb\indexa}
  % \!\left(\!\TranscendenceBasis,\timederivative{\TranscendenceBasis}{1},
  %   \timederivative{\StateVariable{\indexb}}{\indexa},\timederivative{\StateVariable{\indexb}}{\indexa+1}\!\right)
  =\sum_{\indexc}
  \PartialDerivative{\Polynomial{\indexb\indexa}}{\TranscendenceBasisElement{\indexc}}{1}
  \!\left(\!\TranscendenceBasis,\timederivative{\StateVariable{\indexa}}{\indexb}\!\right)
  \timederivative{\TranscendenceBasisElement{\indexc}}{1} +
  \PartialDerivative{\Polynomial{\indexb\indexa}}{\timederivative{\StateVariable{\indexa}}{\indexb}}{1}
  \!\left(\!\TranscendenceBasis,\timederivative{\StateVariable{\indexa}}{\indexb}\!\right)
  \timederivative{\StateVariable{\indexa}}{\indexb+1}\in
  \DifferentialIdeal{\EquationSet} \cap
  \ProlongatedRing{\DifferentialSystemOrder{\empty}}.
  \]
  Therefore~${\tleibnizoperator{\Polynomial{\indexb\indexa}(\AnalyticSolutionSet(t))
    }=0}$ for any~$t\in \OpenInterval{-\varepsilon}{\varepsilon}$;
  more precisely, if~$\AnalyticSolution{\TranscendenceBasis},
  \AnalyticSolution{\timederivative{\TranscendenceBasis}{1}}$ stand
  for the coordinates of~$\AnalyticSolutionSet(t)$ corresponding
  to~$\TranscendenceBasis$
  and~$\timederivative{\TranscendenceBasis}{1}$ respectively, we
  have~${\timederivative
    {\Polynomial{\indexb\indexa}}{1}(\AnalyticSolution{\TranscendenceBasis},\AnalyticSolution{\timederivative{\TranscendenceBasis}{1}},
    \AnalyticSolution{\indexb,\indexa}, \AnalyticSolution{\indexb+1,
      \indexa}) =0}$.  Since we have already proved
  relation~(\ref{solutions}) for
  every~$\TranscendenceBasisElement{\indexc}$, it follows that
  \begin{equation}\label{subs}
    \sum_{\indexc}
    \PartialDerivative{\Polynomial{\indexb\indexa}}{\TranscendenceBasisElement{\indexc}}{1}
    (\AnalyticSolution{\TranscendenceBasis}, \AnalyticSolution{\indexb,\indexa}) \
    \leibnizoperator{\AnalyticSolution{\TranscendenceBasisElement{\indexc}}(t)}
    + \PartialDerivative{\Polynomial{\indexb\indexa}}{\timederivative{\StateVariable{\indexa}}{\indexb}}{1}
    (\AnalyticSolution{\TranscendenceBasis}, \AnalyticSolution{\indexb,\indexa})  \AnalyticSolution{\indexb+1, \indexa}(t) =0.
  \end{equation}
  On the other hand, the polynomial~$\Polynomial{\indexb\indexa}$ is
  in~${\DifferentialIdeal{\EquationSet} \cap
    \ProlongatedRing{\DifferentialSystemOrder{\empty}}}$ and,
  therefore,~${\Polynomial{\indexb\indexa}
    (\AnalyticSolutionSet(t))=0}$. By differentiating this identity
  with respect to~$t$ we obtain:
  \begin{equation}\label{deriv}
    \sum_{\indexc} \PartialDerivative{\Polynomial{\indexb\indexa}}{\TranscendenceBasisElement{\indexc}}{1}
    (\AnalyticSolution{\TranscendenceBasis}(t),
    \AnalyticSolution{\indexb,\indexa}(t))
    \leibnizoperator{\AnalyticSolution{\TranscendenceBasisElement{\indexc}}(t)} +
    \PartialDerivative{\Polynomial{\indexb\indexa}}{\timederivative{\StateVariable{\indexa}}{\indexb}}{1}
    (\AnalyticSolution{\TranscendenceBasis}(t), \AnalyticSolution{\indexb,\indexa}(t))
    \leibnizoperator{\AnalyticSolution{\indexb, \indexa}(t)} =0.
  \end{equation}
  Since we assume that~$\Output{\AnalyticSolutionSet}(t)$ is
  in~$\mathcal{G}_0$, in particular we have:
  \[
  \PartialDerivative{\Polynomial{\indexb\indexa}}{\timederivative{\StateVariable{\indexa}}{\indexb}}{1}
  (\AnalyticSolution{\TranscendenceBasis}(t),\AnalyticSolution{\indexb,\indexa}(t))
  = \PartialDerivative{\Polynomial{\indexb\indexa}}{\timederivative{\StateVariable{\indexa}}{\indexb}}{1}
  (\Phi\circ\Output{\AnalyticSolutionSet}(t))\ne 0.
  \]
  Now, relation~(\ref{solutions}) is an immediate consequence of
  identities~(\ref{subs}) and~(\ref{deriv}).
\end{proof}
We have already shown above how we can recover a solution of the
original system~$\DifferentialSystem$ from a solution of the new
system~$\Output{\DifferentialSystem}$. Now we will show that almost
every solution of~$\DifferentialSystem$ can be recovered from a
solution of~$\Output{\DifferentialSystem}$. We will apply the results
of uniqueness and existence of solutions contained in
Appendix~\ref{app:exuniq}.
\par
Recall that~$\DifferentialIdeal{\Output{\EquationSet}}$ is the
differential ideal
of~$\DifferentialPolynomialRing{\BaseField}
{\TranscendenceBasis,\PrimitiveElement}$ defined by the
system~$\Output{\DifferentialSystem}$ (see Definition~\ref{sigma1}
page~\pageref{sigma1}).  Let~$\Output{\SolutionSpace}$ be the
variety~$V(\DifferentialIdeal{\Output{\EquationSet}} \cap
  \PolynomialRing{\BaseField}{\TranscendenceBasis,
    \PrimitiveElement,
    \timederivative{\TranscendenceBasis}{1},
    \timederivative{\PrimitiveElement}{1}})$
in~$\AffineSpace{2\DifferentialIdealOrder{\EquationSet}+2}$
and~${\widetilde{\pi}:\Output{\SolutionSpace}\to \Variety{H}}$ the
projection~${(\TranscendenceBasis,
  \PrimitiveElement,\timederivative{\TranscendenceBasis}{1},
  \timederivative{\PrimitiveElement}{1})\mapsto
  (\TranscendenceBasis, \PrimitiveElement)}$.

\begin{theorem}\label{sigasig1}
  There exist dense Zariski open sets~$\ZariskiOpenSet\subset \SolutionSpace$ and~$\Output{\ZariskiOpenSet}\subset \Output{\SolutionSpace}$ such that, for every point~$(\Variety{X}_{0}, \ldots ,
  \Variety{X}_{\DifferentialSystemOrder{\empty}})$
  in~$\ZariskiOpenSet$, there exist~$\epsilon>0$, a
  point~${(\Variety{\TranscendenceBasis}_{0},
    \Variety{\PrimitiveElement}_{0},
    \Variety{\timederivative{\TranscendenceBasis}{1}}_0,
    \Variety{\timederivative{\PrimitiveElement}{1}}_{0})}$
  in~$\Output{\ZariskiOpenSet}$ and an analytic
  function~${\Output{\AnalyticSolutionSet} :
    \OpenInterval{-\epsilon}{\epsilon} \to
    \Field{C}^{1+\DifferentialIdealOrder{\EquationSet}}}$ that is a
  solution of the system~$\Output{\DifferentialSystem}$ with initial
  conditions~${\Output{\AnalyticSolutionSet}(0)=(\Variety{\TranscendenceBasis}_{0},
    \Variety{\PrimitiveElement}_{0})}$ satisfying:
  \begin{itemize}
  \item $(\Output{\AnalyticSolutionSet},
    \timederivative{\Output{\AnalyticSolutionSet}}{1}):
    \OpenInterval{-\epsilon}{\epsilon} \to \Output{\ZariskiOpenSet}$,
  \item
    ${(\pi_1\circ\Phi)(\Output{\AnalyticSolutionSet}(t))=\left(\AnalyticSolutionSet(t),
        \ldots,
        \timederivative{\AnalyticSolutionSet}{\DifferentialSystemOrder{\empty}}(t)\right):
      \OpenInterval{-\epsilon}{\epsilon}\to \ZariskiOpenSet}$,
    where~$\AnalyticSolutionSet$ is the unique analytic solution of
    the system~$\DifferentialSystem$ with initial
    conditions~${\bigl(\AnalyticSolutionSet(0), \ldots,
      \timederivative{\AnalyticSolutionSet}{e-1}(0)\bigr)=(\Variety{X}_{0},
      \ldots, \Variety{X}_{e-1})}$.
  \end{itemize}
\end{theorem}

\begin{proof}
  Let~$\AuxiliaryZariskiOpenSet$ be the dense Zariski open set of
  regular points of~$\SolutionSpace$ where the projection to
  ${V\!\left(\DifferentialIdeal{\EquationSet}\cap
      \ProlongatedRing{\DifferentialSystemOrder{\empty}-1}\right)}$ is
  unramified and~$\Output{\AuxiliaryZariskiOpenSet_{0}}$ be the dense
  Zariski open set of regular points of~$\Output{\SolutionSpace}$
  where the projection~$\widetilde{\pi}$ is unramified.
  \par
  Let us denote by~$\Output{\ZariskiOpenSet}$ the dense Zariski open
  subset~${\Output{\AuxiliaryZariskiOpenSet_{0}} \cap
    \{\Polynomial{0} \ne 0\}}$ of~$\Output{\SolutionSpace}$,
  where~$\Polynomial{0}$ denotes the polynomial
  in~$\PolynomialRing{\BaseField}{\TranscendenceBasis,\PrimitiveElement}$
  introduced in Notation~\ref{defP0} page~\pageref{defP0}. Note
  that, since all the points in $\Output{\SolutionSpace}$ satisfy
  that~${B(\TranscendenceBasis)\tPartialDerivative{q}{\PrimitiveElement}{1}(\TranscendenceBasis,
    \PrimitiveElement)\ne 0}$, the
  morphism~${\Phi:\Variety{H} \to \SolutionSpace}$ is an
  isomorphism between the
  set~$\widetilde{\pi}(\Output{\ZariskiOpenSet})$ (denoted
  by~$\Output{\AuxiliaryZariskiOpenSet_{1}}$ in the sequel) and its
  image.  Let us denote by~$\ZariskiOpenSet$ the
  set~${\AuxiliaryZariskiOpenSet \cap
    (\pi_1\circ\Phi)(\Output{\AuxiliaryZariskiOpenSet_{1}})}$. We have
  then the following situation:
  \[
  \begin{array}{ccrcccc}
    \Variety{V} &\supset& \Phi(\Output{\AuxiliaryZariskiOpenSet_{1}})&
    \underset{\cong}{\overset{\Phi}{\longleftarrow}}
    &\Output{\AuxiliaryZariskiOpenSet_{1}} & \subset& \Variety{H} \\
    &&\pi_1\downarrow &&\uparrow\widetilde{\pi}\\
    \SolutionSpace&\supset&\ZariskiOpenSet&&\Output{\ZariskiOpenSet}&\subset& \Output{\SolutionSpace}
  \end{array}
  \]
  Let~$\AffinePoint{\StateVariableSet}$ be the point~${(\AffinePoint{\StateVariableSet}_{0}, \ldots ,
    \AffinePoint{\StateVariableSet}_{\DifferentialSystemOrder{\empty}})}$ in the
set~$\ZariskiOpenSet\subset \AuxiliaryZariskiOpenSet$. By
  Theorem~\ref{existeunaordene} page~\pageref{existeunaordene}, there
  exist a real~${\varepsilon>0}$, an open
  neighborhood~$\ZariskiOpenSet_{\AffinePoint{\StateVariableSet}}$
  of~$\AffinePoint{\StateVariableSet}$,  $\ZariskiOpenSet_{\AffinePoint{\StateVariableSet}}\subset{\ZariskiOpenSet \subset
    \SolutionSpace}$, and a unique analytic
  solution~${\AnalyticSolutionSet_{\AffinePoint{\StateVariableSet}} :
    \OpenInterval{-\varepsilon}{\varepsilon} \to \Field{C}^{n}}$
  of~$\DifferentialSystem$ such that the
  image~${\left(\AnalyticSolutionSet_{\AffinePoint{\StateVariableSet}},
      \ldots, \timederivative
      {\AnalyticSolutionSet_{\AffinePoint{\StateVariableSet}}}{\DifferentialSystemOrder{\empty}}
    \right)\OpenInterval{-\varepsilon}{\varepsilon}}$ is
  in~$\ZariskiOpenSet_{\AffinePoint{\StateVariableSet}}$ and the
    relation~${\left(\AnalyticSolutionSet_{\AffinePoint{\StateVariableSet}}(0),
      \ldots ,
      \timederivative{\AnalyticSolutionSet_{\AffinePoint{\StateVariableSet}}}{\DifferentialSystemOrder{\empty}-1}(0)
    \right)=(\AffinePoint{\StateVariableSet}_{0}, \ldots,
    \AffinePoint{\StateVariableSet}_{\DifferentialSystemOrder{\empty}-1})}$ holds.
  \par
  Since the point~$\AffinePoint{\StateVariableSet}$ is in ${\pi_{1}
    \circ \Phi(\Output{\AuxiliaryZariskiOpenSet_{1}})}$, there exists
  a point $\xi$ in~$\AffineSpace{(\DifferentiationIndex
    - \DifferentialSystemOrder{\empty}) \EquationSetCardinal +
    \DifferentialIdealOrder{\EquationSet}}$ such that~$(\AffinePoint{\StateVariableSet}_{0}, \ldots ,
    \AffinePoint{\StateVariableSet}_{\DifferentialSystemOrder{\empty}}, \xi)$ is in~$\Phi(\Output{\AuxiliaryZariskiOpenSet_{1}})\subset \Variety{V}$. Then, there is a point $(\Variety{\TranscendenceBasis}_{0},
    \Variety{\PrimitiveElement}_{0})$
  in $\Output{\AuxiliaryZariskiOpenSet_{1}}$ such that ${\Phi(\Variety{\TranscendenceBasis}_{0},
    \Variety{\PrimitiveElement}_{0})=(\AffinePoint{\StateVariableSet}_{0}, \ldots ,
    \AffinePoint{\StateVariableSet}_{\DifferentialSystemOrder{\empty}}, \xi)}$
  and, since~${\Output{\AuxiliaryZariskiOpenSet_{1}}=
    \widetilde{\pi}(\Output{\ZariskiOpenSet})}$, the
  relation~${(\Variety{\TranscendenceBasis}_{0},
    \Variety{\PrimitiveElement}_{0})=\widetilde{\pi}(\Variety{\TranscendenceBasis}_{0},
    \Variety{\PrimitiveElement}_{0},
    \timederivative{\Variety{\TranscendenceBasis}_{0}}{1},
    \timederivative{\Variety{\PrimitiveElement}_{0}}{1})}$ holds for
  some point~$\AffinePoint{P}=(\Variety{\TranscendenceBasis}_{0},
    \Variety{\PrimitiveElement}_{0},
  \timederivative{\Variety{\TranscendenceBasis}_{0}}{1},
  \timederivative{\Variety{\PrimitiveElement}_{0}}{1})$
  in~$\Output{\ZariskiOpenSet}$. Recalling
  that~$\Output{\ZariskiOpenSet}$ is a subset
  of~$\Output{\AuxiliaryZariskiOpenSet_{0}}$, by
  Remark~\ref{exunquasireg}, there exist a
  real~${\Output{\varepsilon}>0}$, an open
  neighborhood~$\Output{\AuxiliaryZariskiOpenSet_{\AffinePoint{P}}}\subset \Output{\ZariskiOpenSet}$
  of~$\AffinePoint{P}$ and an analytic
  solution~${\Output{\AnalyticSolutionSet} :
    \OpenInterval{-\Output{\varepsilon}}{\Output{\varepsilon}} \to
    \Field{C}^{1+ \DifferentialIdealOrder{\EquationSet}}}$
  of~$(\Output{\DifferentialSystem})$ with initial
  conditions~${\Output{\AnalyticSolutionSet}(0)=(\Variety{\TranscendenceBasis}_{0},
    \Variety{\PrimitiveElement}_{0})}$ such that the
  image~$(\Output{\AnalyticSolutionSet},
  \Output{\timederivative{\AnalyticSolutionSet}{1}})
  \OpenInterval{-\Output{\varepsilon}}{\Output{\varepsilon}}$ is
  in~$\Output{\AuxiliaryZariskiOpenSet_{\AffinePoint{P}}}$.  Then, for
  every~$t \in\OpenInterval{-\Output{\varepsilon}}{\Output{\varepsilon}}$, the
  point~$\Output{\AnalyticSolutionSet}(t)$ is in the
  subset~$\widetilde \pi(\Output{\ZariskiOpenSet})$
  of~$\Variety{G}_0$, where $\mathcal{G}_0$ is the dense Zariski open
  subset of~$\Variety{H}$ from Notation~\ref{defP0}
  page~\pageref{defP0}.
  \par
  Now, Theorem~\ref{sig1asig} implies that the relation~${\pi_1\circ
    \Phi(\Output{\AnalyticSolutionSet} (t)) =
    (\AnalyticSolutionSet(t),\dots,
    \timederivative{\AnalyticSolutionSet}
    {\DifferentialSystemOrder{\empty}}(t))}$ holds for~$t\in\OpenInterval{-\varepsilon}{\varepsilon}$,
  where~${\AnalyticSolutionSet :
    \OpenInterval{-\varepsilon}{\varepsilon} \to
    \Field{C}^{\EquationSetCardinal}}$ is a solution
  of~$\DifferentialSystem$.
  % Since $\theta (t) \in \widetilde \cU_1$ for
  % every $t$, and
  % $\Phi|_{\Output{\AuxiliaryZariskiOpenSet_{1}}}:\Output{\AuxiliaryZariskiOpenSet_{1}}
  % \underset{\cong}{\to} \Phi(\Output{\AuxiliaryZariskiOpenSet_{1}})$
  % we get that
  % $\Phi(\theta(t))\in \Phi(\Output{\AuxiliaryZariskiOpenSet_{1}})$,
  % for all $t\in
  % (-\epsilon, \epsilon)$.
  Since ${\pi_1\circ \Phi(\Output{\AnalyticSolutionSet}(0))}={\pi_1\circ \Phi(\Variety{\TranscendenceBasis}_{0},
    \Variety{\PrimitiveElement}_{0})}={\pi_1(\AffinePoint{\StateVariableSet}_{0}, \ldots ,
    \AffinePoint{\StateVariableSet}_{\DifferentialSystemOrder{\empty}}, \xi)
    =\AffinePoint{\StateVariableSet}}$, which lies in~$\ZariskiOpenSet$, taking a
  smaller~$\varepsilon$ if necessary, we get that for every~$t\in \OpenInterval{-\varepsilon}{\varepsilon}$ the point~${\left(\AnalyticSolutionSet(t),\dots,
      \timederivative{\AnalyticSolutionSet}{\DifferentialSystemOrder{\empty}}(t)\right)}$
  is
  in~${\pi_1\circ\Phi(\Output{\AuxiliaryZariskiOpenSet_{1}})\cap\AuxiliaryZariskiOpenSet}=\ZariskiOpenSet$. Moreover, as~${(\AnalyticSolutionSet(0), \dots,
    \timederivative{\AnalyticSolutionSet}{\DifferentialSystemOrder{\empty}-1}(0))
    = (\AffinePoint{\StateVariableSet}_{0}, \ldots ,
    \AffinePoint{\StateVariableSet}_{\DifferentialSystemOrder{\empty}-1})}$, if we
  denote~$\min\{\Output{\varepsilon}, \varepsilon\}$ by~$\epsilon$,
  using the uniqueness statement of Theorem~\ref{existeunaordene}, we
  conclude that~${\AnalyticSolutionSet_{\Variety{\StateVariableSet}}(t)=
    \AnalyticSolutionSet(t)}$ for every~$t\in\OpenInterval{-\epsilon}{\epsilon}$.
\end{proof}

\section*{Conclusion}
In this paper we presented a new index reduction method for a class of implicit DAE systems which is based on a characterization  of the differentiation index from an algebraic point of view.
We proved that any of these systems
is generically equivalent to a  first
order semiexplicit DAE system with differentiation index $1$ and
we described a probabilistic algorithm to compute the index and
the new DAE system by using the Kronecker solver for polynomial equations.

Our results rely
on some \emph{a priori} hypotheses on the
considered differential system, for example the primality of the ideal $[F]$, which seems natural in practice, and the primality of the prolongation ideals, which is in general quite difficult to test.
Assuming an admissible initial condition for our system to be known, polynomial time numerical methods of resolution, such as the one described in \cite{corless-ilie:2009} or
power series computations (\cite{bostan:2007}), may be used in order to obtain a solution. Since it is always possible to
test, by simple substitution in the input equations, if a
solution of the new system is actually a solution of the
original DAE, one can attempt to use our method even if
all our requirements on the considered system are not guaranteed.

%In most cases, if some hypothesis is not fulfilled, some zero divisor will be found during the Kronecker resolution or when trying to compute the vector field on $\mathcal{V}$ as explained in subsection~\ref{vector_field}.

 A next step in a future work would be to generalize the method to positive differential dimension and to regular components of systems, without any extra technical hypothesis.

\emph{Urbs Rom{\ae} non uno die condita fuit.}

\appendix

\section{On the specialization of free variables in a regular
  sequence}\label{appendix:specialization}

This appendix deals with Bertini-type results from Commutative Algebra that justify
the random evaluation of suitable free variables made in
Proposition~\ref{omega0} page~\pageref{omega0}. We have decided to
include them for the sake of completeness and the lack of adequate
references.
\par
Throughout the appendix,~$\BaseField$ denotes a field of
characteristic~$0$
and~$\PolynomialRing{\BaseField}{\StateVariableSet}$ the ring of
polynomials in the
variables~${\StateVariableSet:=\{\StateVariable{1},\dots,
  \StateVariable{\StateVariableSetCardinal}\}}$ with coefficients
in~$\BaseField$.

% $K$ denotes a field of characteristic $0$ and $\overline{K}$ its
% algebraic closure. The ring of $n$-variate polynomials over $K$ is
% denoted by $K[X_1,\ldots,X_n]$ or simply $K[X]$, where
% $X:=X_1,\ldots,X_n$. %The affine space $\overline{K}^n$
% equipped with the Zariski topology of $K$-definable algebraic sets
% is denoted by $\mathbb{A}^n$. For each ideal $I\subset K[X]$, $V(I)$
% is the algebraic set defined as the common zeros of the polynomials
% in $I$ on the space $\mathbb{A}^n$.

We start recalling a well-known result concerning the behavior of
radical ideals under field extensions of~$\BaseField$ (see \cite[\S 27,
(27.2), Lemma 2]{matsumura:1970} and \cite[Volume 2, Chapter 10, \S
11]{hodge:1994}):
\begin{theorem}
  \label{hp}
  Let~${\BaseField\subseteq \BaseField^{\star}}$ be a field extension
  and let~$\Ideal{\empty}$ be a radical equidimensional ideal
  in~$\PolynomialRing{\BaseField}{\StateVariableSet}$. Then, the
  subset~$\Ideal{\empty}^{\star}:={ \Ideal{\empty} \otimes \BaseField^{\star}}$
  of~$\PolynomialRing{\BaseField^{\star}}{\StateVariableSet}$ is also
  a radical equidimensional ideal. Furthermore, if~$\Ideal{\empty}$ is a prime ideal
  and~$\StateVariableSubset$ is a subset of the variables
  $\StateVariableSet$ such
  that~${\PolynomialRing{\BaseField}{\StateVariableSubset}
    \hookrightarrow
    \PolynomialRing{\BaseField}{\StateVariableSet}/\Ideal{\empty}}$,
  then~${\PolynomialRing{\BaseField^{\star}}{\StateVariableSubset}
    \hookrightarrow
    \PolynomialRing{\BaseField^{\star}}{\StateVariableSet}/\mathfrak p}$
  for any prime ideal~$\mathfrak p$
  in~$\AssociatedPrime{\Ideal{\empty}^{\star}}$ (in other words, the
  transcendence of the variables modulo~$\Ideal{\empty}$ is preserved
  modulo any primary component of~${\Ideal{\empty}}^{\star}$).
\end{theorem}

\begin{notation}
 For an ideal $\Ideal{\empty}\subset \PolynomialRing{\BaseField}{\StateVariableSet}$, a subset $W\subset X$ of cardinality $s$ and any  $\SpecializationPoint\in \BaseField^{{\SpecializationVariableSetCardinal}}$, we denote by $\Ideal{\empty}|_{\SpecializationPoint}$ the ideal of $\PolynomialRing{\BaseField}{\StateVariableSet\setminus \SpecializationVariableSet }$ obtained after the specialization $\SpecializationVariableSet\mapsto \SpecializationPoint$ in the polynomials of $\Ideal{\empty}$.
\end{notation}

\begin{lemma} \label{lemma1} Let~$\Ideal{\empty}$ be a radical
  equidimensional ideal
  in~$\PolynomialRing{\BaseField}{\StateVariableSet}$ of
  dimension~$s$.  Let~$\SpecializationVariableSet$
  be a subset of $s$ variables in~$\StateVariableSet$ such that the canonical
  morphism ${\PolynomialRing{\BaseField}{\SpecializationVariableSet}
    \rightarrow
    \PolynomialRing{\BaseField}{\StateVariableSet}/\Ideal{\empty}}$ is
  injective. Then, there exists a nonempty~$\BaseField$-definable
  Zariski open set~$\ZariskiOpenSet$
  in~$\AffineSpace{\SpecializationVariableSetCardinal}$ such
  that~$\Ideal{\empty}|_{\SpecializationPoint}$
  is a radical $0$-dimensional ideal
  in~$\PolynomialRing{\BaseField}{\StateVariableSet\setminus \SpecializationVariableSet }$ for any
  point~$\SpecializationPoint$ in~${\ZariskiOpenSet \cap
    \BaseField^\SpecializationVariableSetCardinal}$.
\end{lemma}
\begin{proof}
  Let~$\widetilde{\Ideal{\empty}}$ denote the ideal~${\Ideal{\empty}
    \otimes \BaseField(\SpecializationVariableSet)}$
  in~$\PolynomialRing{\BaseField(\SpecializationVariableSet)}{\StateVariableSubset}$
  where~$\StateVariableSubset$ stands for the set of variables ~${\StateVariableSet \setminus \SpecializationVariableSet}$.
  From the hypotheses and Theorem~\ref{hp}, this ideal is a
  radical~$0$-dimensional ideal in the polynomial
  ring~$\PolynomialRing{\BaseField(\SpecializationVariableSet)}{\StateVariableSubset}$.
  \par
  Then, the Shape Lemma implies that there exist a square-free
  polynomial~$\PrimitiveElementMinimalPolynomial(\SpecializationVariableSet,\PrimitiveElement)$
  in~$\PolynomialRing{\PolynomialRing{\BaseField}{\SpecializationVariableSet}}{\PrimitiveElement}$,
  where~$\PrimitiveElement$ is a new single indeterminate,
  polynomials~$\GeometricResolutionFiber{y}$
  in~$\PolynomialRing{\PolynomialRing{\BaseField}{W}}{\PrimitiveElement}$
  for each variable~$y\in \StateVariableSubset$, and a~$\BaseField$-linear
  form $\ell$ in the variables~$\StateVariableSubset$ such that the
  equality of ideals
  \begin{equation} \label{WW} \widetilde{\Ideal{\empty}}=\left( {\!
      \PrimitiveElementMinimalPolynomial(\SpecializationVariableSet,\ell),
             \left(\PartialDerivative{\PrimitiveElementMinimalPolynomial}{\PrimitiveElement}{1}
        (\SpecializationVariableSet,\ell)
        y-\GeometricResolutionFiber{y}(\SpecializationVariableSet,\ell)
     \right)_{y\in \StateVariableSubset}
    }\right)
  \end{equation}
  holds
  in~$\PolynomialRing{\BaseField(\SpecializationVariableSet)}{\StateVariableSubset}$. Moreover,
  we may assume that the
  polynomials~${\PrimitiveElementMinimalPolynomial,
    \tPartialDerivative{\PrimitiveElementMinimalPolynomial}{\PrimitiveElement}{1}
    y -
    \GeometricResolutionFiber{y}}$ have trivial
  content
  in~$\PolynomialRing{\BaseField}{\SpecializationVariableSet}$.  From
  identity~(\ref{WW}), it follows that there exists a nonzero
  denominator~$g$
  in~$\PolynomialRing{\BaseField}{\SpecializationVariableSet}$ such
  that the relation~${\Ideal{\empty}\,
    \PolynomialRing{\BaseField}{\StateVariableSet}_{g} =
   \left({\PrimitiveElementMinimalPolynomial
      (\SpecializationVariableSet, \ell),
      (\tPartialDerivative{\PrimitiveElementMinimalPolynomial}{\PrimitiveElement}{1}
      (\SpecializationVariableSet, \ell) y -
      \GeometricResolutionFiber{y}(\SpecializationVariableSet,\ell))_{y\in \StateVariableSubset}}\right)
      \PolynomialRing{\BaseField}{\StateVariableSet}_{g}}$ holds.
  \par
  Finally, let~$\ZariskiOpenSet$ be the Zariski open subset
  of~$\AffineSpace{\SpecializationVariableSetCardinal}$ where the
  product~${\left(g\,
      \tPartialDerivative{\PrimitiveElementMinimalPolynomial}{\PrimitiveElement}{1}
      \discr_{\PrimitiveElement}{\PrimitiveElementMinimalPolynomial}
    \right)\!  (\SpecializationVariableSet,\ell)}$ is nonzero. Clearly
  for any point~$\SpecializationPoint$ in~$\ZariskiOpenSet$ we have
  \[
  \Ideal{\empty}|_{\SpecializationPoint}= \left({
    \PrimitiveElementMinimalPolynomial(\SpecializationPoint,\ell),(\tPartialDerivative{\PrimitiveElementMinimalPolynomial}{\PrimitiveElement}{1}
    (\SpecializationPoint,\ell)y-
    \GeometricResolutionFiber{y}
    (\SpecializationPoint,\ell))_{y\in \StateVariableSubset}} \right),
  \]
  that is a~$0$-dimensional radical ideal
  in~$\PolynomialRing{\BaseField}{\StateVariableSubset}$.
\end{proof}
The previous Lemma can be generalized as follows:
\begin{lemma}
  \label{lemma2}
  Let~$\Ideal{\empty}$ be a radical equidimensional ideal
  in~$\PolynomialRing{\BaseField}{\StateVariableSet}$
  and~ $\SpecializationVariableSet$ be
  a subset  of $s$ variables
  in~$\StateVariableSet$ such that the natural
  morphism~${\PolynomialRing{\BaseField}{\SpecializationVariableSet}
    \rightarrow
    \PolynomialRing{\BaseField}{\StateVariableSet}/\Ideal{\empty}}$ is
  injective (in particular the relation~${s\le\dim \Ideal{\empty}}$
  holds). Then, there exists a nonempty~$\BaseField$-definable Zariski
  open set~$\ZariskiOpenSet$
  in~$\AffineSpace{\SpecializationVariableSetCardinal}$ such that, for
  any point~$\SpecializationPoint$ in~${\ZariskiOpenSet\cap
    \BaseField^{\SpecializationVariableSetCardinal}}$, the
  ideal~${\Ideal{\empty}|_{\SpecializationPoint}}$ has dimension~${\dim{\Ideal{\empty}}-
    \SpecializationVariableSetCardinal}$ and its primary components of
  maximal dimension are prime ideals.
\end{lemma}
\begin{proof}
  Let us denote~$\dim{\Ideal{\empty}}$ by~$d$. If~${d=s}$, the result follows from Lemma~\ref{lemma1}. Suppose now
  that~${s<d}$.  Let us denote by~$\IdealSet$ the set of primary
  components~$\PrimeIdeal{\empty}$ of~$\Ideal{\empty}$ such
  that~${\PolynomialRing{\BaseField}{\SpecializationVariableSet}
    \rightarrow
    \PolynomialRing{\BaseField}{\StateVariableSet}/\PrimeIdeal{\empty}}$
  is a monomorphism.  Since the canonical
  map~${\PolynomialRing{\BaseField}{\SpecializationVariableSet}
    \rightarrow
    \PolynomialRing{\BaseField}{\StateVariableSet}/\Ideal{\empty}}$ is
  assumed to be injective, the set~$\IdealSet$ is a nonempty subset
  in~$\AssociatedPrime{\Ideal{\empty}}$. If we consider the canonical
  projection~${\pi: \AffineSpace{\StateVariableSetCardinal}\rightarrow
    \AffineSpace{\SpecializationVariableSetCardinal}}$ the prime
  ideals in~$\IdealSet$ correspond exactly with the
  components~$V(\PrimeIdeal{\empty})$
  of~$V(\Ideal{\empty})$ such
  that~${\overline{\pi(V(\PrimeIdeal{\empty}))}=\AffineSpace{\SpecializationVariableSetCardinal}}$. Therefore,
  the set~${\bigcup_{\PrimeIdeal{\empty}\notin \IdealSet}
    \overline{\pi(V(\PrimeIdeal{\empty}))}}$ is a proper closed subset
  of~$\AffineSpace{\SpecializationVariableSetCardinal}$;
  let~$\ZariskiOpenSet_{0}$ denote its complement
  in~$\AffineSpace{\SpecializationVariableSetCardinal}$ (that is a
  nonempty Zariski open set).
  \par
  Let~$\TranscendenceBasis$ denote a set of $d-s$~$\BaseField$-linear combinations of the
  variables~${\StateVariableSet \setminus \SpecializationVariableSet}$
  such that the natural
  morphism~${\PolynomialRing{\BaseField}{\TranscendenceBasis,
      \SpecializationVariableSet} \rightarrow
    \PolynomialRing{\BaseField}{\StateVariableSet}/\PrimeIdeal{\empty}}$
  is a monomorphism for all~$\PrimeIdeal{\empty}$ in~$\IdealSet$.  After
  a change of variables we may suppose without loss of generality
  that~$\TranscendenceBasis$ is a subset of~${\StateVariableSet
    \setminus \SpecializationVariableSet}$. Let $\StateVariableSubset$ be the set of
  the~${\StateVariableSetCardinal-d}$ many remaining
  variables~$\StateVariableSet \setminus \{ \TranscendenceBasis,
  \SpecializationVariableSet\}$.
  \par
  By Theorem~\ref{hp}, the
  ideal~${\widetilde{\Ideal{\empty}}:=\Ideal{\empty} \otimes
    \BaseField(\TranscendenceBasis)}$ is a radical equidimensional
  ideal in~$\PolynomialRing{\BaseField(\TranscendenceBasis)}
  {\SpecializationVariableSet,\StateVariableSubset}$.  Moreover,
  if~$\PrimeIdeal{\empty}$ is in~$\IdealSet$, we have
  that~${\PrimeIdeal{\empty} \otimes \BaseField(\TranscendenceBasis)}$
  is a prime ideal of dimension~${d-\card{\TranscendenceBasis}=s}$ and
  so,~$\widetilde{\Ideal{\empty}}$ is also $s$-dimensional. Therefore,
  the ideal~$\widetilde{\Ideal{\empty}}\subset\PolynomialRing{\BaseField(\TranscendenceBasis)}{\SpecializationVariableSet,
    \StateVariableSubset}$ and the
  variables~$\SpecializationVariableSet$ meet the hypotheses of
  Lemma~\ref{lemma1} and then, there exists a Zariski nonempty open
  set~$\mathcal{G}$
  in~$\AffineSpace{\SpecializationVariableSetCardinal}_{\overline{\BaseField(\TranscendenceBasis)}}$,
  definable over~$\BaseField(\TranscendenceBasis)$, such that
  ${\widetilde{\Ideal{\empty}}|_{\SpecializationPoint}}$
  is a radical~$0$-dimensional ideal
  in~$\PolynomialRing{\BaseField(\TranscendenceBasis)}{\StateVariableSubset}$
  for every~$\SpecializationPoint$ in~${\mathcal{G}\cap
    \BaseField(\TranscendenceBasis)^s}$.  Let~$\ZariskiOpenSet_{1}$ be
  a Zariski nonempty open set
  in~$\AffineSpace{\SpecializationVariableSetCardinal}$ definable
  over~$\BaseField$ such that~$\ZariskiOpenSet_{1}$ is a subset
  of~$\mathcal{G}$.  Then, if~$\SpecializationPoint$ is a point
  in~${\ZariskiOpenSet_{1} \cap
   \BaseField^{\SpecializationVariableSetCardinal}}$, %condition~(\ref{U invert}) implies that
  the primary components
  of~${{{\Ideal{\empty}}|_{\SpecializationPoint}}}$
  in~$\PolynomialRing{\BaseField}{\TranscendenceBasis,\StateVariableSubset}$
  with no nonzero polynomial pure in the
  variables~$\TranscendenceBasis$ are~${(d-
    \SpecializationVariableSetCardinal)}$-dimensional prime ideals.
  \par
  On the other hand, since the projection~${\pi:
    \AffineSpace{\StateVariableSetCardinal} \rightarrow
    \AffineSpace{\SpecializationVariableSetCardinal}}$ induces a
  dominant regular morphism of varieties~${\pi:
    V(\Ideal{\empty}) \rightarrow
    \AffineSpace{\SpecializationVariableSetCardinal}}$, by the theorem
  on the dimension of the fibers (see for instance~\cite[Chapter~I, \S
  6.3, Theorem~7]{shafarevich:1994}, suitably adapted to the case of a
  ground field not necessarily algebraically closed), there exists
  a~$\BaseField$-definable Zariski nonempty open
  subset~$\ZariskiOpenSet_{2}$
  of~$\AffineSpace{\SpecializationVariableSetCardinal}$ such that for
  every point~$\SpecializationPoint$ in~$\ZariskiOpenSet_{2}$ the
  fiber~${\pi^{-1}(\SpecializationPoint)}$ is~$\BaseField$-definable
  and geometrically equidimensional of dimension equal to~${\dim
    V({\Ideal{\empty}})-\dim
    \AffineSpace{\SpecializationVariableSetCardinal}=d-
    \SpecializationVariableSetCardinal}$.
  \par
  Summarizing, if the point~$\SpecializationPoint$ is
  in~${\ZariskiOpenSet_{0}\cap
    \ZariskiOpenSet_{1}\cap\ZariskiOpenSet_{2}\cap
    \BaseField^{\SpecializationVariableSetCardinal}}$, then the
  ideal~${\Ideal{\empty}+({\SpecializationVariableSet-\SpecializationPoint})}$
  satisfies:
  \begin{enumerate}
  \item Any of its primary components of maximal dimension
    contains an isolated prime~$\PrimeIdeal{\empty}$ in~$\IdealSet$
    (that is a subset of~$\AssociatedPrime{\Ideal{\empty}}$) such
    that~${\PolynomialRing{\BaseField}{\TranscendenceBasis,\SpecializationVariableSet}
      \hookrightarrow
      \PolynomialRing{\BaseField}{\StateVariableSet}/\PrimeIdeal{\empty}}$
    (since~$\SpecializationPoint$ is in~$\ZariskiOpenSet_{0}\cap \ZariskiOpenSet_{2}$ and due
    to the choice of the variables~$\TranscendenceBasis$).
  \item Its primary components of maximal dimension which contain no
    nonzero polynomials
    in~$\PolynomialRing{\BaseField}{\TranscendenceBasis}$ are prime
    (because~$\SpecializationPoint$ is in~$\ZariskiOpenSet_{1}$).
  \item It is geometrically equidimensional of dimension~${d-
      \SpecializationVariableSetCardinal}$
    (since~$\SpecializationPoint$ is in~$\ZariskiOpenSet_{2}$).
  \end{enumerate}
  From condition~(2), the Lemma will be proved if we are able \emph{to
    exhibit a~$\BaseField$-definable Zariski nonempty open set
    contained in~${\ZariskiOpenSet_{0} \cap \ZariskiOpenSet_{1} \cap
      \ZariskiOpenSet_{2}}$ such that, for any
    point~$\SpecializationPoint$ lying in this open set, all the
    isolated components
    of~${\Ideal{\empty}+({\SpecializationVariableSet-\SpecializationPoint})}$
    contain no nonzero polynomials
    of~$\PolynomialRing{\BaseField}{\TranscendenceBasis}$}. The
  remaining part of the proof is devoted to showing this fact.
  \par
  Let~$\PrimeIdeal{\empty}$ be a prime ideal in~$\IdealSet$. We have
  the following
  injections~${\PolynomialRing{\BaseField}{\SpecializationVariableSet}
    \hookrightarrow \PolynomialRing{\BaseField}{\TranscendenceBasis,
      \SpecializationVariableSet} \hookrightarrow
    \PolynomialRing{\BaseField}{\StateVariableSet}
    /\PrimeIdeal{\empty}}$, and each
  variable~$\StateVariableSubsetElement{\indexa}$
  in~$\StateVariableSubset$, with~${\indexa=1,\ldots
    ,\StateVariableSetCardinal - d}$, verifies a polynomial
  equation~${\Polynomial{\indexa}(\StateVariableSubsetElement{\indexa})=0}$
  modulo~$\PrimeIdeal{\empty}$, where~$\Polynomial{\indexa}$ is a
  nonzero element
  in~$\PolynomialRing{\PolynomialRing{\BaseField}{\TranscendenceBasis,
      \SpecializationVariableSet}}{T}$ with~${\deg_{T}
    \Polynomial{\indexa}>0}$ (here~$T$ denotes a new variable). We
  denote by~$\LeadingCoefficient({\Polynomial{\indexa}})$ the leading
  coefficient of~$\Polynomial{\indexa}$
  in~$\PolynomialRing{\BaseField}{\TranscendenceBasis,\SpecializationVariableSet}$.
  \par\medskip\noindent
  \textbf{Claim}: \emph{There exists a~$\BaseField$-definable nonempty
    Zariski open
    subset~$\ZariskiOpenSet_{\indexa,\PrimeIdeal{\empty}}$
    in~${\ZariskiOpenSet_{0} \cap \ZariskiOpenSet_{1} \cap
      \ZariskiOpenSet_{2}}$ such that for any
    point~$\SpecializationPoint$
    in~${\ZariskiOpenSet_{\indexa,\PrimeIdeal{\empty}} \cap
      \BaseField^{\SpecializationVariableSetCardinal}}$ and for any
    irreducible component~$\Variety{C}$
    of~${\pi^{-1}(\SpecializationPoint) \cap
      V(\PrimeIdeal{\empty})}$ (necessarily of
    dimension~${d- \SpecializationVariableSetCardinal}$), we have
    that~$\LeadingCoefficient({\Polynomial{\indexa}})$ does not vanish
    identically over~$\Variety{C}$.}
  \par\medskip\noindent
  Let us prove this claim. To do so, consider the
  set~${{V}(\PrimeIdeal{\empty})\cap
    {V}(\LeadingCoefficient({\Polynomial{\indexa}}))}$
  in~$\AffineSpace{\StateVariableSetCardinal}$. If this set is empty
  there is no component of $\pi^{-1}(\SpecializationPoint)\cap V(\PrimeIdeal{\empty})$ contained
  in~${V}(\LeadingCoefficient({\Polynomial{\indexa}}))$ and
  the claim follows. On the other hand, if the
  set~${{V}(\PrimeIdeal{\empty})\cap
    {V}(\LeadingCoefficient({\Polynomial{\indexa}}))}$ is
  nonempty, since we have the
  injection~${\PolynomialRing{\BaseField}{\TranscendenceBasis,
      \SpecializationVariableSet} \hookrightarrow
    \PolynomialRing{\BaseField}{\StateVariableSet}/\PrimeIdeal{\empty}}$, then $\LeadingCoefficient({\Polynomial{\indexa}})$ is not a zero divisor modulo $\PrimeIdeal{\empty}$ and so
  this algebraic set must be an equidimensional algebraic variety of dimension~${\dim
   {V}(\PrimeIdeal{\empty})-1=d-1}$.  We consider two cases:
  \begin{itemize}
  \item If the
    relation~${\overline{\pi({V}(\PrimeIdeal{\empty})\cap
        V(\LeadingCoefficient({\Polynomial{\indexa}})))}=
      \AffineSpace{\SpecializationVariableSetCardinal}}$ holds: from
    the theorem on the dimension of fibers applied to the restriction
    of the projection~$\pi$ to~${{V}(\PrimeIdeal{\empty})\cap
      {V}(\LeadingCoefficient({\Polynomial{\indexa}}))}$, there
    exists a~$\BaseField$-definable nonempty Zariski
    open~$\ZariskiOpenSet_{\indexa,\PrimeIdeal{\empty}}$
    in~$\AffineSpace{\SpecializationVariableSetCardinal}$, that can be
    assumed contained in~${\ZariskiOpenSet_{0} \cap
      \ZariskiOpenSet_{1} \cap \ZariskiOpenSet_{2}}$, such that, for
    any point~$\Variety{W}$
    in~${\ZariskiOpenSet_{\indexa,\PrimeIdeal{\empty}}\cap
      \BaseField^{\SpecializationVariableSetCardinal}}$, any component
    of the fiber~${\pi^{-1}(\Variety{W})\cap
      {V}(\PrimeIdeal{\empty})\cap
      {V}(\LeadingCoefficient({\Polynomial{\indexa}}))}$ has
    dimension~${\dim{{V}(\PrimeIdeal{\empty})\cap
        {V}(\LeadingCoefficient({\Polynomial{\indexa}}))}-
      \SpecializationVariableSetCardinal=d-1-
      \SpecializationVariableSetCardinal}$.
    \par
    Now, if~$\Variety{C}$ is an irreducible component
    of~$\pi^{-1}(\Variety{W})\cap {V}(\PrimeIdeal{\empty})$,
    since~$\Variety{W}$ is in~$\ZariskiOpenSet_{2}$ the
    relation~${\dim \Variety{C}=d-
      \SpecializationVariableSetCardinal}$
    holds. If~$\LeadingCoefficient({\Polynomial{\indexa}})$ vanishes
    identically over~$\Variety{C}$, then~$\Variety{C}$ is contained
    in~${\pi^{-1}(\Variety{W}) \cap
      {V}(\PrimeIdeal{\empty})\cap
      {V}(\LeadingCoefficient({\Polynomial{\indexa}}))}$ that is
    a~${(d-1- \SpecializationVariableSetCardinal)}$-dimensional
    variety. So~$\Variety{C}$ is not contained
    in~${V}(\LeadingCoefficient({\Polynomial{\indexa}}))$.
  \item If~$\overline{{\pi({V}(\PrimeIdeal{\empty})\cap
        {V}(\LeadingCoefficient({\Polynomial{\indexa}})))}}$ is
    a proper subset
    of~$\AffineSpace{\SpecializationVariableSetCardinal}$: define the
    set~$\ZariskiOpenSet_{\indexa,\PrimeIdeal{\empty}}$ to be the open
    Zariski set~${\ZariskiOpenSet_{0} \cap \ZariskiOpenSet_{1} \cap
      \ZariskiOpenSet_{2} \cap
      \overline{\pi({V}(\PrimeIdeal{\empty})\cap
        {V}(\LeadingCoefficient({\Polynomial{\indexa}})))}^{\textrm{C}}}$. Clearly~${\pi^{-1}(\Variety{W})
      \cap {V}(\PrimeIdeal{\empty})\cap
      {V}(\LeadingCoefficient({\Polynomial{\indexa}}))}$ is
    empty and so~$\ZariskiOpenSet_{\indexa,\PrimeIdeal{\empty}}$
    works.
  \end{itemize}
  Hence, our Claim is proved.  In order to finish the proof of the
  lemma, consider the nonempty Zariski open
  set~${\ZariskiOpenSet:={\bigcap_{\indexa,\PrimeIdeal{\empty}\in
        \IdealSet} \ZariskiOpenSet_{\indexa,\PrimeIdeal{\empty}}}}$.
  It suffices to prove that for any~$\SpecializationPoint$
  in~$\ZariskiOpenSet$ and any isolated primary
  component~$\mathfrak{q}$
  of~${\Ideal{\empty}+({\SpecializationVariableSet-\SpecializationPoint})}
  \subset \PolynomialRing{\BaseField}{\StateVariableSet}$,
  the
  relation~${\PolynomialRing{\BaseField}{\TranscendenceBasis} \cap
    \mathfrak{q} =\{0\}}$ holds.  The primary ideal~$\mathfrak{q}$
  defines an irreducible component~$\Variety{C}$ of the
  fiber~${\pi^{-1}(\SpecializationPoint)\cap
    {V}(\Ideal{\empty})}$; therefore,~$\Variety{C}$ is a
  subset of~${V}(\PrimeIdeal{\empty})$ for
  some~$\PrimeIdeal{\empty}$ in~$\IdealSet$
  (since~$\SpecializationPoint$ is in~$\ZariskiOpenSet_{0}$) and the
  relation~${\dim \mathfrak{q}=d-
    \SpecializationVariableSetCardinal}$ holds because
  $\SpecializationPoint$ is in~$\ZariskiOpenSet_{2}$. Hence we have
  the sequence of natural
  morphisms~\[{\PolynomialRing{\BaseField}{\TranscendenceBasis,
      \SpecializationVariableSet} \hookrightarrow
    \PolynomialRing{\BaseField}{\StateVariableSet}/\PrimeIdeal{\empty}
    \rightarrow
    \PolynomialRing{\BaseField}{\StateVariableSet}/\sqrt{\mathfrak{q}}=
    \PolynomialRing{\BaseField}{\Variety{C}}},\]
    where the last
  morphism is the projection to the quotient; in particular if we
  call~$\phi$ the composition of the morphisms we have
  ~${\phi(\SpecializationVariableSet)=\SpecializationPoint}$ and the coordinate
  ring~$\PolynomialRing{\BaseField}{\Variety{C}}$ is generated as
  a~$\BaseField$-algebra by~$\phi(\TranscendenceBasis)$ and the class
  of the variables~$\StateVariableSubset$. From the definition
  of~$\ZariskiOpenSet$ and the previous Claim, it follows that the
  class of each~$\StateVariableSubsetElement{\indexa}$ is algebraic
  over the
  subring~$\PolynomialRing{\BaseField}{\phi(\TranscendenceBasis)}$
  of~$\PolynomialRing{\BaseField}{\Variety{C}}$ and so, the
  relation~${d- \SpecializationVariableSetCardinal=\dim \Variety{C}\le
    \card \TranscendenceBasis=d- \SpecializationVariableSetCardinal}$
  holds. Hence~$\phi$ is a monomorphism and in
  particular~$\mathfrak{q}$ does not contain polynomials
  in~${\PolynomialRing{\BaseField}{\TranscendenceBasis}
    \setminus\{0\}}$.
\end{proof}
\begin{corollary}
  Let~$\Ideal{\empty}$ be a radical equidimensional ideal
  in~$\PolynomialRing{\BaseField}{\StateVariableSet}$, and
  let~$\TranscendenceBasis$ and~$\SpecializationVariableSet$ be
  two disjoint subsets of variables
  such that the natural
  morphism~${\PolynomialRing{\BaseField}{\TranscendenceBasis,
      \SpecializationVariableSet} \rightarrow
    \PolynomialRing{\BaseField}{\StateVariableSet}/\Ideal{\empty}}$ is
  injective. Then, there exists a nonempty~$\BaseField$-definable
  Zariski open set~$\ZariskiOpenSet$
  in~$\AffineSpace{\SpecializationVariableSetCardinal}$, where $\SpecializationVariableSetCardinal$ is the cardinality of $W$,  verifying the
  previous lemma and such that, for every point~$\SpecializationPoint$
  in~${\ZariskiOpenSet \cap
    \BaseField^{\SpecializationVariableSetCardinal}}$, the
  morphism~${\PolynomialRing{\BaseField}{\TranscendenceBasis}
    \rightarrow
    \PolynomialRing{\BaseField}{\StateVariableSet}/
    (\Ideal{\empty}+({\SpecializationVariableSet-\SpecializationPoint}))}$
  is injective. \end{corollary}
\begin{proof}
  \label{U}
  From Theorem~\ref{hp} we are able to apply Lemma~\ref{lemma2} to
  the field~$\widetilde{\BaseField}:=\BaseField(\TranscendenceBasis)$ and the ideal~$\widetilde{\Ideal{\empty}}:={\Ideal{\empty} \otimes\widetilde{\BaseField}} \subset\PolynomialRing{\widetilde{\BaseField}}{\StateVariableSet
  \setminus \TranscendenceBasis}$. Therefore we obtain a nonempty
  Zariski open subset~$\widetilde{\ZariskiOpenSet}$
  of~$\AffineSpace{\SpecializationVariableSetCardinal}_{\widetilde{\BaseField}}$. On
  the other hand, we may also apply Lemma~\ref{lemma2} to the
  ideal~$\Ideal{\empty}$ and the
  variables~$\SpecializationVariableSet$ over the ground
  field~$\BaseField$, obtaining an open set~$\ZariskiOpenSet_{0}$.
  \par
  Now take~$\ZariskiOpenSet$ an arbitrary~$\BaseField$-definable
  nonempty Zariski open set contained in~${\widetilde{\ZariskiOpenSet}
    \cap \ZariskiOpenSet_{0}}$. It suffices to see
  that~$\PolynomialRing{\BaseField}{\TranscendenceBasis}$ is included
  in the ring~${\PolynomialRing{\BaseField}{\StateVariableSet}/
    (\Ideal{\empty}+({\SpecializationVariableSet-\SpecializationPoint}))}$
  for~$\SpecializationPoint$ in~$\ZariskiOpenSet$, which is immediate
  from the fact that~$\SpecializationPoint$ is
  in~$\widetilde{\ZariskiOpenSet}$ and, in particular, there exists a
  component of maximal dimension
  of~${\Ideal{\empty}+({\SpecializationVariableSet -
      \SpecializationPoint})}$ which contains no nonzero polynomial
  pure in the variables~$\TranscendenceBasis$.
\end{proof}
If the ideal~$\Ideal{\empty}$ is prime, the proof of
Lemma~\ref{lemma2} allows us to show a more precise version of the
previous Corollary:
\begin{corollary}
  \label{intersection}
  Let~$\Ideal{\empty}$ be a prime ideal
  in~$\PolynomialRing{\BaseField}{\StateVariableSet}$ and
  let~$\TranscendenceBasis$ and~$\SpecializationVariableSet$ be
  two disjoint subsets of variables
  such that $\{{\TranscendenceBasis,
    \SpecializationVariableSet}\}$ is a transcendence basis of the
  fraction field
  of~${\PolynomialRing{\BaseField}{\StateVariableSet}/\Ideal{\empty}}$. Then,
  there exists a nonempty~$\BaseField$-definable Zariski open
  set~$\ZariskiOpenSet\subset \AffineSpace{\SpecializationVariableSetCardinal}$ such that, for
  every point~$\SpecializationPoint$ in~${\ZariskiOpenSet \cap
    \BaseField^{\SpecializationVariableSetCardinal}}$, the
  ideal~${\Ideal{\empty}|_{\SpecializationPoint}}$ has
  dimension~${\dim\Ideal{\empty}- \SpecializationVariableSetCardinal}$
  and its primary components of maximal dimension are prime ideals
  containing no nonzero polynomial pure in the
  variables~$\TranscendenceBasis$.
\end{corollary}
\begin{proof}
  Simply observe that in the proof of Lemma~\ref{lemma2} the subset of
  associated primes~$\IdealSet$ is the unitary
  set~$\{\Ideal{\empty}\}$ if~$\Ideal{\empty}$ is assumed to be a
  prime ideal. Therefore the Claim and the remaining part of the proof
  run \textit{mutatis mutandis}.
\end{proof}
Now, we can prove the main result of this appendix:
\begin{theorem}
  \label{esp_reg_seq}
  Let~${\StateVariableSetCardinal \ge 2}$ and~$\StateVariableSet$
  a set of $n$ indeterminates over a
  field~$\BaseField$ of characteristic~$0$. Let~${\{\TheEquation{1},
    \ldots, \TheEquation{r}\}}$ be a reduced regular sequence
  in~$\PolynomialRing{\BaseField}{\StateVariableSet}$ (that is, a
  regular sequence such that the
  ideals~$({\TheEquation{1},\ldots,\TheEquation{\indexa}})$
  in~$\PolynomialRing{\BaseField}{\StateVariableSet}$ are radical
  for~${\indexa=1,\ldots,r}$). Let~$\SpecializationVariableSet$
  be a subset of $s$ many variables in~$\StateVariableSet$, with~${\SpecializationVariableSetCardinal <
    \StateVariableSetCardinal}$, such that the canonical
  map~${\PolynomialRing{\BaseField}{\SpecializationVariableSet}
    \rightarrow \PolynomialRing{\BaseField}{\StateVariableSet}/
    ({\TheEquation{1}, \ldots, \TheEquation{r}})}$ is
  injective. Let us denote by~$\StateVariableSubset$ the set of
  remaining variables~${\StateVariableSet \setminus
    \SpecializationVariableSet}$.
  \par
  Then, there exists a nonempty~$\BaseField$-definable Zariski open
  set~$\ZariskiOpenSet\subset \AffineSpace{\SpecializationVariableSetCardinal}$ such that, for
  every point~$\SpecializationPoint$ in~${\ZariskiOpenSet \cap
    \BaseField^{\SpecializationVariableSetCardinal}}$ and all~${
    \indexa =1,\ldots,r}$, the
  polynomials~${\TheEquation{1}(\SpecializationPoint,\StateVariableSubset),
    \ldots, \TheEquation{\indexa}(\SpecializationPoint,
    \StateVariableSubset)}$ form a reduced regular sequence in the
  polynomial ring~$\PolynomialRing{\BaseField}{\StateVariableSubset}$.
\end{theorem}
\begin{proof}
  We prove this theorem by recurrence in~$r$.
  \par
  If~${r=1}$, since the polynomial $\TheEquation{1}$ is assumed to be
  square-free, we take~$\ZariskiOpenSet$ as the projection
  of~${\{\discr_{\StateVariableSubsetElement{\empty}}(\TheEquation{1})\ne
    0\}}$ to $\AffineSpace{\SpecializationVariableSetCardinal}$,
  where~$\StateVariableSubsetElement{\empty}$ is any variable
  in~$\StateVariableSubset$ which appears in~$\TheEquation{1}$.
  \par
  Assume the result holds for an integer~${r\ge
    1}$. Let~${\TheEquation{1}, \ldots, \TheEquation{r+1}}$ be a
  regular sequence in~$\PolynomialRing{\BaseField}{\StateVariableSet}$
  such that the
  ideals~$({\TheEquation{1},\ldots,\TheEquation{\indexa}})$ are
  radical for~${\indexa=1,\ldots,{r+1}}$ and
  let~$\SpecializationVariableSet$ be a subset of~$\StateVariableSet$
  such that the canonical morphism ~${\PolynomialRing{\BaseField}{\SpecializationVariableSet}
    \hookrightarrow \PolynomialRing{\BaseField}{\StateVariableSet}/
    ({\TheEquation{1}, \ldots, \TheEquation{r+1}})}$ is injective. In
  particular, ${\PolynomialRing{\BaseField}{\SpecializationVariableSet}
    \hookrightarrow \PolynomialRing{\BaseField}{\StateVariableSet}/
    ({\TheEquation{1},\ldots, \TheEquation{r}})}$ is injective
  too.  Hence, by the induction hypothesis, there exists a
  nonempty~$\BaseField$-definable Zariski open
  subset~$\ZariskiOpenSet_{1}$
  of~$\AffineSpace{\SpecializationVariableSetCardinal}$ such that, for
  any point~$\AffinePoint{P}_{1}$ in~$\ZariskiOpenSet_{1}$ and for
  any~${\indexa=1,\ldots ,r}$, the
  polynomials~${\TheEquation{1}(\AffinePoint{P}_{1},\StateVariableSubset),
    \ldots,
    \TheEquation{\indexa}(\AffinePoint{P}_{1},\StateVariableSubset)}$
  form a regular sequence which generates a radical ideal
  in~$\PolynomialRing{\BaseField}{\StateVariableSubset}$.
  \par
  From Macaulay's unmixedness theorem (see for instance
  \cite[Chapter~VI, \S3, Theorem~3.14]{kunz:1985}), the
  ideal~${\Ideal{\empty}:=(\TheEquation{1}, \ldots,
    \TheEquation{r+1})}$ is equidimensional of
  dimension~${\StateVariableSetCardinal-(r+1)}$ and so, the hypotheses
  of Lemma~\ref{lemma2} are met for the ideal $I$ and the
  variables~$\SpecializationVariableSet$. Then, there exists
  a~$\BaseField$-definable nonempty Zariski open
  subset~$\ZariskiOpenSet_{2}$
  of~$\AffineSpace{\SpecializationVariableSetCardinal}$, such that for
  any point~$\AffinePoint{P}_{2}$ in~${\ZariskiOpenSet_{2}\cap
    \BaseField^{\SpecializationVariableSetCardinal}}$, the primary
  components of maximal dimension of~$({{\TheEquation{1}(\AffinePoint{P}_{2},
      \StateVariableSubset), \ldots,
      \TheEquation{r+1}(\AffinePoint{P}_{2}, \StateVariableSubset)}})\subset \PolynomialRing{\BaseField}{\StateVariableSubset}$
  are prime ideals of dimension~${\dim \Ideal{\empty} -
    \SpecializationVariableSetCardinal=\StateVariableSetCardinal-(r+1)-\SpecializationVariableSetCardinal}$.
  \par
  Take $\ZariskiOpenSet:=\ZariskiOpenSet_{1} \cap \ZariskiOpenSet_{2}$. We will show that this open set verifies the statement of the theorem for~${r+1}$.  Let~$\SpecializationPoint$ be a point in~${\ZariskiOpenSet \cap
    \BaseField^{\SpecializationVariableSetCardinal}}$. Since~$\TheEquation{1}(\SpecializationPoint,\StateVariableSubset),\ldots,
  \TheEquation{r}(\SpecializationPoint,\StateVariableSubset)$ is a
  regular sequence that generates an equidimensional radical ideal
  (because~$\SpecializationPoint$ is in~$\ZariskiOpenSet_{1}$), in
  order to prove
  that~$\TheEquation{r+1}(\SpecializationPoint,\StateVariableSubset)$
  is not a zero divisor
  modulo~${(\TheEquation{1}(\SpecializationPoint,\StateVariableSubset),
    \ldots,
    \TheEquation{r}(\SpecializationPoint,\StateVariableSubset))}$, it
  suffices to show that the dimension drops by~$1$ when this
  polynomial is added. This follows directly from the fact
  that~$\SpecializationPoint$ also belongs
  to~$\ZariskiOpenSet_2$. Therefore~${\TheEquation{1}(\SpecializationPoint,\StateVariableSubset),
    \ldots,
    \TheEquation{r+1}(\SpecializationPoint,\StateVariableSubset)}$ is
  a regular sequence
  in~$\PolynomialRing{\BaseField}{\StateVariableSubset}$. In
  particular, the generated ideal is unmixed by Macaulay's theorem,
  and then it is radical since~$\SpecializationPoint\in\ZariskiOpenSet_{2}$.
\end{proof}
\section{Existence and uniqueness of solutions}
\label{app:exuniq}
Several previous articles consider the problem of the existence and
uniqueness of solutions for \textit{first order}
implicit \DifferentialAlgebraicSystem\ systems (see for
instance~\cite{rabier:1994, pritchard:2003, pritchard:2007,
  dalfonso:2009}). By adding new variables for the higher order
derivatives in the usual way, these results can be extended to
\DifferentialAlgebraicSystem\ systems of arbitrary order. For
instance, let~$\DifferentialSystem$ be the
\DifferentialAlgebraicSystem\ system introduced in
Section~\ref{sistema_original} assuming that the hypotheses of
Sections~\ref{basic} and~\ref{defindex} are fulfilled and that~$\BaseField$ is a
subfield of~$\Field{C}$. Then we have the following generalization
of~\cite[Theorem 24]{dalfonso:2009}:
\begin{theorem}
  \label{existeunaordene}
  Let $\Variety{V}_{0}\subset \AffineSpace{\StateVariableSetCardinal
    \DifferentialSystemOrder{\empty}}$ and $\Variety{V}_{1}\subset\AffineSpace{
    \StateVariableSetCardinal(\DifferentialSystemOrder{\empty}+1)}$ be the algebraic varieties defined by
  the ideals~${\DifferentialIdeal{\EquationSet} \cap
    \ProlongatedRing{\DifferentialSystemOrder{\empty}-1}}$
  and~${\DifferentialIdeal{\EquationSet} \cap
    \ProlongatedRing{\DifferentialSystemOrder{\empty}}}$ respectively,
  and let~${\pi: \Variety{V}_{1}\rightarrow \Variety{V}_{0}}$ be the
  projection to the first~${\StateVariableSetCardinal
    \DifferentialSystemOrder{\empty}}$ coordinates.
  \par
  Then, for every regular point~$\AffinePoint{\StateVariableSet}:={(\AffinePoint{\StateVariableSet}_{0}, \dots,
    \AffinePoint{\StateVariableSet}_{\DifferentialSystemOrder{\empty}-1},
    \AffinePoint{\StateVariableSet}_{\DifferentialSystemOrder{\empty}})}$
  in~$\Variety{V}_{1}$ where the projection~$\pi$ is unramified, there
  exist~${\varepsilon >0}$, a relative open
  neighborhood~$\ZariskiOpenSet\subset\Variety{V}_{1}$  of
  $\AffinePoint{\StateVariableSet}$ and a unique analytic
  function ${\AnalyticSolutionSet:
    \OpenInterval{-\varepsilon}{\varepsilon} \to
    \Field{C}^{\StateVariableSetCardinal}}$ which is a solution
  of~$\DifferentialSystem$ with initial
  condition $${\left(\AnalyticSolutionSet(0), \ldots,
      \AnalyticSolutionSet^{(\DifferentialSystemOrder{\empty}-1)}(0)\right)=(\AffinePoint{\StateVariableSet}_{0},
    \ldots, \AffinePoint{\StateVariableSet}_{\DifferentialSystemOrder{\empty}-1})}$$ such
  that~${\left(\AnalyticSolutionSet(t), \dots,
      \AnalyticSolutionSet^{(\DifferentialSystemOrder{\empty}-1)}(t),\AnalyticSolutionSet^{(\DifferentialSystemOrder{\empty})}(t)
    \right)}$ is in~$\ZariskiOpenSet$ for all~$t$.
\end{theorem}
\begin{proof}
  We make a straightforward change of variables in order to obtain an
  equivalent first-order system: for each~$\indexa\in{\{0,\dots,
    \DifferentialSystemOrder{\empty}-1\}}$, consider a new
  set~$\NewVariableSet_{\indexa}$ of
  variables~${(\NewVariable{\indexa,1}, \dots,
    \NewVariable{\indexa,\StateVariableSetCardinal})}$ representing
  the derivatives~$\timederivative{\StateVariableSet}{\indexa}$, and
  let~$\FirstOrderDifferentialSystem$ be the first-order
  \DifferentialAlgebraicSystem\ system
  \begin{equation}
    \label{orden1}
    \FirstOrderDifferentialSystem :=
    \left\{
      \begin{array}[c]{ccl}%
        \NewVariable{\indexa,\indexb}-
        \timederivative{\NewVariable{\indexa-1,\indexb}}{1}
        &=& 0,\qquad \indexa=1,\dots,
        \DifferentialSystemOrder{\empty}-1,\quad \indexb=1,\dots,
        \StateVariableSetCardinal,
        \\[\medskipamount]
        \TheEquation{1}(\NewVariableSet, \timederivative{\NewVariableSet}{1}) &=& 0, \\
        &\vdots &\\
        \TheEquation{\StateVariableSetCardinal}(\NewVariableSet, \timederivative{\NewVariableSet}{1}) &=& 0, \\
      \end{array}
    \right.
  \end{equation}
  where~$\NewVariableSet$ denotes~${\NewVariableSet_{0}, \dots,
    \NewVariableSet_{\DifferentialSystemOrder{\empty}-1}}$. We apply
  now to this system the existence and uniqueness result
  in~\cite[Theorem 24]{dalfonso:2009}, which holds in the
  case~${\DifferentialSystemOrder{\empty}=1}$. In order to do so, let
  us verify that the required assumptions hold.
  \par
  Denote by~$\mathfrak{A}$ the differential ideal associated with the   system~$\FirstOrderDifferentialSystem$  and consider
  the map~${\Upsilon :
    \DifferentialPolynomialRing{\BaseField}{\StateVariableSet} \to
    \DifferentialPolynomialRing{\BaseField}{\NewVariableSet}}$ defined
  by
  \[
  \Upsilon(\timederivative{\StateVariable{\indexa}}{\indexb})
  =
  \begin{cases}
    \NewVariable{\indexa,\indexb} & \ \textrm{if}\ \indexa < \DifferentialSystemOrder{\empty},\\
    \timederivative{\NewVariable{\DifferentialSystemOrder{\empty}-1,\indexb}}{\indexa-\DifferentialSystemOrder{\empty}+1}
    &\ \textrm{if}\ \indexa\ge \DifferentialSystemOrder{\empty}.
  \end{cases}
  \]
  Note that~$\Upsilon$ is an injection that
  maps~${\DifferentialIdeal{\EquationSet} =[\TheEquation{1},\dots,
    \TheEquation{\StateVariableSetCardinal}]}$ to~$\mathfrak{A}$.
    For
  each differential polynomial~$\TheEquation{\empty}$
  in~$\DifferentialPolynomialRing{\BaseField}{\StateVariableSet}$, the
  expression $
  \Upsilon({\dot{\TheEquation{\empty}}})-(\Upsilon({\TheEquation{\empty}}))'$ {belongs to
    the differential ideal} $\DifferentialIdeal{
    \NewVariable{\indexa,\indexb} -
    \timederivative{\NewVariable{\indexa-1,\indexb}}{1}\  ;\  1\le \indexa
    \le \DifferentialSystemOrder{\empty}-1,\ 1\le \indexb \le
    \StateVariableSetCardinal}.$

  This implies that the relation~${\mathfrak{A} =
    \Upsilon(\DifferentialIdeal{\EquationSet}) + \DifferentialIdeal{
      \NewVariable{\indexa,\indexb} -
      \timederivative{\NewVariable{\indexa-1,\indexb}}{1}\ ;\  1\le
      \indexa \le \DifferentialSystemOrder{\empty}-1, 1\le \indexb \le
      \StateVariableSetCardinal}}$ holds. In particular, we have the
  isomorphism~${\DifferentialPolynomialRing{\BaseField}{\StateVariableSet}
    /\DifferentialIdeal{\EquationSet} \simeq
    \DifferentialPolynomialRing{\BaseField}{\NewVariableSet} /
    \mathfrak{A}}$ and then, if~$\DifferentialIdeal{\EquationSet}$ is a
  prime ideal, then so is~$\mathfrak{A}$.
  \par
  Moreover, we have the identities:
  \[
  \begin{array}{lcl}
    \mathfrak{A} \cap
    \PolynomialRing{\BaseField}{\NewVariableSet} &=&
    \Upsilon \left(\DifferentialIdeal{\EquationSet} \cap
      \ProlongatedRing{\DifferentialSystemOrder{\empty}-1} \right),\\
    \mathfrak{A} \cap \PolynomialRing{\BaseField}{\NewVariableSet,
      \timederivative{\NewVariableSet}{1}} &=&
    ({\NewVariable{\indexa,\indexb}-
      \timederivative{\NewVariable{\indexa-1,\indexb}}{1}\ ;
      \ 1\le \indexa \le \DifferentialSystemOrder{\empty}-1,\ 1\le \indexb \le \StateVariableSetCardinal})+
    \Upsilon \left(\DifferentialIdeal{\EquationSet} \cap
      \ProlongatedRing{\DifferentialSystemOrder{\empty}} \right).
  \end{array}
  \]
  Therefore, if the
  polynomials~$\mathtt{F}_{\DifferentialSystemOrder{\empty}-1}$
  in~$\ProlongatedRing{\DifferentialSystemOrder{\empty}-1}$
  and~$\mathtt{F}_{\DifferentialSystemOrder{\empty}}$
  in~$\ProlongatedRing{\DifferentialSystemOrder{\empty}}$ are
  generators of~${\DifferentialIdeal{\EquationSet} \cap
    \ProlongatedRing{\DifferentialSystemOrder{\empty}-1}}$
  and~${\DifferentialIdeal{\EquationSet} \cap
    \ProlongatedRing{\DifferentialSystemOrder{\empty}}}$ respectively,
  then~$\Upsilon(\mathtt{F}_{\DifferentialSystemOrder{\empty}-1})$ and
  \[{\mathtt{G} := \{ \NewVariable{1,1}-
    \timederivative{\NewVariable{0,1}}{1},\dots,
    \NewVariable{1,\StateVariableSetCardinal}-
    \timederivative{\NewVariable{0,\StateVariableSetCardinal}}{1},
    \ldots,
    \timederivative{\NewVariable{\DifferentialSystemOrder{\empty}-2,1}}{1}
    - \NewVariable{\DifferentialSystemOrder{\empty}-1,1}, \ldots,
    \timederivative{\NewVariable{\DifferentialSystemOrder{\empty}-2,\StateVariableSetCardinal}}{1}
    -
    \NewVariable{\DifferentialSystemOrder{\empty}-1,\StateVariableSetCardinal}\}
    \cup \Upsilon(\mathtt{F}_{\DifferentialSystemOrder{\empty}})}\]
      are generators of~${\mathfrak{A} \cap
    \PolynomialRing{\BaseField}{\NewVariableSet}}$ and~${\mathfrak{A}
    \cap \PolynomialRing{\BaseField}{\NewVariableSet,
      \timederivative{\NewVariableSet}{1}}}$ respectively.

  In
  particular, the Jacobian
  submatrix~$\tPartialDerivative{\mathtt{G}}{\timederivative{\NewVariableSet}{1}}{1}$
  has the block form
  \[
  \left(
    \begin{array}{cc}
      -\textrm{Id}_{(\DifferentialSystemOrder{\empty}-1) \StateVariableSetCardinal} & 0\\
      0 & \Upsilon\left(D_{\timederivative{\StateVariableSet}{\DifferentialSystemOrder{\empty}}} \mathtt{F}_{\DifferentialSystemOrder{\empty}}\right)
    \end{array}
  \right)\!.
  \]
  Thus, if ${V}({{\mathfrak{A} \cap
      \PolynomialRing{\BaseField}{\NewVariableSet,
        \timederivative{\NewVariableSet}{1}}}})$
  in~$\AffineSpace{2\StateVariableSetCardinal
    \DifferentialSystemOrder{\empty}}$ and ${V}({\mathfrak{A} \cap
    \PolynomialRing{\BaseField}{\NewVariableSet}})$ in~$\AffineSpace{
    \StateVariableSetCardinal \DifferentialSystemOrder{\empty}}$
  are the varieties defined by
  the specified ideals,
  and~$\widetilde{\AffinePoint{\StateVariableSet}}$ is the point
  in~${V}({{\mathfrak{A} \cap
      \PolynomialRing{\BaseField}{\NewVariableSet,
        \timederivative{\NewVariableSet}{1}}}})$ corresponding
  to~$\AffinePoint{\StateVariableSet}\in\Variety{V}_{1}$, the
  unramifiedness of the projection~${\pi : \Variety{V}_{1}\to
    \Variety{V}_{0}}$ at~$\AffinePoint{\StateVariableSet}$
  % implies that $D_{X^{(e)}} H(\AffinePoint{\StateVariableSet})$ has
  % full column rank $n$ and then, $D_{\dot Z}\mathtt{G}
  % (\psi(\AffinePoint{\StateVariableSet}))$ has full column rank $ne$
  % and thus, the projection $\widetilde \pi : \mathcal{V}_1 \to
  % \mathcal{V}_0$ is unramified at
  % $\psi(\AffinePoint{\StateVariableSet})$.
  is equivalent to the unramifiedness of the
  projection~${\widetilde{\pi}:  {V}({{\mathfrak{A} \cap
        \PolynomialRing{\BaseField}{\NewVariableSet,
          \timederivative{\NewVariableSet}{1}}}})} \to
    {V}({\mathfrak{A} \cap
      \PolynomialRing{\BaseField}{\NewVariableSet}})$
  at~$\widetilde{\AffinePoint{\StateVariableSet}}$. Similarly, the
  fact that~$\AffinePoint{\StateVariableSet}$ is a regular point
  of~$\Variety{V}_{1}$ implies
  that~$\widetilde{\AffinePoint{\StateVariableSet}}$ is a regular
  point of~$
    {V}({{\mathfrak{A} \cap
        \PolynomialRing{\BaseField}{\NewVariableSet,
          \timederivative{\NewVariableSet}{1}}}})$.
\end{proof}
\begin{remark}
  \label{exunquasireg}
  In the case of first order \DifferentialAlgebraicSystem\ systems,
  the existence and uniqueness of solutions as stated
  in~\cite[Theorem~24]{dalfonso:2009} can also be extended to the case
  when the ideal~$\DifferentialIdeal{\EquationSet}$ is not
  prime, but the system~$\DifferentialSystem$ is quasi-regular.  The
  result follows as in the proof of~\cite[Theorem~24]{dalfonso:2009}: if ${\PrimeIdeal{\empty}}$ denotes a minimal
  prime differential ideal of~$\DifferentialIdeal{\EquationSet}$, then
  ${\PrimeIdeal{\empty}}$ plays the same role as the
  ideal~$\mathfrak{Q}$ in that proof and we can take~${d =
    \ord({{\PrimeIdeal{\empty}}})}$.
\end{remark}
\begin{small}
  \bibliographystyle{acm} \bibliography{DAJOSS02}

\def\gathen#1{{#1}}\def\cprime{$'$} \def\gathen#1{{#1}}\def\cprime{$'$}
\begin{thebibliography}{10}

\bibitem{bostan:2007}
{\sc Bostan, A., Chyzak, F., Ollivier, F., Schost, {\'E}., Salvy, B., and
  Sedoglavic, A.}
\newblock Fast computation of power series solutions of systems of differential
  equations.
\newblock In {\em Proceedings of the 18th ACM-SIAM Symposium on Discrete
  Algorithms\/} (New Orleans, Louisiana, USA, 2007), pp.~1012 -- 1021.

\bibitem{blop:2009}
{\sc Boulier, F., Lazard, D., Ollivier, F., and Petitot, M.}
\newblock Computing representation for radicals of finitely generated
  differential ideals.
\newblock {\em AAECC 20}, 1 (2009), 73 -- 121.

\bibitem{brenan:1983}
{\sc Brenan, K.~E.}
\newblock {\em Stability and convergence for higher index differential
  algebraic equations with applications to trajectory control}.
\newblock PhD thesis, Department of Mathematics, University of California,
  1983.

\bibitem{brenan:1996}
{\sc Brenan, K.~E., Campbell, S. L.~V., and Petzold, L.~R.}
\newblock {\em Numerical Solution of Initial-Value Problems in
  Differential-Algebraic Equations}.
\newblock No.~14 in SIAM Classics in Applied Mathematics. Society for
  Industrial and Applied Mathematics, 1996.

\bibitem{campbell:1995}
{\sc Campbell, S. L.~V., and Gear, C.~W.}
\newblock The index of general nonlinear {DAE}s.
\newblock {\em Numerische Mathematik 72}, 2 (Dec. 1995), 173--196.

\bibitem{CluzeauH:2003}
{\sc Cluzeau, T., and Hubert, E.}
\newblock Resolvent representation for regular differential ideals.
\newblock {\em Appl. Algebra Eng. Commun. Comput. 13}, 5 (2003), 395--425.

\bibitem{CluzeauH:2008}
{\sc Cluzeau, T., and Hubert, E.}
\newblock Probabilistic algorithms for computing resolvent representations of
  regular differential ideals.
\newblock {\em Appl. Algebra Eng. Commun. Comput. 19}, 5 (2008), 365--392.

\bibitem{corless-ilie:2009}
{\sc Corless, R.~M., and Ilie, S.}
\newblock Polynomial cost for solving ivp for high-index dae.
\newblock {\em BIT Numerical Mathematics 48\/} (2008), 29 -- 49.

\bibitem{dahan:2008}
{\sc Dahan, X., Jin, X., Moreno~Maza, M., and Schost, E.}
\newblock Change of order for regular chains in positive dimension.
\newblock {\em Theoretical Computer Science 392\/} (2008), 37--65.

\bibitem{dalfonso:2009}
{\sc D'Alfonso, L., Jeronimo, G., Massaccesi, G., and Solern{\'o}, P.}
\newblock On the index and the order of quasi-regular implicit systems of
  differential equations.
\newblock {\em Linear Algebra and its Applications 430}, 8-9 (Apr. 2009),
  2102--2122.

\bibitem{dalfonso:2006}
{\sc D'Alfonso, L., Jeronimo, G., and Solern{\'o}, P.}
\newblock On the complexity of the resolvent representation of some prime
  differential ideals.
\newblock {\em Journal of Complexity 22}, 3 (June 2006), 396--430.

\bibitem{dalfonso:2008}
{\sc D'Alfonso, L., Jeronimo, G., and Solern{\'o}, P.}
\newblock A linear algebra approach to the differentiation index of generic
  {DAE} systems.
\newblock {\em Applicable Algebra in Engineering, Communication and Computing
  19}, 6 (Dec. 2008), 441--473.

\bibitem{DuLe:2006:cpkpsss}
{\sc Durvye, C., and Lecerf, G.}
\newblock A concise proof of the kronecker polynomial system solver from
  scratch.
\newblock {\em Expositiones Mathematicae 26}, 2 (2007).
\newblock doi:10.1016/j.exmath.2007.07.001.

\bibitem{fliess:1995}
{\sc Fliess, M., L{\'e}vine, J., Martin, {\relax Ph}., and Rouchon, P.}
\newblock Implicit differential equations and {L}ie-{B}{\"a}cklund mappings.
\newblock In {\em Proceedings of the 34th IEEE Conference on Decision and
  Control\/} (New Orleans, Louisiana, USA, Dec. 1995), vol.~3, pp.~2704--2709.

\bibitem{gear:1988}
{\sc Gear, C.~W.}
\newblock Differential-algebraic equation index transformations.
\newblock {\em SIAM Journal on Scientific and Statistical Computing 9}, 1 (Jan.
  1988), 39--47.

\bibitem{gear:1989}
{\sc Gear, C.~W.}
\newblock {DAE}s: {ODE}s with constraints and invariants.
\newblock In {\em Numerical methods for ordinary differential equations\/}
  (L'Aquila, Italy, Sept. 1989), vol.~1386 of {\em Lecture Notes in
  Mathematics}, Springer, pp.~54--68.

\bibitem{giusti:2001}
{\sc Giusti, M., Lecerf, G., and Salvy, B.}
\newblock A {G}r{\"o}bner free alternative for polynomial system solving.
\newblock {\em Journal of Complexity 17}, 1 (Mar. 2001), 154--211.

\bibitem{heintz:2000}
{\sc Heintz, J., Krick, T., Puddu, S., Sabia, J., and Waissbein, A.}
\newblock Deformation techniques for efficient polynomial equation solving.
\newblock {\em Journal of Complexity 16}, 1 (Mar. 2000), 70--109.

\bibitem{hodge:1994}
{\sc Hodge, W. V.~D., and Pedoe, D.}
\newblock {\em Methods of algebraic geometry}, vol.~2.
\newblock Cambridge University Press, 1994.

\bibitem{jacobi:1865}
{\sc Jacobi, C. G.~J.}
\newblock De investigando ordine systematis aequationum differentialum
  vulgarium cujuscunque.
\newblock {\em Borchardt Journal f{\"u}r die reine und angewandte Mathematik
  LXIV}, 4 (1865), 297--320.
\newblock English translation in \cite{ollivier:2009}.

\bibitem{ollivier:2009}
{\sc Jacobi, C. G.~J.}
\newblock Looking for the order of a system of arbitrary ordinary differential
  equations.
\newblock {\em Applicable Algebra in Engineering, Communication and Computing
  20}, 1 (Apr. 2009), 7--32.
\newblock Translated from the Latin by F.~Ollivier.

\bibitem{johnson:1978}
{\sc Johnson, J.}
\newblock Systems of~$n$ partial differential equations in~$n$ unknown
  functions: the conjecture of {M.} {J}anet.
\newblock {\em Transactions of the American Mathematical Society 242\/} (Aug.
  1978), 329--334.

\bibitem{kolchin:1973}
{\sc Kolchin, E.~R.}
\newblock {\em Differential algebra and algebraic groups}, vol.~54 of {\em Pure
  and applied Mathematics}.
\newblock Academic press, New York, 1973.

\bibitem{kondratieva:1982}
{\sc Kondratieva, M.~V., Mikhalev, A.~V., and Pankratiev, E.~V.}
\newblock On {J}acobi's bound for systems of differential polynomials.
\newblock {\em Algebra 79\/} (1982), 75--85.

\bibitem{kondratieva:2009}
{\sc Kondratieva, M.~V., Mikhalev, A.~V., and Pankratiev, E.~V.}
\newblock Jacobi's bound for independent systems of algebraic partial
  differential equations.
\newblock {\em AAECC 20}, 1 (2009), 65 -- 71.

\bibitem{kunkel:2006}
{\sc Kunkel, P., and Mehrmann, V.~L.}
\newblock {\em Differential-algebraic equations. Analysis and numerical
  solution}.
\newblock European Mathematical Society, 2006.

\bibitem{kunz:1985}
{\sc Kunz, E.}
\newblock {\em Introduction to commutative algebra and algebraic geometry}.
\newblock Birkh{\"a}user, 1985.

\bibitem{lang:2002}
{\sc Lang, S.}
\newblock {\em Algebra}, 3~ed., vol.~211 of {\em Graduate Texts in
  Mathematics}.
\newblock Springer, 2002.

\bibitem{levey:1994}
{\sc Le~Vey, G.}
\newblock Differential algebraic equations a new look at the index.
\newblock Tech. Rep. 808, Institut de Recherche en Informatique et en
  Automatique, 1994.

\bibitem{lecerf:2003}
{\sc Lecerf, G.}
\newblock Computing the equidimensional decomposition of an algebraic closed
  set by means of lifting fibers.
\newblock {\em Journal of Complexity 19}, 4 (2003), 564 -- 596.

\bibitem{lotstedt:1986}
{\sc L{\"o}tstedt, P., and Petzold, L.~R.}
\newblock Numerical solution of nonlinear differential equations with algebraic
  constraints {I}: Convergence results for backward differentiation formulas.
\newblock {\em Mathematics of Computation 46}, 174 (Apr. 1986), 491--516.

\bibitem{matera:2003}
{\sc Matera, G., and Sedoglavic, A.}
\newblock Fast computation of discrete invariants associated to a differential
  rational mapping.
\newblock {\em Journal of Symbolic Computation 36}, 3--4 (Sept.-Oct. 2003),
  473--499.

\bibitem{matsumura:1970}
{\sc Matsumura, H.}
\newblock {\em Commutative Algebra}, 2~ed., vol.~56 of {\em Mathematics Lecture
  Note}.
\newblock WA Benjamin, 1970.

\bibitem{mattsson:1993}
{\sc Mattsson, S.~E., and S{\"o}derlind, G.}
\newblock Index reduction in differential-algebraic equations using dummy
  derivatives.
\newblock {\em SIAM Journal of Scientific Computing 14}, 3 (May 1993),
  677--692.

\bibitem{mishra:1993}
{\sc Mishra, B.}
\newblock {\em Algorithmic algebra}.
\newblock Springer-Verlag New York, Inc., New York, NY, USA, 1993.

\bibitem{ollivier:2007}
{\sc Ollivier, F., and Sadik, B.}
\newblock La borne de {J}acobi pour une diffi{\'e}t{\'e} d{\'e}finie par un
  syst{\`e}me quasi r{\'e}gulier.
\newblock {\em Comptes Rendus de l'Acad{\'e}mie des sciences 345}, 3 (Aug.
  2007), 139--144.

\bibitem{pantelides:1988}
{\sc Pantelides, C.~C.}
\newblock The consistent initialization of differential-algebraic systems.
\newblock {\em SIAM Journal on Scientific and Statistical Computing 9}, 2 (Mar.
  1988), 213--231.

\bibitem{petzold:1986}
{\sc Petzold, L.~R., and L{\"o}tstedt, P.}
\newblock Numerical solution of nonlinear differential equations with algebraic
  constraints {II}: Practical implications.
\newblock {\em SIAM Journal on Scientific and Statistical Computing 7}, 3 (July
  1986), 720--733.

\bibitem{poulsen:2001}
{\sc Poulsen, M.~Z.}
\newblock {\em Structural analysis of {DAE}s}.
\newblock PhD thesis, Informatics and Mathematical Modelling, Technical
  University of Denmark, 2001.

\bibitem{pritchard:2003}
{\sc Pritchard, F.~L.}
\newblock On implicit systems of differential equations.
\newblock {\em Journal of Differential Equations 194}, 2 (Nov. 2003), 328--363.

\bibitem{pritchard:2007}
{\sc Pritchard, F.~L., and Sit, W.~Y.}
\newblock On initial value problems for ordinary differential-algebraic
  equations.
\newblock {\em Radon Ser. Comp. Appl. Math. 1\/} (2007), 1--57.

\bibitem{pryce:2001}
{\sc Pryce, J.~D.}
\newblock A simple structural analysis method for daes.
\newblock {\em BIT Numerical Mathematics 41\/} (2001), 364 -- 394.

\bibitem{rabier:1994}
{\sc Rabier, P.~J., and Rheinboldt, W.~C.}
\newblock A geometric treatment of implicit differential-algebraic equations.
\newblock {\em Journal of Differential Equations 109}, 1 (Apr. 1994), 110--146.

\bibitem{reid:2001}
{\sc Reid, G.~J., Lin, P., and Wittkopf, A.~D.}
\newblock Differential elimination-completion algorithms for {DAE} and {PDAE}.
\newblock {\em Studies in Applied Mathematics 106}, 1 (Jan. 2001), 1--45.

\bibitem{ritt:1950}
{\sc Ritt, J.~F.}
\newblock {\em Differential {A}lgebra}, vol.~33 of {\em American Mathematical
  Society Colloquium Publications}.
\newblock American Mathematical Society, New York, N. Y., U.S.A., 1950.

\bibitem{ritt:1932}
{\sc Ritt, J.~P.}
\newblock {\em Differential equations from the algebraic standpoint}, vol.~14
  of {\em Amer. Math. Soc. Colloq. Publ.}
\newblock Walter De Gruyter Inc, 1932.

\bibitem{schost:2003}
{\sc Schost, {\'E}.}
\newblock Computing parametric geometric resolutions.
\newblock {\em Applicable Algebra in Engineering, Communication and Computing
  13}, 5 (Feb. 2003), 349--393.

\bibitem{seidenberg:1952}
{\sc Seidenberg, R.~A.}
\newblock Some basic theorems in differential algebra (characteristic $p$
  arbitrary).
\newblock {\em Trans. Amer. Math. Soc. 73\/} (1952), 174--190.

\bibitem{seiler:1999}
{\sc Seiler, W.~M.}
\newblock Indices and solvability of general systems of differential equations.
\newblock In {\em Computer Algebra in Scientific Computing 99\/} (1999), V.~G.
  Ghanza, E.~W. Mayr, and E.~V. Vorozhtsov, Eds., Springer, pp.~365--386.

\bibitem{shafarevich:1994}
{\sc Shafarevich, I.~R.}
\newblock {\em Basic algebraic geometry}, 2~ed., vol.~1.
\newblock Springer, 1994.

\end{thebibliography}
\end{small}
\end{document}